\documentclass{amsart}

\usepackage[english]{babel} 
\usepackage[utf8]{inputenc} 
\usepackage[T1]{fontenc}
\usepackage{lmodern} 
\usepackage{amsthm}
\usepackage{thmtools}
\usepackage[dvipsnames]{xcolor}
\usepackage{tikz-cd}
\usepackage[pagebackref]{hyperref}
\hypersetup{
    colorlinks = true,
    linkbordercolor = {red},
    linkcolor ={NavyBlue},
    anchorcolor = {pink},
    citecolor =  {Maroon},
    filecolor = {NavyBlue},
    menucolor = {NavyBlue},
    runcolor =  {NavyBlue},
    urlcolor = {NavyBlue},
}
\usepackage{mathtools}
\usepackage{todonotes}
\usepackage{amssymb}
\usepackage[shortlabels]{enumitem}
\usepackage{bussproofs}
\usepackage{fullpage} 
\usepackage{scalerel}
\usepackage{stackengine}
\usepackage{mleftright}
\usepackage[capitalise]{cleveref} 

\declaretheorem[name=Theorem,refname={Theorem,Theorems},Refname={Theorem,Theorems},numberwithin=section,]{theorem}
\declaretheorem[name=Proposition,refname={Proposition,Propositions},Refname={Proposition,Propositions},sibling=theorem,]{proposition}
\declaretheorem[name=Lemma,refname={Lemma,Lemmas},Refname={Lemma,Lemmas},sibling=theorem,]{lemma}
\declaretheorem[name=Corollary,refname={Corollary,Corollaries},Refname={Corollary,Corollaries},sibling=theorem,]{corollary}

\declaretheoremstyle[spaceabove=2pt, spacebelow=2pt,
bodyfont=\normalfont\itshape,
postheadspace=1em
]{claimnonum}
\declaretheorem[name=Claim,numbered=no,style=claimnonum]{claim*}

\declaretheorem[name=Definition,refname={Definition,Definitions},Refname={Definition,Definitions},sibling=theorem,style=definition,]{definition}
\declaretheorem[name=Definition,numbered=no,style=definition,]{definition*}
\declaretheorem[name=Notation,refname={Notation,Notations},Refname={Notation,Notations},sibling=theorem,style=definition,]{notation}
\declaretheorem[name=Example,refname={Example,Examples},Refname={Example,Examples},sibling=theorem,style=definition,]{example}
\declaretheorem[name=Examples,refname={Examples,Examples},Refname={Examples,Examples},sibling=theorem,style=definition,]{examples}
\declaretheorem[name=Remark,refname={Remark,Remarks},Refname={Remark,Remarks},sibling=theorem,style=definition,]{remark}



\renewcommand{\O}{\mathcal{O}}
\newcommand{\K}{\mathcal{K}}

\newcommand{\op}{\mathrm{op}}

\newcommand{\cat}[1]{\mathsf{#1}}

\newcommand{\ScottOp}{\mathrm{ScottOp}}
\newcommand{\ScottOpFilt}{\mathrm{ScottOpFilt}}
\newcommand{\Op}{\mathrm{Op}}
\newcommand{\KSat}{\mathrm{KSat}}

\newcommand{\ko}{\vartriangleleft}
\newcommand{\ok}{\blacktriangleleft}

\newcommand{\black}{\color{black}}

\newcommand{\F}{\mathcal{F}}
\newcommand{\I}{\mathcal{I}}

\newcommand{\B}{\mathfrak{B}}

\newcommand{\restr}[2]{#1\restriction_{#2}}  

\newcommand{\Up}{\mathrm{Up}}

\newlength\arrowheight
\DeclareRobustCommand{\doubleuparrow}{%
	\mathrel{\ThisStyle{%
			\settoheight{\arrowheight}{$\SavedStyle\uparrow$}%
			\scalerel*{\stackengine{.3\arrowheight}{$\SavedStyle\uparrow$}%
				{$\SavedStyle\uparrow$}{O}{c}{F}{F}{L}}{\uparrow}%
	}}%
}

\DeclareRobustCommand{\doubledownarrow}{%
	\mathrel{\ThisStyle{%
			\settoheight{\arrowheight}{$\SavedStyle\downarrow$}%
			\scalerel*{\stackengine{.3\arrowheight}{$\SavedStyle\downarrow$}%
				{$\SavedStyle\downarrow$}{O}{c}{F}{F}{L}}{\downarrow}%
	}}%
}
\renewcommand{\d}{{\downarrow}}
\renewcommand{\u}{{\uparrow}}
\newcommand{\uo}{{\uparrow_\O}}
\newcommand{\dk}{{\downarrow_\K}}
\newcommand{\uuk}{{\doubleuparrow_\K}}
\newcommand{\ddo}{{\doubledownarrow_\O}}

\newcommand{\R}{\mathbb{R}}
\newcommand{\Z}{\mathbb{Z}}
\newcommand{\Q}{\mathbb{Q}}
\newcommand{\N}{\mathbb{N}}

\newcommand{\rel}{\looparrowright}
\renewcommand{\r}[1]{\mathrel{#1}}

\newcommand{\KO}{\cat{KO}}
\newcommand{\Bidcpo}{\cat{BiDcpo}}
\newcommand{\DBidcpo}{\cat{DBiDcpo}}
\newcommand{\EBidcpo}{\cat{EBiDcpo}}
\newcommand{\DEBidcpo}{\cat{DEBiDcpo}}

\newcommand{\CP}{\mathrm{CP}}

\DeclareRobustCommand{\neswarrow}{%
  \mathrel{\text{\ooalign{$\swarrow$\crcr$\nearrow$}}}%
}

\newcommand{\ub}{\mathrm{ub}}
\newcommand{\lb}{\mathrm{lb}}

\newcommand{\im}{\mathrm{im}}

\makeatletter 
\def\l@subsection{\@tocline{2}{0pt}{2pc}{6pc}{}} 
\makeatother

\DeclareRobustCommand{\SkipTocEntry}[5]{} 

\title{On the symmetry behind duality}

\keywords{self-duality, de Groot self-duality, Lawson self-duality, sober spaces, spatial frames, stably compact spaces, continuous domains, bitopological spaces, d-frames, formal concept analysis, polarities, double base lattices, dcpos, compactness, Stone duality, categorical duality, locally compact frames, pointfree topology, frames, locales, ko-spaces, bi-dcpos}
\subjclass[2020]{Primary: 54E55. Secondary: 06D50, 06D22, 06E15, 06F30, 06B35, 06D05}

\author{Marco Abbadini}
\address[Marco Abbadini]{School of Computer Science,
University of Birmingham,
B15 2TT Birmingham (UK)}
\email{m.abbadini@bham.ac.uk}

\author{Achim Jung}
\address[Achim Jung]{School of Computer Science,
University of Birmingham,
B15 2TT Birmingham (UK)}
\email{A.Jung@bham.ac.uk}

\usepackage{microtype} 
\begin{document}

\begin{abstract}
	Dualities such as Stone duality and the duality between sober spaces and spatial frames hinge on an interaction between open sets and compact saturated sets.
In several important classes of spaces—Stone spaces, spectral spaces, and stably compact spaces—this interaction forms a perfect symmetry, reflected dually as order self-duality.
But the class of sober spaces, despite being central to Stone-like dualities, exhibits only a partial symmetry between openness and compactness.

This raises a central question:
can we enlarge the setting enough to recover a perfect symmetry, while still retaining sober spaces and preserving the conditions that make the sober–spatial-frame duality work?

We answer this question affirmatively.
We introduce \emph{ko-spaces}, whose families of open and compact saturated sets satisfy the compatibility needed for duality, and \emph{bi-dcpos}, a pointfree companion generalizing both spatial frames and continuous domains.
We prove that the categories of ko-spaces and distributive bi-dcpos are equivalent (and dually equivalent, too), and that each category carries a symmetry in the form of a self-duality.
On spaces, this extends de Groot duality; on domains, it extends Lawson duality.

Classical results fall out as special cases: the sober–spatial-frame duality reappears inside our symmetric framework, and continuous domains acquire a presentation akin to that of d-frames.
Our work suggests that an appropriate home for Stone-like duality is a fully symmetric two-sorted world in which openness and compactness play on equal footing.
\end{abstract}

\maketitle

\tableofcontents


\section{Introduction}\label{s:introduction}

Categorical dualities provide a bridge between algebra (the syntax) and topology (the semantics), offering a double perspective on the same problem.
For example, Stone's dualities for Boolean algebras and bounded distributive lattices \cite{Stone1936,Stone1938} allow for a representation of these algebras as algebras of certain subsets of certain topological spaces. Similarly, the duality between sober spaces and spatial frames allows one to replace a topological space by the poset of its open sets.

At the heart of these dualities lies the interaction between open sets and compact sets.
For example, in Stone's duality for bounded distributive lattices, a bounded distributive lattice is represented as the lattice of \emph{compact open} subsets of a certain space.
The interaction between open sets and compact sets is what typically allows the duality to take place; for example, a famous sufficient condition for a frame to be spatial is local compactness; and the frames that are spatial are precisely those that have enough Scott-open filters, which, roughly speaking, means those that have enough compact subspaces.

As emphasized by Jung and S\"underhauf \cite{JungSuenderhauf1996}, the spaces appearing in many dualities exhibit an intriguing form of symmetry between openness and compactness.
To illustrate this, let us recall that a poset $P$ is \emph{directed} if it is nonempty and every pair of its elements has a common upper bound. The following fact holds in every topological space:
 \begin{equation*}
  \tag{K}\label{eq:K}
  \parbox{\dimexpr\linewidth-4em}{%
    \strut
    For every compact set $K$, for every directed set $\I$ of open sets with $K \subseteq \bigcup \I$ there is $U \in \I$ such that $K \subseteq U$;%
    \strut
  }
\end{equation*}
In fact, the condition ``for every directed set $\I$ of open sets with $K \subseteq \bigcup \I$ there is $U \in \I$ such that $K \subseteq U$'' is equivalent to $K$ being compact \cite[Prop.~4.4.7]{GoubaultLarrecq2013}.
Let us now recall that a poset $P$ is \emph{codirected} if it is nonempty and every pair of elements has a common lower bound.
A statement symmetrical to \eqref{eq:K} holds in every Hausdorff space \cite[Exerc.~4.4.18]{GoubaultLarrecq2013}:
\begin{equation*}
  \tag{O}\label{eq:O}
  \parbox{\dimexpr\linewidth-4em}{%
    \strut
    For every open set $U$, for every codirected set $\F$ of compact sets with $\bigcap \F \subseteq U$ there is $K \in \F$ such that $K \subseteq U$.%
    \strut
  }
\end{equation*}
Note that the statements \eqref{eq:K} and \eqref{eq:O} are obtained one from the other by swapping the words ``compact'' with ``open'' and reversing all order-theoretic notions ($\subseteq$ with $\supseteq$, directed with codirected, $\bigcap$ with $\bigcup$).

The symmetry between open and compact sets can be observed also in the notion of local compactness, one of whose equivalent formulations for Hausdorff spaces is (see e.g.\ \cite[Prop.~4.8.14]{GoubaultLarrecq2013} or \cite[Lem.~7.3]{Erne2007}):
\begin{quote}
	For every compact set $K$ and open set $U$ such that $K \subseteq U$, there are an open set $U'$ and a compact set $K'$ such that $K \subseteq U' \subseteq K' \subseteq U$.
\end{quote}
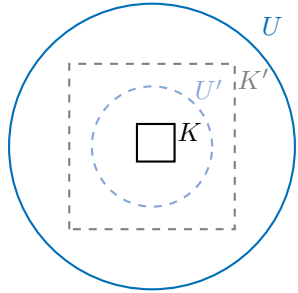
\begin{center}
	\begin{tikzpicture}
	
	  \draw[thick, NavyBlue] (0,0) circle (1.9);
	  \node[NavyBlue] at (1.6,1.6) {$U$};
	
	  \draw[thick, gray, dashed] (-1.1,-1.1) rectangle (1.1,1.1);
	  \node[gray] at (1.35,0.9) {$K'$};
	
	  \draw[thick, dashed, NavyBlue!50, dashed] (0,0) circle (0.8);
	  \node[NavyBlue!50] at (0.76,0.76) {$U'$};
	
	  \draw[thick] (-0.2,-0.2) rectangle (0.3,0.3);
	  \node at (0.5,0.2) {$K$};
	
	\end{tikzpicture}
\end{center}
This property is self-symmetric: swapping open sets with compact sets and inclusions with reverse inclusions returns the same condition.

For certain classes of spaces, the symmetry between open sets and compact sets is perfect. 
In these cases, one can obtain two results at the price of one: from a statement proved for all spaces in one such class, one deduces also the symmetric statement, obtained by swapping the words ``open'' with ``compact'' and reversing all order-theoretic notions.
A simple example is the class of compact Hausdorff spaces: given a space $(X, \O)$ in this class, the space with the same underlying set $X$ and whose open sets are the complements of the compact subsets of $(X, \O)$ is still in the same class of spaces (in this special case of compact Hausdorff spaces, it is still $(X, \O)$), and doing this process twice gives back $(X, \O)$.
For example, to prove that the following two facts hold in compact Hausdorff spaces, it is sufficient to prove just one of them, since they are perfectly symmetrical.
\begin{enumerate}[label = (W\arabic*), ref= W\arabic*]

	\item \label{i:W1}
	Let $U_1, U_2$ be open sets and $K$ a compact set with $K \subseteq U_1 \cup U_2$.
	There are compact sets $L_1, L_2$ such that $L_1 \subseteq U_1$, $L_2 \subseteq U_2$ and $K \subseteq L_1 \cup L_2$.

	\item \label{i:W2}
	Let $K_1, K_2$ be compact sets and $U$ an open set with $K_1 \cap K_2 \subseteq U$.
	There are open sets $V_1, V_2$ such that $K_1 \subseteq V_1$, $K_2 \subseteq V_2$ and $V_1 \cap V_2 \subseteq U$.
\end{enumerate}

\noindent\begin{minipage}{0.5\textwidth}
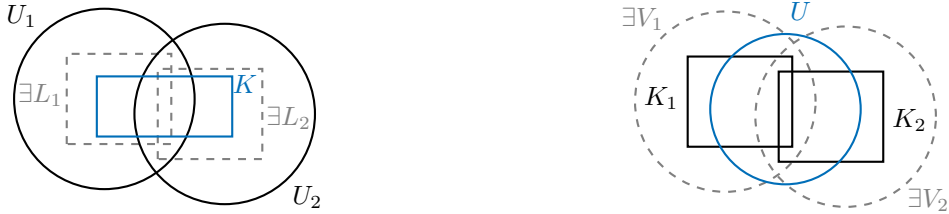
	\begin{center}
	
		\begin{tikzpicture}
			
		  \draw[thick, black] (-0.8,0.1) circle (1.2);
		  \node[black] at (-1.9,1.2) {$U_1$};
		
		  \draw[thick, gray, dashed] (-1.3,-0.5) rectangle (0.09,0.7);
		  \node[gray] at (-1.65,0.15) {$\exists L_1$};

		  \draw[thick, black] (0.8,-0.1) circle (1.2);
		  \node[black] at (1.9,-1.2) {$U_2$};
		
		  \draw[thick, gray, dashed] (-0.09,-0.7) rectangle (1.3,0.5);
		  \node[gray] at (1.65,-0.15) {$\exists L_2$};
		  
		  \draw[thick, NavyBlue] (-0.9,-0.4) rectangle (0.9,0.4);
		  \node[NavyBlue] at (1.07,0.3) {$K$};

		\end{tikzpicture}
		
	\end{center}
\end{minipage}%
\begin{minipage}{0.5\textwidth}
	\begin{center}
		\begin{tikzpicture}
			
		  \draw[thick, dashed, gray] (-0.8,0.1) circle (1.2);
		  \node[gray] at (-1.9,1.2) {$\exists V_1$};
		
		  \draw[thick, black] (-1.3,-0.5) rectangle (0.09,0.7);
		  \node[black] at (-1.65,0.15) {$K_1$};

		  \draw[thick, dashed, gray] (0.8,-0.1) circle (1.2);
		  \node[gray] at (1.9,-1.2) {$\exists V_2$};
		
		  \draw[thick, black] (-0.09,-0.7) rectangle (1.3,0.5);
		  \node[black] at (1.65,-0.15) {$K_2$};
		  
		  \draw[thick, NavyBlue] (0,0) circle (1);
		  \node[NavyBlue] at (0.2,1.3) {$U$};

		\end{tikzpicture}
	\end{center}
\end{minipage}

\noindent
 (These facts are related to the so-called Wilker's conditions, called so in \cite[p.~71]{KeimelLawson2005}, with reference to \cite{Wilker1970}.)

In fact, these two symmetric statements hold, more generally, in all locally compact Hausdorff spaces\footnote{Neither of the two statements would be true in all locally compact Hausdorff spaces if ``compact'' were replaced by ``closed'', as witnessed by any non-normal locally compact Hausdorff space, such as the deleted Tychonoff plank \cite[Part~II, 87]{SteenSeebach1978}.}.
However, here comes the subtlety: for locally compact Hausdorff spaces, we cannot deduce one statement from the other anymore!
This is because the class of locally compact Hausdorff spaces lacks a full symmetry between openness and compactness: the set of complements of the compact sets of a locally compact Hausdorff space $X$ may fail to contain the empty set (as $X$ may fail to be compact) and hence to be a topology.
And there is no topological remedy: any class of topological spaces containing all locally compact Hausdorff spaces lacks a full symmetry between openness and compactness.
To get a symmetric setting that covers these classes we shall consider a generalization of topological spaces where at least $\varnothing$ is allowed to be non-open.

We seek a delicate equilibrium.
While on the one hand we need to consider a generalization of topological spaces, on the other hand we cannot generalize too much as otherwise we may be losing some properties of the class of spaces of interest.
(For example, one may be losing Wilker's conditions.)

In this paper, the property we seek to maintain is a Stone-like categorical duality.

As some Stone-like dualities involve non-Hausdorff spaces, it is worth mentioning that, outside the $T_1$ context, ``compact'' should be replaced by ``compact saturated''.
A \emph{saturated} set is an intersection of open sets; equivalently, an upset in the specialization order.
For example, in $\R$ with the upper topology, $[0, 1)$ is compact but not saturated, while its saturation $[0, \infty)$ is both compact and saturated.
In a $T_1$ space every subset is saturated and so the notions of compact and of compact saturated collapse.

If in the statement \eqref{eq:O} above we replace ``compact'' by ``compact saturated'', we get:
\begin{equation*}
  \tag{O'}\label{eq:O'}
  \parbox{\dimexpr\linewidth-4em}{%
    \strut
    For every open set $U$, for every codirected set $\F$ of \emph{compact saturated} sets with $\bigcap \F \subseteq U$, there is $K \in \F$ such that $K \subseteq U$.%
    \strut
  }
\end{equation*}
A topological space satisfying this property is called \emph{well-filtered}.
All sober spaces (of which Hausdorff spaces are particular examples) are well-filtered, and hence sober spaces enjoy a partial symmetry between open and compact saturated sets.
Interestingly, sober spaces are precisely the spaces appearing in the famous duality with spatial frames.

The symmetry between open sets and compact saturated sets is perfect in some classes of spaces enjoying well-behaved dualities: Stone spaces, compact Hausdorff spaces, spectral spaces (which are the spaces appearing in Stone's duality for bounded distributive lattices), and more generally stably compact spaces, which are to spectral spaces what compact Hausdorff spaces are to Stone spaces.
For every stably compact space $X$, there is a stably compact space $X^\partial$---called the \emph{de Groot dual} of $X$, with reference to \cite{DeGroot1967}---with the same underlying set and whose open sets are the complements of the compact saturated sets of $X$, and vice versa whose compact saturated sets are the complements of the open sets of $X$. 
For example, the de Groot dual of $[0,1]$ with the upper topology is $[0,1]$ with the lower topology.
We warn the reader that, in this instance, ``duality'' means symmetry rather than categorical duality. We refer to \cite[Sec.~9.1.2]{GoubaultLarrecq2013} for more details on de Groot duality.
As a consequence of de Groot duality, for every property involving open sets and compact saturated sets and holding in all spectral spaces, the symmetric property holds in all spectral spaces, as well.

The symmetry in spectral spaces corresponds via Stone duality to the {self-duality} of bounded distributive lattices: if $(A, \leq)$ is a bounded distributive lattice then so is its {\emph{order dual}} $(A, \geq)$, and if a sentence holds in all bounded distributive lattices so does its order-dual, obtained by reversing the order, swapping joins with meets, etc.
Via duality, this symmetry manifests itself also in the pointfree and logical settings, as summarized in the following table (see \cite{GierzKeimel1981,Smyth1986,Smyth1992,JungSuenderhauf1996} for (strong) proximity lattices and \cite{JungKegelmannEtAl1999,Kegelmann2002} for the Multi Lingual Sequent Calculus (MLS)).

\begin{table}[ht]
	\centering
	\begin{tabular}{ |c | c | c | c | }
		\hline
		\textbf{Point-based setting} & \textbf{Point-free setting} & \textbf{Finitary setting} & \textbf{Logical setting} \\
		\hline
		\hline
       Spectral spaces	& Coherent frames &	 Bound.\ distr.\ lattices	& Negation-free sequent calc.\\
       \hline
		Stably compact spaces	& Stably compact frames &	(Strong) proximity latt.& MLS \\ 
		\hline
		\hline
		\emph{De Groot self-duality} & \emph{Lawson self-duality}	& \emph{Order self-duality} & \emph{Symmetry of the calculus}\\
		\hline
	\end{tabular}
\end{table}

As recognized by Jung and Moshier \cite{JungMoshier2006}, the symmetries in Stone dualities are a fundamental part of the nature of the structures involved. Consequently, one has a way of guessing a ``right'' construction or proof by seeking one where the symmetry is preserved.

To summarize, the interaction between open sets and compact saturated sets is central in dualities. In some of these dualities, the class of topological spaces at play exhibits a perfect symmetry between these two families of subsets, but in other ones the symmetry is only partial.
In particular, the duality between sober spaces and spatial frames lacks a formulation with full symmetry.

This paper aims to illustrate that this partial symmetry is not accidental, but, rather, part of a formal, complete {symmetry} in a larger framework that is suitable for duality theory.
Again, we seek a delicate equilibrium: on the one hand, if we don't generalize enough, then we might not get a symmetric setting (or we might not cover all sober spaces); on the other hand, if we generalize too much, we might be unable to obtain a duality that specializes to the one between sober spaces and spatial frames.
What is interesting is that the desired equilibrium exists.

To reach our aim, we introduce two new classes of structures: ko-spaces (\cref{d:ko-space}) and bi-dcpos (\cref{d:bi-dcpo}).

\begin{definition*}
	A \emph{ko-space} is a triple $(X, \K, \O)$ with $X$ a poset, $\K$ a set of upsets of $X$ closed under codirected intersections and containing all principal upsets and $\O$ a set of upsets of $X$ closed under directed unions and containing all complements of principal downsets such that for all codirected $\mathcal{F} \subseteq \K$ and directed $\mathcal{I} \subseteq \O$ with $\bigcap \mathcal{F} \subseteq \bigcup \mathcal{I}$ there are $K \in \mathcal{F}$ and $U \in \mathcal{I}$ with $K \subseteq U$.
\end{definition*}

A well-filtered $T_0$ space $(X, \O)$ can be identified with the ko-space $(X, \K, \O)$ where $X$ is equipped with the specialization order and $\K$ is the set of compact saturated subsets of $X$.

We note that, in ko-spaces, ``compactness'' and ``openness'' are independent primitive notions.
This is motivated by the fact that, even within the class of Hausdorff spaces, the topology is \emph{not} determined by the family of compact saturated sets, since two distinct Hausdorff topologies on the same set may share the same compact saturated sets\footnote{For example, in any discrete space the compact sets are precisely the finite sets, but the same happens in the Arens-Fort space \cite[Part II, 26]{SteenSeebach1978} or in any Fortissimo space \cite[Part II, 25]{SteenSeebach1978}, which are non-discrete Hausdorff spaces.}. 
Symmetrically, we cannot require compact saturated sets to be determined by the topology, either.

To include some non-compact spaces we do not require ``compact subsets'' to be closed under arbitrary intersections, and, symmetrically, we do not require open sets to be closed under arbitrary unions: \emph{directed} ones are enough.\footnote{This is in line with the independent observation by various authors \cite{Banaschewski1988,HofmannLawson1984,Johnstone1985,JohnstoneVickers1991,Erne2007} that frames sometimes are not general enough for a perfect categorical or order-theoretical approach to pointfree topology, in particular with regard to universal constructions, duality theories and applications in computer science, where frequently only directed suprema are available.}
While (possibly unbounded) distributive lattices lie in the scope of dualities (see \cite[Sec.~10]{Priestley1972} and \cite[1.2.4, p.~10]{ClarkDavey1998}), we cannot represent bottomless distributive lattices as lattices of \emph{compact open} subsets of a topological space (as in the bounded setting), since the empty subset is always compact and open!
For one such representation, we need to generalize topological spaces so that at least $\varnothing$ may be non-open (or non-compact).
Here we simply note that ko-spaces have the required flexibility.
The idea is that, outside the bounded context, we may refine Stone's maxim ``one must always topologize'' to ``one must always close under directed unions''; these two processes coincide for bounded distributive lattices.

The category of ko-spaces admits a categorical self-duality whose restriction to the class of stably compact spaces is de Groot duality.

We arrive now at the second class of structures, which generalize spatial frames and continuous domains:
\begin{definition*}
  	A \emph{bi-dcpo} is a triple $(\K, \O, \ko)$ with $\ko$ a relation between a poset $\K$ with all codirected meets and a poset $\O$ with all directed joins such that (i)  for all codirected $\F \subseteq \K$ and directed $\I \subseteq \O$ with $\bigwedge \F \ko \bigvee \I$ there are $k \in \F$ and $u \in \I$ with $k \ko u$, (ii) for all $k, k' \in \K$ and $u, u' \in \O$,  $k' \leq k \ko u \leq u'$ implies $k' \ko u'$, and (iii) for all $u,v \in \O$ with $u \nleq v$ there is $k \in \K$ with $k \ko u$ and $k \not\ko v$, and (iv) for all $k, l \in \K$ with $k \nleq l$ there is $u \in \O$ with $l \ko u$ and $k \not\ko u$.
\end{definition*}
For any ko-space $(X, \K, \O)$, one obtains a bi-dcpo $(\K, \O, \subseteq)$ by equipping $\K$ and $\O$ with the inclusion order and letting ${\subseteq} \colon \K \rel \O$ be the inclusion relation.
Moreover, for every dcpo $D$ with enough Scott-open filters (such as any spatial frame and any continuous domain), one obtains a bi-dcpo $(\K, \O, \ko)$ by setting $\K$ as the set of Scott-open filters with the reverse inclusion order, $\O$ as $D$ and $\ko \colon \K \rel \O$ as the reverse membership relation $\ni$.

The category of bi-dcpos admits a categorical self-duality that extends Lawson duality for continuous domains.

Not all bi-dcpos are spatial, i.e., arise from a ko-space.
To characterize spatiality, we take inspiration from the equivalence between spatiality and distributivity for continuous lattices \cite[Sec.~V]{GierzHofmannEtAl2003}.
To overcome the lack of a lattice structure on bi-dcpos, we apply the theory of polarities to the relation $\vartriangleleft \colon \K \rel \O$ to embed both $\K$ and $\O$ into a complete lattice: its ``concept lattice''. (This can be thought of as its ``canonical extension'', for readers more familiar with this notion.)
This complete lattice is the set of upsets (= saturated sets) of the associated space, when this exists.
The bi-dcpos with a distributive concept lattice are called \emph{distributive}; using the known connection between distributivity and Gentzen’s cut rule, we express this property in first-order terms (\cref{d:distributive}).

The main results of this paper are:
\begin{itemize}
    \item a one-to-one correspondence between ko-spaces and distributive bi-dcpos (\cref{t:bij-corr}), which extends on morphisms to
    \item a categorical equivalence (and duality) between the corresponding categories (\cref{t:MAIN}).
\end{itemize}

Our results integrate insights from the study of bitopology \cite{Kelly1963}, d-frames \cite{JungMoshier2006}, formal concept analysis \cite{GanterWille1999}, Chu spaces \cite{Barr1979}, domain theory \cite{GierzHofmannEtAl2003} and canonical extensions \cite{JonssonTarski1951,GehrkeJonsson1994,GehrkeHarding2001,GehrkeJonsson2004}, offering a unified perspective on some aspects of duality theory.
Our proofs mostly rely on the basic theory of formal concept analysis \cite{GanterWille1999} (and in particular on the version of the fundamental theorem of formal concept analysis in the form of a categorical equivalence between purified polarities and double base lattices in \cite{Erne2004a}), on the study of distributivity for concept lattices in \cite{Erne1993} and on the bijective correspondence between weakening relations between posets and adjoint pairs between the corresponding lattices of upsets \cite{GalatosJipsen2020}.

\addtocontents{toc}{\SkipTocEntry}
\subsection*{Local compactness and continuous domains.}
In \cref{s:local-compactness} we focus on local compactness (whose pointfree version is called continuity in domain theory).
The emerging message is that local compactness makes the one-sorted approach (e.g.\ frames) and the two-sorted approach (e.g., bi-dcpos) equally expressive.

This is embodied in our presentation of continuous domains as a special class of bi-dcpos (\cref{t:continuous-domains-and-bicontinuous-bi-dcpos}), which may be seen as an explanation for the existence of Lawson duality for continuous domains.

Moreover, we prove that, under bicontinuity (a mild strengthening of local compactness), the distributivity of a bi-dcpo can be checked just by looking at the distributivity of either the dcpo or the co-dcpo of which the bi-dcpo is made, when it happens to be a lattice (\cref{t:distributivity-from-O-K}).
In particular, this gives another perspective on why local compactness of a frame implies spatiality: local compactness allows one to transfer the distributivity of the underlying lattice of a frame to the distributivity (and hence spatiality) of its ``canonical extension'' (in the terminology of \cite{Jakl2020}), i.e.\ of the concept lattice of the bi-dcpo associated to it.

\addtocontents{toc}{\SkipTocEntry}

\subsection*{Literature comparison}
Our approach is similar in nature to several previous works, which gave us precious insights but differ from ours in some key aspects.

For example, our theory is close to the theory of prespaces \cite{Erne2007}, which are sets equipped with a family of subsets closed under directed unions and finite intersections.
Ko-spaces differ from prespaces in two main aspects: (i) open sets in ko-spaces are not required to be closed under finite intersections, and (ii) ko-spaces have \emph{two} non-interdefinable sets $\K$ and $\O$ of subsets.
Regarding the first difference, since the closure under finite intersections is not needed for our duality, we do not require it. Moreover, on the pointfree side, by not requiring the closure under finite meets we can cover all continuous domains via bicontinuous bi-dcpos (\cref{t:continuous-domains-and-bicontinuous-bi-dcpos}).
Regarding the second difference, while under local compactness the one-family approach of prespaces suffices to get a symmetric setting\footnote{In fact, for locally compact well-filtered spaces not only compact saturated sets are definable from open sets, but also vice versa (see \cref{c:lc-d-prespace}).}, to cover all Hausdorff spaces in a symmetric setting we need two families, as already mentioned.

This brings us to the comparison with bitopological spaces \cite{Kelly1963}, which are sets equipped with two topologies.
If a bitopological space $(X, \tau_1, \tau_2)$ has some ko-space counterpart, it is $(X, \{X \setminus U \mid U \in \tau_2\}, \tau_1)$; that is, we take the \emph{complements} of the elements of one of the two topologies. We do so because, in our topological interpretation of a ko-space $(X, \K, \O)$, $\K$ is the set of compact saturated sets and $\O$ the set of open sets, and properties with this convention seem to be more familiar and natural. For example, local compactness has a particularly natural formulation in terms of inclusions between compact saturated sets and open sets (if $K \subseteq U$ then there are $K', U'$ such that $K \subseteq U' \subseteq K' \subseteq U$).
With our convention, all the basic relations are subset inclusion relations, and so one can use intuitions about partial orders; this is particularly pleasant on the pointfree side, i.e., the side of bi-dcpos, where all the basic relations are seen to be the restriction of the order on a larger complete lattice.
Our proposal to go beyond bitopological spaces is motivated by the fact that the set of complements of compact saturated sets in a sober space is not a topology in general; accordingly, the closure conditions of the two families in a ko-space are more relaxed. This allows us to cover in a natural way all sober spaces, and more generally all well-filtered $T_0$ spaces.
On the other hand, we impose a precise compatibility condition in order not to leave the two ``topologies'' unrelated, as otherwise one would lose the typical property that makes openness and compactness interact so peculiarly, and duality possible.

Turning to the pointfree side, bi-dcpos are similar to d-frames \cite{JungMoshier2006}, the pointfree counterpart of bitopological spaces. (Biframes \cite{BanaschewskiBruemmerEtAl1983} provide a different pointfree approach, which is, however, more distant from ours.)
The differences between bi-dcpos and d-frames are similar to the ones between ko-spaces and bitopological spaces.

Bi-dcpos resemble also topological systems.
A \emph{topological system}  \cite[5.1.1]{Vickers1989} consists of a set $X$ (called the set of \emph{points}), a frame $A$ (called the set of \emph{opens}) and a subset of $X \times A$ with certain properties.
Also, a \emph{continuous map} between topological systems (\cite[Def.~5.2.1]{Vickers1989}) is similar to a Galois morphism between bi-dcpos (\cref{d:Galois-morphism}).
From a narrative point of view, the terminologies slightly differ: what is for Vickers a ``\emph{point}'' is for us a ``\emph{compact saturated set}''.
From a technical point of view, there are key differences: for example, we do not require the set of opens to be a frame, but just a dcpo; moreover, we require the compact saturated sets to be a co-dcpo in the induced order, while no similar requirement appears for topological systems.


\section{Ko-spaces}

We recall that the \emph{specialization (pre)order} of a topological space is defined by stipulating $x$ to be below $y$ if and only if every open set containing $x$ contains $y$.
A space is $T_0$ if and only if its specialization preorder is a partial order, and $T_1$ if and only if it is the equality.
We also recall that a subset of a topological space is \emph{saturated} if it is an intersection of open sets; equivalently, if it is an upset in the specialization preorder.
A topological space is $T_1$ if and only if every subset is saturated.

We also recall from the introduction that a topological space $X$ is \emph{well-filtered} if for every codirected set $\F$ of compact saturated subsets of $X$ and every open subset $U$ of $X$ with $\bigcap \F \subseteq U$ there is $K \in \F$ with $K \subseteq U$ (see e.g.\ \cite[Def.~8.3.5]{GoubaultLarrecq2013}). Many spaces of interest are well-filtered; in fact,
\[
\text{Hausdorff} \Rightarrow \text{sober} \Rightarrow \text{well-filtered}, 
\]
and the converses of both implications fail in general.

\begin{notation}[$\Op$, $\KSat$]
	Given a topological space $X$, we let $\Op(X)$ denote the set of open subsets of $X$, and $\KSat(X)$ the set of compact saturated subsets of $X$.
\end{notation}

The motivating example for the notion below of a \emph{ko-space} is, for any well-filtered $T_0$ space $X$, the triple $(X, \KSat(X), \Op(X))$, where $X$ is equipped with the specialization order.
In the name ``ko-space'', ``k'' stands for compact, and ``o'' for open.

\begin{definition}[Ko-space] \label{d:ko-space}
	A \emph{ko-space} is a triple $(X, \K, \O)$, where $X$ is a poset and $\K$ and $\O$ are sets of upsets of $X$ with the following properties. 
	(We call \emph{k-sets} and \emph{o-sets} the elements of $\K$ and $\O$, respectively.)\footnote{The terms ``o-set'' and ``k-set'' should be reminiscent of ``open set'' and ``compact set''. We avoided the latter ones to suggest that o-sets and k-sets are not interdefinable.}
	\begin{enumerate}[label = (S\arabic*), ref = S\arabic*]
	
		\item \label{i:order-ps-direct}\label{i:order-ps-codirect} ((Co)directedness) $\K$ is closed under codirected intersections and $\O$ under directed unions.

		\item \label{i:double-compactness-ps}(Double compactness)\footnote{The conjunction of \eqref{i:compactness-ps} and \eqref{i:cocompactness-ps} can be equivalently expressed as follows: for every codirected subset $\F$ of $\K$ and every directed subset $\I$ of $\O$ with $\bigcap \F \subseteq \bigcup \I$ there are $K \in \F$ and $U \in \I$ with $K \subseteq U$.}
		\begin{enumerate}
		
			\item \label{i:compactness-ps}
			(Compactness of k-sets) For every $K \in \K$ and every directed subset $\I$ of $\O$ with $K \subseteq \bigcup \I$ there is $U \in \I$ with $K \subseteq U$.
		
		\item \label{i:cocompactness-ps}(Cocompactness of o-sets, or well-filteredness) For every $U \in \O$ and every codirected subset $\F$ of $\K$ with $\bigcap \F \subseteq U$ there is $K \in \F$ with $K \subseteq U$.
		
		\end{enumerate}

		\item \label{i:order-ps-ux-is-k}\label{i:order-ps-dxc-is-o} (Principal upsets are compact, principal downsets are closed) For all $x \in X$, $\u x \in \K$ and $X \setminus \d x \in \O$.
	\end{enumerate}
\end{definition}

We refer to the partial order on $X$ as the \emph{specialization order}.
It is determined both by $\O$ and by $\K$.
Indeed, for all $x, y \in X$,
\begin{equation}\label{eq:order}
x \leq y \iff (\forall U \in \O\ (x \in U \Rightarrow y \in U)) \iff (\forall K \in \K\ (x \in K \Rightarrow y \in K)).
\end{equation}

\begin{remark}[Equivalent definition of a ko-space, without the order] \label{r:equivalent-def}
	An equivalent definition where the order is not primitive is:
	a \emph{ko-space} is a triple $(X, \K, \O)$ where $X$ is a set and $\K$ and $\O$ are sets of subsets of $X$
	(whose elements we call \emph{k-sets} and \emph{o-sets}) satisfying \eqref{i:order-ps-direct} ((co)directedness), \eqref{i:double-compactness-ps} (double compactness), and
	\begin{enumerate}[label = (A\arabic*), ref = A\arabic*]
		
		\item \label{i:ps-norm}(Double fitness)
		every o-set is a union of k-sets, and every k-set is an intersection of o-sets;
		
		\item($T_0$)
		o-sets separate points; equivalently (by \eqref{i:ps-norm}), k-sets separate points;

		\item \label{i:ps-minimal-K}\label{i:ps-maximal-O}
		(Principality)
		for all $x \in X$ there is a smallest k-set containing $x$ and a largest o-set not containing $x$.
		
	\end{enumerate}

	The order can be retrieved with the formula in \eqref{eq:order}.
\end{remark}

\begin{remark}[De Groot dual]
	The notion of a ko-space is perfectly symmetrical, in the following sense.
	Let $(X, \K, \O)$ be a ko-space.
	The \emph{de Groot dual} of $(X, \K, \O)$ is 
	\[
	X^\partial \coloneqq (X^\op, \{X \setminus U \mid U \in \O\}, \{X \setminus K \mid K \in \K\}),
	\]
	where $X^\op$ is the order dual of $X$.
	This is a ko-space.
	Note also that $(X^\partial)^\partial = X$.
	The motivation for the name ``de Groot dual'' stems from de Groot duality, which assigns to a stably compact space its \emph{de Groot dual} (named so because of \cite{DeGroot1967}), i.e.\ the stably compact space with the same underlying set and whose open sets are the complements of the compact saturated sets.
\end{remark}

\begin{lemma}[Ko-spaces from topological spaces] \label{l:well-filtered-spaces}
	Let $X$ be a topological space.
	The triple $(X, \KSat(X), \Op(X))$ (with $X$ equipped with the specialization preorder) is a ko-space if and only if $X$ is $T_0$ and well-filtered.
\end{lemma}

\begin{proof}
	($\Rightarrow$). Suppose that $(X, \KSat(X), \Op(X))$ is a ko-space.
	Then $X$ is $T_0$ because the specialization preorder is a partial order, and well-filtered by \eqref{i:cocompactness-ps} in \cref{d:ko-space}.

	($\Leftarrow$).
	Suppose that $X$ is $T_0$ and well-filtered.
	Since $X$ is $T_0$, the specialization preorder is a partial order.
	Clearly, all open sets and all compact saturated sets are upsets.
	$\Op(X)$ is clearly closed under arbitrary unions.
	Since $X$ is well-filtered, $\KSat(X)$ is closed under codirected intersections \cite[Prop.~8.3.6]{GoubaultLarrecq2013}.
	The condition \eqref{i:compactness-ps} in \cref{d:ko-space} (compactness of k-sets) holds by definition of topological compactness, and the condition \eqref{i:cocompactness-ps} in \cref{d:ko-space} (well-filteredness) holds since $X$ is well-filtered.
	Every singleton $\{x\}$ is compact, and so its saturation $\u x$ is compact and saturated.
	By definition of the specialization preorder, $X \setminus \d x$ is the union of all the open sets to which $x$ does not belong, and hence it is open.
\end{proof}

The following example provides a motivation for having both $\K$ and $\O$ as primitive notions, so that $\K$ can be the most appropriate family of compact sets in a particular context, and not necessarily the whole family of compact saturated sets; indeed, while the compact sets of relevance in spectral spaces are the compact saturated sets, the compact sets of relevance in compactly generated weakly Hausdorff spaces \cite{May1999} are the compact Hausdorff subspaces\footnote{In general, compact Hausdorff subspaces of a compactly generated weakly Hausdorff space differ from compact saturated ones, as witnessed by any compact non-Hausdorff such space, such as the one-point compactification of $\Q$.}.

\begin{examples}[Examples of ko-spaces]
	\hfill
	\begin{enumerate}
	
		\item (Compactly generated weakly Hausdorff spaces)
		A \emph{compactly generated} space (e.g.\ \cite[Def.~1.1]{Strickland2009}) is a topological space $X$ such that, for every subset $Y$ of $X$, $Y$ is closed if and only if, for every compact Hausdorff space $K$ and every continuous map $f \colon K \to X$, $f^{-1}[Y]$ is closed.
	A \emph{weakly Hausdorff} space (e.g.\ \cite[Def.~1.2]{Strickland2009}) is a topological space $X$ such that, for every compact Hausdorff space $K$ and every continuous map $f \colon K \to X$, the image $f[K]$ is closed in $X$.
	
	Let $(X, \O)$ be a compactly generated weakly Hausdorff space.
	Let $\K$ be the set of compact Hausdorff subspaces (i.e.\ those subspaces which are compact and Hausdorff in the induced topology).
	We prove that $(X, \K, \O)$---where $X$ carries the equality relation as a specialization order---is a ko-space.
	
	The set of open subsets of $X$ is obviously closed under arbitrary unions.
	The set of compact Hausdorff subspaces is closed under nonempty intersections: indeed, letting $K$ be any element in a family of compact Hausdorff subspaces, the intersection of the family is an intersection of compact subsets of the compact Hausdorff space $K$. 
	Thus, $\K$ is closed under codirected intersections.
	Every element of $\K$ is compact, and so the condition \eqref{i:compactness-ps} in \cref{d:ko-space} is satisfied.
	Let us prove that the condition \eqref{i:cocompactness-ps} in \cref{d:ko-space} (cocompactness of o-sets) is satisfied, as well.
	Let $U$ be an open subset of $X$ and $\F$ a codirected subset of $\K$ with $\bigcap \F \subseteq U$.
	Since $\F$ is nonempty, there is $K_0 \in \F$.
	Then, $	\bigcap_{K \in \F} (K \cap K_0) \subseteq U \cap K_0$.
	Since $K_0$ is compact Hausdorff and $U \cap K_0$ is open in $K_0$, there is $K \in \F$ such that $K \cap K_0 \subseteq U \cap K_0$.
	Since $\F$ is codirected, there is $L \in \F$ such that $L \subseteq K \cap K_0 \subseteq U \cap K_0 \subseteq U$.
	This proves the condition \eqref{i:cocompactness-ps} in \cref{d:ko-space} (cocompactness of o-sets).
	The singletons (i.e.\ the principal upsets) are clearly compact Hausdorff, and so they belong to $\K$.
	The complements of the singletons (i.e.\ the complements of the principal downsets) belong to $\O$ because $(X, \O)$ is weakly Hausdorff and so $T_1$.
	This proves that $(X, \K, \O)$ is a ko-space.
	
		\item \label{i:fin-cofin-etc}
		Let $X$ be a set.
		Set
		\begin{align*}
			\mathcal{P}(X)&\coloneqq \text{power set of $X$},\\
			\mathrm{Fin}(X)& \coloneqq \text{set of finite subsets of $X$},\\
			\mathrm{CoFin}(X)& \coloneqq \text{set of cofinite subsets of $X$},\\
			\mathrm{Singl}(X) & \coloneqq \text{set of singletons of $X$},\\
			\mathrm{CoSingl}(X) & \coloneqq \text{set of complements of singletons of $X$}.
		\end{align*}
		The following structures, with equality as specialization order, are ko-spaces:
		\begin{center}
			\begin{tabular}{l}
				$(X, \mathrm{Fin}(X), \mathcal{P}(X))$,\\
				$(X, \mathcal{P}(X), \mathrm{CoFin}(X))$,\\
				$(X, \mathrm{Singl}(X), \mathcal{P}(X))$,\\
				$(X,\mathcal{P}(X), \mathrm{CoSingl}(X))$,\\
				$(X, \mathrm{Singl}(X), \mathrm{CoFin}(X))$,\\
				$(X, \mathrm{Fin}(X), \mathrm{CoSingl}(X))$,\\
				$(X, \mathrm{Fin}(X), \mathrm{CoFin}(X))$,\\
				$(X, \mathrm{Singl}(X), \mathrm{CoSingl}(X))$.
			\end{tabular}
		\end{center}
		These examples show that neither the set $\K$ of k-sets nor the set $\O$ of o-sets determines the other one.
This is in contrast with the case of open sets and compact saturated sets in topology, where compact saturated sets are determined by open sets by definition.

		\item
		Let $X$ be a $T_1$ topological space.
		As in \eqref{i:fin-cofin-etc}, we let $\mathrm{Singl}(X)$ be the set of singletons of $X$ and $\mathrm{Fin}(X)$ be the set of finite subsets of $X$.
		$(X, \mathrm{Singl}(X), \Op(X))$ and $(X, \mathrm{Fin}(X), \Op(X))$ (with equality as specialization order) are ko-spaces.
		
		\item
		Set 
		\[\K \coloneqq \{[a,+\infty) \subseteq \R \mid a \in (-\infty, + \infty]\} \quad \text{and} \quad \O \coloneqq \{(a,+\infty) \subseteq \R \mid a \in [-\infty,+\infty) \}.\]
		Then $(\R, \K, \O)$, with the usual ``less-or-equal'' relation as the specialization order, is a ko-space.
		
	\end{enumerate}
\end{examples}

\begin{remark}[Finite ko-spaces]
	A finite ko-space consists of a finite poset $X$ and two sets $\K$ and $\O$ of upsets of $X$ such that, for every $x \in X$, $\u x \in \K$ and $X \setminus \d x \in \O$.
\end{remark}

We recall that a \emph{dcpo} (for \textbf{d}irected \textbf{c}omplete \textbf{p}artially \textbf{o}rdered set) is a poset that has suprema of all directed subsets; see e.g.\ \cite[Def.~O-2.1]{GierzHofmannEtAl2003}.
A \emph{Scott-closed} subset of a dcpo $D$ is a downset that is closed under directed joins.
A \emph{Scott-open} subset of a dcpo $D$ is the complement of a Scott-closed subset of $D$; i.e., an upset $U$ such that, for every directed subset $I$ of $D$, if $\bigvee I \in U$ then $I \cap U \neq \varnothing$.

\begin{notation}[$\ScottOp$]
	For a dcpo $D$, we let $\ScottOp(D)$ denote the set of Scott-open sets of $D$.
\end{notation}

\begin{example}[Ko-spaces from dcpos]\label{ex:dcpo}
	If $D$ is a dcpo, the triple
	\[
	(D, \{\u x \mid x \in D\}, \ScottOp(D))
	\]
	(with the order of $D$ as specialization order)
	and its de Groot dual 
	\[
	(D^\op, \{C \subseteq D \mid C \text{ is Scott-closed}\}, \{D \setminus \u x \mid x \in D\})
	\]
	(with the dual of the order of $D$ as specialization order) are ko-spaces.
\end{example}

\begin{proposition}[Characterization of posets underlying a ko-space] \label{p:proposition}
	A poset $X$ is the underlying poset of some ko-space if and only if, for every directed $I \subseteq X$ and every codirected $F \subseteq X$ such that each element of $I$ is below each element of $F$, there is an element of $X$ that is an upper bound for $I$ and a lower bound for $F$.
\end{proposition}

\begin{proof}
	($\Rightarrow$). Let $X$ be a ko-space, let $I \subseteq X$ be directed, let $F \subseteq X$ be codirected, and suppose that for all $x \in I$ and $y \in F$ we have $x \leq y$, which implies $\u x \nsubseteq X \setminus \d y$.
	Then, by double compactness, $\bigcap_{x \in I} \u x \nsubseteq \bigcup_{y \in F} X \setminus \d y$.
	Thus, there is an upper bound for $I$ that is a lower bound for $F$.
	
	($\Leftarrow$).
	Let $X$ be a poset with the property in the statement, let $\K$ be the closure of $\{\u x\mid x \in X\}$ under codirected intersection and $\O$ the closure of $\{X \setminus \d x \mid x \in X\}$ under directed unions.
	$\K$ and $\O$ are sets of upsets, $\K$ is closed under codirected intersections and $\O$ under directed unions, and $\K$ contains all principal upsets and $\O$ all complements of principal downsets.
	The property in the statement gives double compactness.
\end{proof}

\begin{remark}[Examples of underlying posets]
	All dcpos and all co-dcpos (= order duals of dcpos) are the underlying posets of some ko-space.
	The same is true for $\Z$, which is neither a dcpo nor a co-dcpo.
	Examples of posets that are not the underlying poset of any ko-space are $\Q$ and the chain with $\omega^\op$ on top of $\omega$.
\end{remark}


\section{(Distributive) (embedded) bi-dcpos}

We will consider the following two ``pointfree versions'' of a ko-space $(X, \K, \O)$:
\begin{enumerate}
	\item the triple $(\K, \O, \ko)$, where $\ko \colon \K \rel \O$ is the inclusion relation of a k-set into an o-set.
	\item the triple $(\Up(X), \K, \O)$, where $\Up(X)$ is the poset of upsets of $X$, of which $\K$ and $\O$ are subsets.
\end{enumerate}
The ko-space $(X, \K, \O)$ can be retrieved (up to isomorphism) from either of the two.

We will axiomatize the structures arising in the first way as the ``distributive bi-dcpos'', and the ones arising in the second way as the ``distributive embedded bi-dcpos''.
Distributive bi-dcpos are special instances of what we call bi-dcpos, which are special instances of polarities.
Distributive embedded bi-dcpos are special instances of what we call embedded bi-dcpos, which are special instances of double base lattices.

\subsection{\texorpdfstring{Purified polarities $\stackrel{1:1}{\leftrightarrow}$ double base lattices}{Purified polarities <-> double base lattices}}

 \emph{Polarities} (also known as \emph{formal contexts}) and \emph{double base lattices} are the objects of study of \emph{formal concept analysis}.\footnote{For readers familiar with Chu spaces \cite{Barr1979,Pratt1995}, polarities are precisely the Chu spaces with $\{0,1\}$ as the set of values.} 
 For a textbook reference, see \cite{GanterWille1999} (or \cite[Sec.~3]{DaveyPriestley2002}).
The idea (and the term \emph{polarity}) goes back to Birkhoff \cite[Sec.~V.7, p.~122]{Birkhoff1967}.

We use the symbol $\rel$ for relations (e.g., $R \colon X \rel Y$).

\begin{definition}[{Polarity}]
	A \emph{polarity} (also known as a \emph{formal context}\footnote{See e.g.\ \cite[Def.~18]{GanterWille1999}.}) is a triple $(\K, \O, \ko)$ where $\K$ and $\O$ are sets and $\ko \colon \K \rel \O$ is a relation from $\K$ to $\O$.
\end{definition}

We invite the reader to think of $\ko$ as a ``less than or equal to''-like relation from $\K$ to $\O$, induced by an order in a certain poset where both $\K$ and $\O$ live.
Also, the letters $\K$ and $\O$ should suggest thinking of $k \ko u$ as the inclusion of a compact saturated set $k$ into an open set $u$.
The notion of polarity is symmetrical: if $(\K, \O, \ko)$ is a polarity, then $(\O, \K, \ko^\partial)$ (with $u \ko^\partial k$ if and only if $k \ko u$) is a polarity, too.

\begin{definition}[Specialization preorder, ${\uparrow_\O}$, ${\downarrow_\K}$]
	Let $(\K, \O, \ko)$ be a polarity.
	\begin{enumerate}
	
		\item
		For $k \in \K$ we set $\uo k \coloneqq \{u \in \O \mid k \ko u\}$, and for $u \in \O$ we set $\dk u \coloneqq \{k \in \K \mid k \ko u\}$.
		
		\item
		Both $\K$ and $\O$ admit a natural preorder, called the \emph{specialization preorder} (e.g.\ \cite[p.~543]{Erne1993}):
		\begin{enumerate}
			\item for $k,l\in \K$, $k \leq l$ if and only if ${\uo k} \supseteq {\uo l}$,
			\item for $u, v \in \O$, $u \leq v$ if and only if $\dk u \subseteq \dk v$.
		\end{enumerate}
	
		\item
		For $k \in \K$ we set $\dk k \coloneqq \{l \in \K \mid l \leq k\}$, and for $u \in \O$ we set $\uo u \coloneqq \{v \in \O \mid u \leq v\}$.
	\end{enumerate}
\end{definition}

\begin{remark}[Polarity between sets vs polarity between preorders] \label{r:polarity-posets}
	In the definition of a polarity, we assume to start with a relation between two \emph{sets}, which then induces a preorder on those two sets.
	Alternatively, one could also start with a relation between two \emph{preorders}, taking the preorders as primitive; then one would need axioms to make sure that the specialization preorders agree with the given preorder: for example, one can take
	\begin{enumerate}
	
		\item
		(Weakening): $k' \leq k \ko u \leq u'$ implies $k' \ko u'$,
		
		\item
		(Extensionality) $\dk u \subseteq \dk v$ implies $u \le v$, and $\uo k \supseteq \uo l$ implies $k \le l$.
	
	\end{enumerate}
	
	To see the equivalence between the two approaches, suppose first that
	$(\K,\O,\ko)$ is a polarity (i.e., a relation between two sets). If
	$k' \le k \ko u \le u'$, then $k'\le k$ means $\uo k' \supseteq \uo k$, and
	so $k'\ko u$, and $u\le u'$ means $\dk u \subseteq \dk u'$, and so $k'\ko u'$.
	Thus weakening holds. Extensionality is immediate from the definition of the
	specialization orders:
	\[
	u\le v \iff \dk u \subseteq \dk v,
	\qquad
	k\le l \iff \uo k \supseteq \uo l .
	\]
	
	Conversely, suppose that $\K$ and $\O$ are preorders and that $\ko$ satisfies weakening and extensionality. We show
	that the given preorders are exactly the specialization orders induced by $\ko$.
	If $k\le l$, then weakening gives $\uo k\supseteq \uo l$; the converse is
	extensionality. Hence
	\[
	k\le l \iff \uo k \supseteq \uo l .
	\]
	Similarly, if $u\le v$, then weakening gives $\dk u\subseteq \dk v$; the
	converse is extensionality. Hence
	\[
	u\le v \iff \dk u \subseteq \dk v .
	\]
	Thus the specialization preorders are the given preorders.
\end{remark}

\begin{definition}[Purified polarities]
	A polarity $(\K, \O, \ko)$ is called \emph{purified} (or \emph{clarified}, by some authors) if the specialization preorders on $\K$ and $\O$ are partial orders. 
\end{definition}

That is, a polarity is purified if distinct elements of $\K$ are separated via an element of $\O$ and vice versa.

Whenever we have a purified polarity $(\K, \O, \ko)$ and treat $\K$ and $\O$ as partial orders, it is always assumed that the partial orders are the specialization orders.

\begin{notation}[$\ub_\O$, $\lb_\K$]
	Let $(\K, \O, \ko)$ be a polarity.	
	For $A \subseteq \K$ we define the ``set of upper bounds of $A$ in $\O$''
	\[
	\ub_\O(A) \coloneqq \{u \in \O \mid A \subseteq \dk u\},
	\]
	and for $B \subseteq \O$ we define the ``set of lower bounds of $B$ in $\K$''
	\[
	\lb_\K(B) \coloneqq \{k \in \K \mid B \subseteq \uo k\}.
	\]
\end{notation}

\begin{definition}[Concept lattice, basic functions]
	Let $(\K, \O, \ko)$ be a polarity.
	\begin{enumerate}
	
		\item
		A \emph{formal concept} of $(\K, \O, \ko)$ is a pair $(A, B)$ with $A \subseteq \K$ and $B \subseteq \O$ such that $B = \ub_\O(A)$ and $A = \lb_\K(B)$.
		The set $\mathfrak{B}(\K, \O, \ko)$ of formal concepts of $(\K, \O, \ko)$ is equipped with the \emph{hierarchical order}:
		\[
		(A_1, B_1) \leq (A_2, B_2) \iff A_1 \subseteq A_2 \iff B_2 \subseteq B_1.
		\]
		$\mathfrak{B}(\K, \O, \ko)$ is a complete lattice, called the \emph{concept lattice of $(\K, \O, \ko)$} \cite[Defs.~20-21, pp.~18-19]{GanterWille1999}.
		
		\item
		The maps
		\begin{align*}
			\iota_\K \colon \K & \longrightarrow \mathfrak{B}(\K, \O, \ko) & \iota_\O \colon \O & \longrightarrow \mathfrak{B}(\K, \O, \ko)\\
			k & \longmapsto (\dk k, \uo k) & u & \longmapsto (\dk u, \uo u),
		\end{align*}
		preserve and reflect the preorder, and are
		called \emph{basic functions} \cite{DeitersErne1998}.
	\end{enumerate}
\end{definition}

The basic functions are preorder-preserving and preorder-reflecting.
A polarity is purified if and only if the basic functions are injective; in this case, they are order-embeddings.

\begin{definition}[Double base lattice {\cite[p.\ 130]{Erne2004a}}] \label{d:dbl}
	A \emph{double base lattice} is a triple $(L, \K, \O)$ where $L$ is a complete lattice, $\K$ and $\O$ are subsets of $L$, and
	\begin{enumerate}
		\item[] (Density)\footnote{We use the term ``density'' as this property is reminiscent of the density property in canonical extensions, which is related to the fact that this property can be expressed as ``$\K$ is join-dense in $L$ and $\O$ is meet-dense in $L$''.} every element of $L$ is a join of elements of $\K$ and a meet of elements of $\O$.
	\end{enumerate}
\end{definition}

The notion of double base lattice is symmetrical: if $(L, \K, \O)$ is a double base lattice, then $(L^\op, \O, \K)$ is a double base lattice, too.

In this and the following section, we will speak of the \emph{isomorphism class} of a certain set equipped with some structure (or something similar).
Although we have not yet defined a category, the meaning should be clear: the isomorphism class of a structured set $X$ is the class of all structured sets $Y$ for which there is a bijection between $X$ and $Y$ that preserves and reflects the structure.
We do this to keep the treatment light and close to the intuition, without the categorical machinery to distract some readers from the core ideas; moreover, some of the morphisms we will introduce are relations, and so it is not obvious that the invertible morphisms are precisely the structure-preserving bijections (although they are, see \cref{r:iso-are-bijections}).

For the following, we refer to \cite[Thm.~3, p.~20]{GanterWille1999}, \cite[pp.~70-72]{DaveyPriestley2002} and \cite[Proposition~1.3]{Erne2004a}.
\begin{theorem}[The Fundamental Theorem of Formal Concept Analysis] \label{t:fundam-thm-FCA}
	There is a bijective correspondence between (isomorphism classes of)
	\begin{enumerate}
	
		\item 
		purified polarities and
		
		\item
		double base lattices.
		
	\end{enumerate}
		
\end{theorem}

The bijective correspondence is:
\begin{enumerate}
	\item[] (1)$\to$(2):
	one associates to a purified polarity $(\K, \O, \ko)$ its concept lattice $\mathfrak{B}(\K, \O, \ko)$ together with the images $\iota_\K[\K]$ and $\iota_\O[\O]$ of the basic functions;
	\item[] (2)$\to$(1):
	one associates to a double base lattice $(L, \K, \O)$ the polarity $(\K, \O, \leq)$, where $\leq \colon \K \rel \O$ is the restriction of the order on $L$.
\end{enumerate}
The isomorphisms witnessing the bijectivity are:
\begin{enumerate}

	\item
	given a purified polarity $(\K, \O, \ko)$, the pair of maps
	\[
		\iota_\K \colon \K \to \im(\iota_\K) \quad \quad \iota_\O \colon \O \to \im(\iota_\O)
	\]
	provide an isomorphism from the polarity $(\K, \O, \ko)$ to the polarity $(\im(\iota_\K), \im(\iota_\O), \leq)$ obtained from $(\mathfrak{B}(\K, \O, \ko), \im(\iota_\K), \im(\iota_\O))$.

	\item
	given a double base lattice $(L, \K, \O)$, the following inverse functions witness an isomorphism:
	\begin{align*}
		L & \longrightarrow \mathfrak{B}(\K, \O, \leq) & \mathfrak{B}(\K, \O, \leq) & \longrightarrow L\\
		a & \longmapsto (\dk a,\uo a), & (A,B) & \longmapsto \bigvee A = \bigwedge B;
	\end{align*}

\end{enumerate}

\subsection{Bi-dcpos}

For a good intuition for the next definition, we suggest thinking of $\K$ as the set of compact saturated subsets of some sober space (ordered by inclusion), $\O$ as the set of open subsets of the same space (ordered by inclusion), and $\ko$ as the relation of inclusion between compact saturated sets and open sets. In particular, $\ko$ is to be interpreted as a two-sorted ``less than or equal to''.

\begin{definition}[Bi-dcpo] \label{d:bi-dcpo}
	A \emph{bi-dcpo}\footnote{Not to be confused with a dcpo whose order dual is a dcpo.} is a purified polarity $(\K, \O, \ko)$ with the following properties.
	(We call \emph{k-elements} and \emph{o-elements} the elements of $\K$ and $\O$, respectively.)
	\begin{enumerate}[label = (L\arabic*), ref = L\arabic*]
	
		\item ((Co)directedness) $\K$ has all codirected meets and $\O$ all directed joins (in the specialization orders).		
		
		\item (Double compactness)
		\begin{enumerate}
		
			\item (Compactness of k-elements) For every $k \in \K$ and every directed subset $I$ of $\O$ with $k \ko \bigvee I$ there is $u \in I$ with $k \ko u$.
		
			\item (Cocompactness of o-elements) For every $u \in \O$ and every codirected subset $F$ of $\K$ with $\bigwedge F \ko u$ there is $k \in F$ with $k \ko u$.
		
		\end{enumerate}
		
	\end{enumerate}
\end{definition}

The name ``bi-dcpo'' is due to the fact that $\O$ is a dcpo and $\K$ is the order dual of a dcpo.

The notion of a bi-dcpo is symmetrical: if $(\K, \O, \ko)$ is a bi-dcpo then $(\O, \K, \ko^\partial)$ (with $u \ko^\partial k \Leftrightarrow k \ko u$) is a bi-dcpo, too.
This is reminiscent of Lawson duality for continuous domains (see \cref{r:Lawson-duality}).

\begin{remark}[Equivalent definition of a bi-dcpo, with the orders as primitive]
	If we want the orders on $\K$ and $\O$ to be primitive, then, by \cref{r:polarity-posets}, a bi-dcpo is a triple $(\K, \O, \ko)$ where $\ko$ is a relation between two posets $\K$ and $\O$ such that
	\begin{enumerate}
	
		\item
		(Weakening): $k' \leq k \ko u \leq u'$ implies $k' \ko u'$,
		
		\item
		(Extensionality) $\dk u \subseteq \dk v$ implies $u \le v$, and $\uo k \supseteq \uo l$ implies $k \le l$.
		
		\item ((Co)directedness) $\K$ has all codirected meets and $\O$ all directed joins.
		
		\item (Double compactness)
		\begin{enumerate}
		
			\item (Compactness of k-elements) For every $k \in \K$ and every directed subset $I$ of $\O$ with $k \ko \bigvee I$ there is $u \in I$ with $k \ko u$.
		
			\item (Cocompactness of o-elements) For every $u \in \O$ and every codirected subset $F$ of $\K$ with $\bigwedge F \ko u$ there is $k \in F$ with $k \ko u$.
		
		\end{enumerate}
	\end{enumerate}
	This is the condition mentioned in the introduction.
\end{remark}

\begin{example}[Bi-dcpos from ko-spaces]
	For every ko-space $(X, \K, \O)$, the triple $(\K, \O, \subseteq \colon \K \rel \O)$ is a bi-dcpo.
	The fact that $(\K, \O, \subseteq \colon \K \rel \O)$ is purified with inclusion as specialization order for both $\K$ and $\O$ follows from the facts that $\K$ and $\O$ are sets of upsets of the poset $X$ and that, for every $x \in X$, $\u x \in \K$ and $X \setminus \d x \in \O$.
\end{example}

\begin{example}[Bi-dcpos from bounded distributive lattices]
	For every bounded distributive lattice $A$, the triple
	\[
		(\text{set of filters of $A$},\ \text{set of ideals of $A$},\ \ko),
	\]
	where $F \ko I \Leftrightarrow F \cap I \neq \varnothing$, is a bi-dcpo.
	The specialization orders are the reverse inclusion order on the set of filters and the inclusion order on the set of ideals.
	This triple is isomorphic to the triple $(\K, \O, \subseteq \colon \K \rel \O)$ with $\K$ the set of compact saturated subsets of the spectral space dual to $A$ and $\O$ the set of open subsets.
\end{example}

We call a bi-dcpo \emph{spatial} if, up to isomorphism, it is of the form $(\K, \O, \subseteq)$ for some ko-space $(X, \K, \O)$.
Not every bi-dcpo is spatial: for example, $(\{a, b, c\}, \{a, b, c\}, =)$ is not spatial.
We will see that a bi-dcpo is spatial if and only if its concept lattice is distributive. In the case of $(\{a, b, c\}, \{a, b, c\}, =)$, this is the non-distributive diamond lattice $M_3$.
{
	\tikzcdset{every arrow/.append style={dash}}
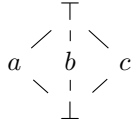
	\[
	\begin{tikzcd}[column sep = tiny, row sep = tiny] 
		&\top \dlar  \dar \drar&\\
		a \drar & b \dar & c \dlar\\
		& \bot &
	\end{tikzcd}
	\]
}

\begin{remark}[Bi-dcpos as dcpos with some Scott-open subsets] \label{r:identify}
	Given a bi-dcpo $(\K, \O, \ko)$, by extensionality we can identify an element $k \in \K$ with the subset $\uo k$ of $\O$.\footnote{This kind of identification can be done at the general level of purified polarities, and gives a $T_0$ \emph{base space} \cite[Sec.~1]{Erne2004}. The self-duality of purified polarities can be expressed nicely also in this setting \cite[Lem.~1.1]{Erne2004}.}
	In this light, a bi-dcpo is precisely a dcpo $\O$ equipped with an order-separating set $\K$ of Scott-open subsets that is closed under directed unions (as, for example, all Scott-open subsets, or all Scott-open filters if $\O$ is open-filter-determined; see below).
\end{remark}

\begin{example}[Bi-dcpos from dcpos] \label{ex:bi-dcpos}
	Any dcpo $D$ can be identified with the bi-dcpo $(\ScottOp(D), D, \ni)$.
	The specialization order on $\ScottOp(D)$ is the reverse inclusion order, and the specialization order on $D$ is the original order.
\end{example}

We recall that an element $u$ of a dcpo $D$ is said to be \emph{way below} an element $v \in D$ if, for every directed subset $E \subseteq D$ with $v \leq \bigvee E$, there is $e \in E$ with $u \leq e$ \cite[Definition~I-1.1]{GierzHofmannEtAl2003}.
We also recall that a \emph{continuous domain} is a dcpo $D$ such that, for all $u \in D$, the set of elements way below $u$ is directed, and $u$ is its supremum (see e.g.\ \cite[Definition I-1.6.(ii)]{GierzHofmannEtAl2003}, where continuous domains are called simply \emph{domains}).

A dcpo $D$ is said to be \emph{open-filter-determined} (\cite[Def.~IV-2.11]{GierzHofmannEtAl2003}) if for all $x, y \in D$ with $x \nleq y$ there is a Scott-open filter $U$ such that $x \in U$ and $y \notin U$.
The notion of an open-filter-determined dcpo generalizes both the notion of a continuous domain and the notion of a spatial frame.
(We recall that if a locale is spatial then it is open-filter-determined, and that the converse is equivalent to the Ultrafilter Principle\label{open-filtered-then-spatial}.\footnote{
Indeed, ``the ultrafilter principle is equivalent to the fact that every frame which is spatial as a preframe is spatial as a frame'' \cite[Thm.~5.2]{Erne2007}, and a frame is spatial as a preframe if and only if it has enough Scott-open filters \cite[Thm.~5.1]{Erne2007}. We also recall that a frame is open-filter-determined if and only if it has enough compact fitted sublocales; this follows from \cite[Lem.~3.4]{Johnstone1985}, which proves, without any choice principle, the bijection between Scott-open filters and compact fitted sublocales---a general localic version of the Hofmann--Mislove theorem \cite{HofmannMislove1981}.})

\begin{notation}[$\ScottOpFilt(D)$]
	 For a dcpo $D$, we let $\ScottOpFilt(D)$ denote the set of Scott-open filters of $D$.
\end{notation}

\begin{example}[Bi-dcpos from open-filter-determined dcpos] \label{ex:open-filter-determined} 
	Any open-filter-determined dcpo $D$ (e.g., a continuous domain, or a spatial frame) can be identified with the bi-dcpo $(\ScottOpFilt(D), D, \ni)$.
	This identification differs from the one in \cref{ex:bi-dcpos}. This difference will be spotted by the morphisms (Scott-continuous functions vs open filter morphisms).
\end{example}

The relation $\ko$ plays a role similar to the role of the ``$\mathrm{tot}$'' predicate in d-frames \cite{JungMoshier2006}.
The role of the ``$\mathrm{con}$'' relation is played by the relation $\ok \colon \O \rel \K$ defined from $\ko$ in \cref{n:con}.


\subsection{Embedded bi-dcpos}

Given a ko-space $(X, \K, \O)$, the triple $(\Up(X), \K, \O)$ where $\Up(X)$ is the poset of upsets of $X$ is an instance of what we call a \emph{distributive embedded bi-dcpos}. But we first give the definition without distributivity.
Let us recall that a double base lattice (\cref{d:dbl}) is a triple $(L, \K, \O)$ with $L$ a complete lattice, $\K$ a join-dense subset of $L$, and $\O$ a meet-dense subset of $L$.
The following gives another way to think of a bi-dcpo $(\K, \O, \ko)$: instead of thinking $\K$ and $\O$ as living in separate universes, they live in the same complete lattice.

\begin{definition}[Embedded bi-dcpo]
	An \emph{embedded bi-dcpo} is a double base lattice $(L, \K, \O)$ (see \cref{d:dbl}) with the following properties. 
	(We call  \emph{k-elements} and \emph{o-elements} the elements of $\K$ and $\O$, respectively.)
	\begin{enumerate}
	
		\item ((Co)directedness) $\K$ is closed in $L$ under codirected meets and $\O$ under directed joins.
		
		\item (Double compactness)\hfill
		\begin{enumerate}
			
			\item (Compactness of k-elements) For every $k \in \K$ and every directed subset $I$ of $\O$ with $k \leq \bigvee I$ there is $u \in I$ with $k \leq u$.
		
			\item (Cocompactness of o-elements) For every $u \in \O$ and every codirected subset $F$ of $\K$ with $\bigwedge F \leq u$ there is $k \in F$ with $k \leq u$.
			
		\end{enumerate}
		
	\end{enumerate}
\end{definition}

The terminology ``embedded bi-dcpo'' is justified because, as we will see, given an embedded bi-dcpo $(L, \K, \O)$, we have a polarity $(\K, \O, \leq)$, where ${\leq} \colon \K \rel \O$, that satisfies all the properties of a bi-dcpo. Thus, $(L, \K, \O)$ is indeed a bi-dcpo together with an embedding into a complete lattice $L$.

\begin{example}[Embedded bi-dcpos from ko-spaces] \label{ex:d-proto-gives-embedded-bi-dcpo}
	For any ko-space $(X, \K, \O)$, the triple $(\Up(X),\K, \O)$ is an embedded bi-dcpo.
	The fact that $(\Up(X), \K, \O)$ is a double base lattice (i.e., density) follows from the fact that every upset is a union of principal upsets and an intersection of complements of principal downsets and that the principal upsets belong to $\K$ and the complements of the principal downsets belong to $\O$.
	The remaining properties are clear.
\end{example}

\begin{example}[Canonical extensions as embedded bi-dcpos]
	It is known that every bounded lattice $A$ has a canonical extension $A \hookrightarrow L$.\footnote{The theory of canonical extensions has its origin in the work of J\'{o}nsson and Tarski on Boolean algebras with operators \cite{JonssonTarski1951,JonssonTarski1952}. The canonical extension $A \hookrightarrow L$ of a bounded distributive lattice $A$ is isomorphic to the inclusion of the family of compact open subsets of the spectral space $X$ dual to $A$ into the complete lattice of upsets of $X$. Canonical extensions allow one to have a single order-theoretic structure encompassing both an algebra and its dual, which makes it easier and more intuitive to understand how to extend additional operations and further structure. For the canonical extension of a bounded distributive lattice, we refer to \cite{GehrkeJonsson1994} and the more comprehensive \cite{GehrkeJonsson2004}; for the not-necessarily-distributive case, we refer to \cite{GehrkeHarding2001}. At the time of writing, a book by Mai Gehrke and Wesley Fussner with a particular focus on canonical extensions is expected to be released soon.}
	With the terminology of this paper, an injective bounded lattice homomorphism $A \hookrightarrow L$ with $L$ complete is a canonical extension if and only if 
	\[
	(L,	\ \text{closure of $A$ in $L$ under arbitrary meets},\ \text{closure of $A$ in $L$ under arbitrary joins})
	\]
	is an embedded bi-dcpo.%
	\footnote{Concerning canonical extensions, we also mention that the \emph{canonical extension} of a locally compact frame $\O$ in the sense of \cite{Jakl2020} is the basic function $\O \hookrightarrow \mathfrak{B}(\ScottOpFilt(\O), \O, \ni)$ mapping $\O$ into the concept lattice of the bi-dcpo $(\ScottOpFilt(\O), \O, \ni)$ (i.e., the underlying complete lattice of the corresponding embedded bi-dcpo). A further treatment of canonical extensions of frames is in \cite{JaklSuarez2025}. There, various classes of filters on a frame $\O$ are considered: completely prime filters, Scott-open filters, exact filters, and strongly exact filters. We note that, for any frame $\O$, the set of completely prime filters $\mathrm{CPFilt}(\O)$ of $\O$ is closed under directed unions; moreover, every completely prime filter is Scott-open. Therefore, if $\O$ has enough completely prime filters (i.e., if it is spatial), $(\mathrm{CPFilt}(\O), \O, \ni)$ is a bi-dcpo (see \cref{r:identify}), and the complete lattice in the embedded bi-dcpo associated to it is what is called in \cite{JaklSuarez2025} the \emph{$\mathrm{CPFilt}(\O)$-extension $\O^{\mathrm{CPFilt}(\O)}$} of $\O$. The same thing happens for the family of Scott-open filters (\cref{ex:open-filter-determined}). Instead, we do not fall within the realm of bi-dcpos for the families of exact and strongly exact filters; however, we still fall within the setting of polarities and double base lattices, as illustrated in \cite{JaklSuarez2025}.}
\end{example}

We call an embedded bi-dcpo $(L, \K, \O)$ \emph{spatial} if it is isomorphic to $(\Up(X),\K', \O')$ for some ko-space $(X, \K', \O')$.
It is not difficult to show that this is equivalent to $L$ being order-isomorphic to the set of upsets of some poset\footnote{The left-to-right implication is clear. For the right-to-left implication: (co)directedness and double compactness are clear; every principal upset $\u x$ is a k-set because, by density, it is a union of k-sets, and so one of these k-sets has to contain $x$ and hence be $\u x$; analogously, the complement of every principal downset is an o-set.}.
We say that a ko-space $(X, \K, \O)$ \emph{represents} an embedded bi-dcpo if the embedded bi-dcpo is isomorphic to $(\Up(X),\K, \O)$.

We will show that an embedded bi-dcpo $(L, \K, \O)$ is spatial if and only if the lattice $L$ is distributive (\cref{c:spatial-iff-distributive}).
An example of a non-spatial embedded bi-dcpo is $(L, \K, \O)$, where $L$ is a non-distributive finite lattice, and $\K$ and $\O$ are respectively join-dense and meet-dense subsets (as, for instance, $L$ itself).


\subsection{\texorpdfstring{Bi-dcpos $\stackrel{1:1}{\leftrightarrow}$ embedded bi-dcpos}{Bi-dcpos <-> embedded bi-dcpos}}

\begin{theorem}[Bi-dcpos $\stackrel{1:1}{\leftrightarrow}$ embedded bi-dcpos] \label{t:one-to-one}
	The bijective correspondence in \cref{t:fundam-thm-FCA} restricts to a bijective correspondence between (isomorphism classes of)
	 \begin{enumerate}
	 
	 	\item
	 	bi-dcpos and
		
		\item
		embedded bi-dcpos.
	 
	 \end{enumerate}
\end{theorem}

\begin{proof}
	The polarity associated to an embedded bi-dcpo is clearly a bi-dcpo.
	Vice versa, given a bi-dcpo $(\K, \O, \ko)$, the triple $(\mathfrak{B}(\K, \O, \ko), \mathrm{im}(\iota_\K), \mathrm{im}(\iota_\O))$ is an embedded bi-dcpo because, by \cref{l:preservation-of-meets-and-joins}, the basic functions $\iota_\K \colon \K \hookrightarrow \mathfrak{B}(\K, \O, \ko)$ and $\iota_\O \colon \O \hookrightarrow \mathfrak{B}(\K, \O, \ko)$ preserve codirected meets and directed joins, respectively.
\end{proof}

\subsection{\texorpdfstring{Distributive bi-dcpos $\stackrel{1:1}{\leftrightarrow}$ distributive embedded bi-dcpos}{Distributive bi-dcpos <-> distributive embedded bi-dcpos}}

We are interested in distributivity because it is equivalent to spatiality (\cref{c:spatial-iff-distributive} below).

\begin{definition}[Distributive embedded bi-dcpo]
	An embedded bi-dcpo $(L, \K, \O)$ is \emph{distributive} if the lattice $L$ is distributive.
\end{definition}

We will express the distributivity of the embedded bi-dcpo associated to a given bi-dcpo $(\K, \O, \ko)$ via a first-order condition.\footnote{Later on, we will show its equivalence with the existence of enough completely prime pairs, another first-order condition.}
The crucial observation is the following well-known fact.
\begin{remark} [Distributivity = cut rule] \label{r:distributivity-and-cut-rule}
	 The distributivity of a lattice $L$ is equivalent to
	\[
	\forall k, u, c \in L\ (k \leq u \lor c,\, c \land k \leq u) \Rightarrow k \leq u.
	\]
\end{remark}
This condition, expressed in a sequent calculus style, is the cut rule\footnote{The connection between distributivity and the cut rule seems to go back to Lorenzen: the connection is implicit in the study of entailment relations in \cite[Thm.~7]{Lorenzen1951} (see \cite{Lorenzen2017} for an English translation), of which a proof is only sketched after the statement of Theorem~8 in the same paper (but see \cite[Thm.~1]{CederquistCoquand2000} for a full proof). We would like to thank Thierry Coquand and Henri Lombardi for directing us to Lorenzen's work.}:
\begin{prooftree}
  \AxiomC{$k \Rightarrow u,\, c$}
  \AxiomC{$c,\, k \Rightarrow u$}
  \RightLabel{\ (Cut)}
  \BinaryInfC{$k \Rightarrow u$}
\end{prooftree}
To transpose it in the setting of polarities, we replace $c$ by two symbols $l$ and $v$:
\begin{prooftree}
  \AxiomC{$k \Rightarrow u,\, l$}
  
  \AxiomC{$l \Rightarrow v$}
  
  \AxiomC{$v,\, k \Rightarrow u$}

  \RightLabel{\ (Cut')}
  \TrinaryInfC{$k \Rightarrow u$.}
\end{prooftree}
Note that, under the axiom $a \Rightarrow a$ and the weakening rules, (Cut) and (Cut') are interderivable.
(Cut') corresponds to the following condition on a lattice $L$, which is equivalent to distributivity:
\begin{equation*} \label{eq:cut-rule-two-cuts}
	\forall k,u,l,v \in L\ (k \leq u \lor l,\, l \leq v, \, v \land k \leq u) \Rightarrow k \leq u. \tag{$\dagger$}
\end{equation*}
This leads to the following definition, which was called \emph{weak distributivity} in \cite[p.~550]{Erne1993} in the more general setting of polarities.
In our context (i.e., for bi-dcpos), we drop the word ``weak'' and call it simply ``distributivity''. This is a good name because, for a bi-dcpo, as opposed to what happens for general polarities, this property is equivalent to the distributivity of its concept lattice (\cref{t:corr-bi-dcpos-embedded-bi-dcpos}); this happens because, by Zorn's lemma, double compactness implies a property called ``double foundedness''.

\begin{definition}[Distributive bi-dcpo] \label{d:distributive}
	A bi-dcpo $(\K, \O, \ko)$ is called \emph{distributive} if, for all $k, l \in \K$ and $u,v \in \O$ such that
	\begin{enumerate}
		\item \label{i:first} $k \in \lb_\K(\uo u \cap \uo l)$,
		\item \label{i:second}
		$l \ko v$, and
		\item \label{i:third} $u \in \ub_\O(\dk v \cap \dk k)$,
	\end{enumerate}
	we have $k \ko u$.
\end{definition}

Informally, we can read \eqref{i:first} as $k \leq u \lor l$, \eqref{i:second} as $l \leq v$, and \eqref{i:third} as $v \land k \leq u$, making the link to \eqref{eq:cut-rule-two-cuts} more explicit.
Note that this condition is first-order in the language of bi-dcpos.

\begin{example}[Distributivity for open-filter-determined dcpos] \label{ex:distributive}
	By \cref{ex:open-filter-determined}, if $D$ is an open-filter-determined dcpo, then $(\ScottOpFilt(D), D, \ni)$ is a bi-dcpo.
	We now show that if $D$ is additionally a distributive lattice\footnote{Here, it is \emph{not} necessary that $D$ is a \emph{bounded} distributive lattice.} (for example, if $D$ is an open-filter-determined frame) then the bi-dcpo $(\ScottOpFilt(D), D, \ni)$ is distributive.
	Once we will have proved that distributivity is equivalent to spatiality (\cref{c:spatial-iff-distributive}), this will imply the known fact that every open-filter-determined frame is spatial (already discussed at p.~\pageref{open-filtered-then-spatial}).
	
	Let $u, v \in D$ and $K, L \in \ScottOpFilt(D)$.
	Suppose that the following conditions hold.
	\begin{enumerate}
	
		\item \label{i:d-1}
		$(\u u) \cap L\subseteq K$.
		
		\item \label{i:d-2}
		$v \in L$.
		
		\item \label{i:d-3}
		For every $H \in \ScottOpFilt(D)$ such that $K \cup \{v\} \subseteq H$ we have $u \in H$.
		
	\end{enumerate}
	We shall prove $u \in K$.
	
	By \eqref{i:d-1} and \eqref{i:d-2}, $v \vee u \in K$.
	Set $H \coloneqq \{w \in D \mid w \vee u \in K\}$.
	Clearly, $H$ is Scott-open.
	We show that it is a filter.
	$H$ is clearly an upset since $K$ is an upset.
	$H$ is nonempty because $K$ is a nonempty upset.
	For every $h_1, h_2 \in H$ we have, by definition of $H$, $h_1 \vee u \in K$ and $h_2 \vee u \in K$; thus, since $K$ is a filter, we have $(h_1 \vee u) \wedge (h_2 \vee u) \in K$, i.e., by distributivity, $(h_1 \wedge h_2) \vee u \in K$, i.e. $h_1 \wedge h_2 \in H$.
	This shows that $H$ is a filter.
	Moreover, $v \in H$ because $v \vee u \in K$.
	Furthermore, $K \subseteq H$ because $K$ is an upset.
	Therefore, by \eqref{i:d-3}, $u \in H$, i.e.\ $u \vee u \in K$, i.e.\ $u \in K$, as desired.
\end{example}

We now show the non-trivial fact that the distributivity of a bi-dcpo is equivalent to the distributivity of its concept lattice.

\begin{theorem}[Distributive bi-dcpos $\stackrel{1:1}{\leftrightarrow}$ distributive embedded bi-dcpos] \label{t:corr-bi-dcpos-embedded-bi-dcpos}
	The bijective correspondence in \cref{t:one-to-one} restricts to a bijective correspondence between (isomorphism classes of)
	\begin{enumerate}
	
		\item
		distributive bi-dcpos and
		
		\item
		distributive embedded bi-dcpos.
	
	\end{enumerate}
\end{theorem}

\begin{proof}
	We recall from \cite[p.~89]{Wille1985} (see \cite[Def.~26]{GanterWille1999} for a textbook reference) that a polarity $(\K, \O, \ko)$ is called \emph{doubly founded} if, for all $k \in \K$ and $u \in \O$ with $k \not\ko u$, (i) the set $\{v \in \O \mid k \not\ko v\}$ has a maximal element\footnote{In a preordered set, an element is \emph{maximal} if it is above every element above it, and \emph{minimal} if it is below every element below it.} that is above $u$, and (ii) the set $\{l \in \K \mid l \not\ko u\}$ has a minimal element that is below $k$.
	Using Zorn's lemma and double compactness, it is easily shown that every bi-dcpo is doubly founded.
	
	The desired statement then follows from the fact that any bi-dcpo $(\K, \O, \ko)$ is doubly founded and the fact that a doubly founded polarity is weakly distributive (in the terminology of \cite[p.~550]{Erne1993}, i.e.\ it satisfies the condition in \cref{d:distributive}) if and only if its concept lattice is distributive \cite[Thm.~7.9, (c)~$\Leftrightarrow$~(h)]{Erne1993}.
\end{proof}

	We stress that the ``only if'' in the cited result from \cite{Erne1993} in the proof of \cref{t:corr-bi-dcpos-embedded-bi-dcpos} seems to be far from trivial, and by piggybacking on it we have simplified the presentation a lot.

One may wonder whether the distributivity of a bi-dcpo $(\K, \O, \ko)$ is related to the distributivity of $\O$ or $\K$, if these are lattices.
In \cref{t:distributivity-from-O-K} below we show that, if $(\K, \O, \ko)$ is \emph{bicontinuous} and either $\K$ or $\O$ is a distributive lattice, then $(\K, \O, \ko)$ is distributive.


\section{\texorpdfstring{On objects: ko-spaces $\stackrel{1:1}{\leftrightarrow}$ distributive (embedded) bi-dcpos}{On objects: ko-spaces <-> distributive (embedded) bi-dcpos}}

\subsection{\texorpdfstring{Posets $\stackrel{1:1}{\leftrightarrow}$ Raney lattices}{Posets <-> Raney lattices}}

Completely prime pairs will play the role of ``abstract points''.

\begin{definition}[Completely prime pair {\cite[p.~3683]{DaveyPriestley1996}}]
	Let $L$ be a complete lattice.
	A pair $(a, b)$ of elements of $L$ is called \emph{completely prime} (or \emph{splitting} by some authors\footnote{This terminology goes back to the work of Whitman \cite{Whitman1943}, in which the partitions $(\u a, \d b)$ arising from such pairs $(a,b)$ are called \emph{principal splittings}.}) if $\u a$ and $\d b$ partition $L$.\footnote{The terminology ``completely prime pair'' is justified by the following facts, which hold in every complete lattice (\cite[Exercise 10.14, p.~246]{DaveyPriestley2002}): (i) the assignment $(a,b) \mapsto a$ is a bijection from the set of completely prime pairs of $L$ to the set of completely join prime elements of $L$, with inverse $a \mapsto (a, \bigvee (L \setminus \u a))$; (ii) the assignment $(a,b) \mapsto b$ is a bijection from the set of completely prime pairs of $L$ to the set of completely meet prime elements of $L$, with inverse $b \mapsto (\bigwedge (L \setminus \d b), b)$.}
\end{definition}

The following condition characterizes lattices of upsets of posets (as shown in \cref{t:char-Raney} below).

\begin{definition}[Raney lattice]
	A \emph{Raney lattice} is a complete lattice $L$ such that, for all $a, b \in L$ with $a \nleq b$, there is a completely prime pair $(k,u) \in L^2$ such that $k \in \d a$ and $u \in \u b$.
\end{definition}

In other words, a Raney lattice is a complete lattice with enough completely prime pairs.

Raney lattices are also sometimes called \emph{superalgebraic} (and presented as the complete lattices where every element is a join of completely join-prime elements) \cite[p.~381]{BanaschewskiPultr1990/91}.
The terminology ``Raney lattice'' was introduced in \cite{BezhanishviliHarding2020} in honor of Raney's work \cite{Raney1952} via an equivalent condition (a completely distributive complete
lattice that is join-generated by its completely join prime elements). We refer to \cite[Thm.~10.29]{DaveyPriestley2002} for more equivalent conditions.
The overlap of terminology is innocuous since all these conditions characterize lattices of upsets of posets (equipped with the inclusion order).

\begin{theorem}[Posets $\stackrel{1:1}{\leftrightarrow}$ Raney lattices] \label{t:char-Raney}
	There is a bijective correspondence between (isomorphism classes of) 
	\begin{enumerate}
	
		\item
		posets and 
		
		\item
		Raney lattices.
	
	\end{enumerate}
\end{theorem}

\begin{proof}
	We prove the statement with the following assignments witnessing the correspondence:
	\begin{enumerate}
	
		\item
		to a poset $X$, we assign the Raney lattice $\Up(X)$ of upsets of $X$;
		
		\item
		to a Raney lattice $L$, we assign the poset $\CP(L)$ of completely prime pairs in $L$, ordered by $(k, u) \leq (l,v) \iff k \geq l \iff u \geq v$. 
	
	\end{enumerate}
	
	Moreover, we prove that the following isomorphisms witness the bijectivity of the correspondence:
	\begin{enumerate}
	
		\item
		for every poset $X$, the functions
		\begin{align*}
			X & \longrightarrow \CP(\Up(X)) & \CP(\Up(X)) & \longrightarrow X\\
			x & \longmapsto (\u x, X \setminus \d x) & (A,B) & \longmapsto \text{unique $x \in X$ such that $x \in A \setminus B$};
		\end{align*}
	
		\item
		for every Raney lattice $L$, the functions
		\begin{align*}
			L & \longrightarrow \Up(\CP(L)) & \Up(\CP(L)) & \longrightarrow L\\
			a & \longmapsto \{(k,u) \in \CP(L) \mid k \leq a\} = \{(k,u) \in \CP(L) \mid a \nleq u\} & A & \longmapsto \bigvee_{(k,u) \in A} k.
		\end{align*}
	\end{enumerate}

	We prove that, for every poset $X$, $\Up(X)$ is indeed a Raney lattice.
	Let $A, B \in \Up(X)$ be such that $A \nsubseteq B$.
	Then there is $x \in A \setminus B$.
	The pair $(\u x, X \setminus \d x)$ is completely prime, and we have $\u x \subseteq A$ and $B \subseteq X \setminus \d x$.
	This shows that $\Up(X)$ is indeed a Raney lattice.
	
	The proof of the fact that the maps $X \to \CP(\Up(X))$ and $\CP(\Up(X)) \to X$ are inverses is straightforward.
	
	Let $L$ be a Raney lattice.
	We prove that the maps $\alpha \colon L \to \Up(\CP(L))$ and $\beta \colon \Up(\CP(L)) \to L$ are inverses.
	The function $\alpha$ is clearly order-preserving.
	It is also order-reflecting because $L$ is a Raney lattice.
	Thus, it is injective.
	It is then enough to show that $\alpha \circ \beta$ is the identity on $\Up(\CP(L))$.
	Let $A$ be an upset of $\CP(L)$.
	We prove $\alpha(\bigvee_{(k,u) \in A} k) = A$.
	The inclusion $A \subseteq \alpha(\bigvee_{(k,u) \in A} k)$ is obvious. 
	Let us prove the converse inclusion, i.e.\ $\alpha(\bigvee_{(k,u) \in A} k) \subseteq A$.
	Let $(a_0,b_0) \in \alpha(\bigvee_{(k,u) \in A} k)$.
	Then, $a_0 \leq \bigvee_{(k,u) \in A} k$.
	Since $a_0$ is completely join prime, there is $(k,u) \in A$ such that $a_0 \leq k$.
	Since $A$ is an upset, it follows that $(a_0, b_0) \in A$.
	This proves that $\alpha \circ \beta$ is the identity, and so that $\alpha$ and $\beta$ are inverses.
\end{proof}

\Cref{t:char-Raney} is the object-level part of the duality baptized in \cite{BezhanishviliHarding2020} as \emph{Raney duality}, having its origins in the work of Raney \cite{Raney1952}, where the object level of the correspondence was established (see also Balachandran \cite{Balachandran1955} and Bruns \cite{Bruns1959}). The statement of the duality can be found for example in \cite[Thm.~2.4]{BezhanishviliHarding2020}.

\subsection{\texorpdfstring{Bifounded purified polarities $\stackrel{1:1}{\leftrightarrow}$ bifounded double base lattices}{Bifounded purified polarities <-> bifounded double base lattices}}

In the next subsection we will prove that every (embedded) distributive bi-dcpo is spatial.
For this, we shall prove that the complete lattice $L$ in a distributive embedded bi-dcpo $(L, \K, \O)$ is isomorphic to the lattice of upsets of some poset.
Not every distributive complete lattice is isomorphic to the lattice of upsets of some poset.
One needs a further condition, which we call \emph{bifoundedness}; a complete lattice is isomorphic to the lattice of upsets of some poset if and only if it is distributive and bifounded (\cref{t:char-Raney-dist-bifounded}). A simple application of Zorn's lemma will show that the complete lattice $L$ in any embedded bi-dcpo $(L, \K, \O)$ is bifounded (\cref{t:bifounded}).

\begin{definition}[$\neswarrow$]\hfill
	\begin{enumerate}
	
		\item
		If $(\K, \O, \ko)$ is a polarity, $k \in \K$ and $u \in \O$, we write $k \neswarrow u$ if $u$ is a maximal element of $\{v \in \O \mid k \not\ko v\}$ and $k$ a minimal element of $\{l \in \K \mid l \not\ko u\}$ (see, e.g.\ {\cite[Def.~25]{GanterWille1999}}).
			
		\item
		For a complete lattice $L$ and $a,b \in L$, we write $a \neswarrow b$ if $b$ is a maximal element of $\{c \in L \mid a \nleq c\}$ and $a$ a minimal element of $\{c \in L \mid c \nleq b\}$.

	\end{enumerate}
\end{definition}

\begin{definition}[Bifounded]\hfill
	\begin{enumerate}
	
		\item
		A polarity $(\K, \O, \ko)$ is \emph{bifounded} if, for all $k \in \K$ and $u \in \O$ with $k \not\ko u$, there are $l \in \dk k$ and $v \in \uo u$ such that $l \neswarrow v$.
		
		\item
		A complete lattice $L$ is \emph{bifounded} if, for all $a,b \in L$ with $a \nleq b$, there are $k \in \d a$ and $u \in \u b$ such that $k \neswarrow u$.
	
	\end{enumerate}
	
\end{definition}

\begin{lemma}[$\neswarrow$ in polarity $=$ $\neswarrow$ in double base lattice]\label{l:same}
	Let $(L, \K, \O)$ be a double base lattice.
	The set of pairs $(a,b) \in L \times L$ such that $a \neswarrow b$ in the complete lattice $L$ coincides with the set of pairs $(k,u) \in \K \times \O$ such that $k \neswarrow u$ in the polarity $(\K, \O, \leq)$.
\end{lemma}

\begin{proof}
	($\subseteq$).
	Let $a,b \in L$ be such that $a \neswarrow b$ in the complete lattice $L$.
	We start by proving $a \in \K$ and $b \in \O$.
	Since $a \nleq b$, by density there is $k \in \K \cap \d a$ with $k \nleq b$.
	From $a \neswarrow b$, $k \nleq b$ and $k \leq a$ we deduce $a = k$, and so $a \in \K$.
	Dually, $b \in \O$.
	Since $a \neswarrow b$ in $L$, it is clear that we have $a \neswarrow b$ in the polarity $(\K, \O, \leq)$.
	
	($\supseteq$).
	Let $(k,u) \in \K \times \O$ be such that $k \neswarrow u$ in the polarity $(\K, \O, \leq)$, and let us prove that we have $k \neswarrow u$ in the complete lattice $L$.
	We first prove that $k$ is a minimal element of $\{c \in L \mid c \nleq u\}$ in $L$.
	Let $c \in \d k$ be such that $c \nleq u$.
	By density, there is $l \in \K \cap \d c$ with $l \nleq u$.
	Then, $l \leq c \leq k$, and so $l \leq k$.
	From $k \neswarrow u$ in $(\K, \O, \leq)$, $l \nleq u$ and $l \leq k$ we deduce $l = k$ and so $c = k$.
	Hence, $k$ is a minimal element of $\{c \in L \mid c \nleq u\}$ in $L$.
	Dually, $u$ is a maximal element of $\{c \in L \mid k \nleq c\}$.
\end{proof}

\begin{proposition} [Bifounded polarity $\Leftrightarrow$ bifounded concept lattice]\label{p:bifounded-iff-bifounded}
	A polarity is bifounded if and only if its concept lattice is bifounded.
\end{proposition}

\begin{proof}
	 Without loss of generality, we can assume the polarity to be purified and work in a double base lattice $(L, \K, \O)$.
	
	($\Rightarrow$).
	Suppose $(\K, \O, \leq)$ is bifounded.
	Let $a,b \in L$ be such that $a \nleq b$.
	By density, there are $k \in \K \cap \d a$ and $u \in \O \cap \u b$ such that $k \nleq u$.
	Then, there are $l \in \K \cap \d k$ and $v \in \O \cap \u u$ such that $l \neswarrow v$ in the polarity $(\K, \O, \leq)$, and hence $l \neswarrow v$ in the lattice $L$ by \cref{l:same}.
	
	($\Leftarrow$).
	Suppose $L$ is bifounded. Let $k \in \K$ and $u \in \O$ be such that $k \nleq u$.
	There are $l,v \in L$ such that $l \leq k$, $v \geq u$ and $l \neswarrow v$ in the complete lattice $L$.
	By \cref{l:same}, $l \in \K$, $v \in \O$ and $l \neswarrow v$ in $(\K, \O, \leq)$.
\end{proof}

\begin{remark}[Bifounded vs double founded]
	We compare the notion of bifoundedness with the similar, stronger notion of \emph{double foundedness}, which exists both in the context of complete lattices and in the context of polarities (for both concepts see \cite[p.~89]{Wille1985} or, e.g., \cite[Def.~26, p.~32]{GanterWille1999}). 
It is easily seen that every doubly founded polarity is bifounded, and that the converse is false (a witness being $(\N \cup \{+\infty\}, \N \cup \{+\infty\}, \leq)$).
	The first reason why we work with bifoundedness rather than double foundedness is that the theorem ``isomorphic to the lattice of upsets $\Leftrightarrow$ distributive + bifounded'' would be false if bifoundedness was replaced by double foundedness: for example, the lattice of upsets of $\omega^\op$ is not doubly founded. Moreover, bifoundedness is well-behaved from a further perspective: the concept lattice of a polarity is bifounded if and only if the polarity satisfies a simple condition, which we still call bifoundedness (\cref{p:bifounded-iff-bifounded}). In contrast, it is false that a polarity is doubly founded if and only if its concept lattice is doubly founded, with the standard terminology.
	(Instead, a complete lattice $L$ is doubly founded if and only if \emph{every} polarity obtained from a double base lattice $(L, \K, \O)$ with $L$ as the first component is doubly founded.)
	This comparison raises the question of whether the usage of double foundedness in the proof of \cref{t:corr-bi-dcpos-embedded-bi-dcpos} could be replaced by bifoundedness: more precisely, is a bifounded polarity weakly distributive (in the sense of \cite[p.~550]{Erne1993}) if and only if its concept lattice is distributive?
\end{remark}

\subsection{\texorpdfstring{Ko-spaces $\stackrel{1:1}{\leftrightarrow}$ distributive (embedded) bi-dcpos}{Ko-spaces <-> distributive (embedded) bi-dcpos}}

The following is the key lemma, and outlines the main usage of distributivity.

\begin{lemma}[The key lemma: completely prime $\Leftrightarrow$ $\neswarrow$]\label{l:weak-strong}
	Let $L$ be a distributive complete lattice. A pair $(k,u) \in L^2$ is completely prime if and only if $k \neswarrow u$.
\end{lemma}

\begin{proof}
	$(\Rightarrow)$. This implication is immediate (and does not require distributivity).
	
	$(\Leftarrow)$. 
	Suppose $k \neswarrow u$.
	Then $k \nleq u$ and so $\u k$ and $\d u$ are disjoint.
	We prove that their union is $L$.
	Let $a \in L$.
	We should prove that $a \in \u k$ or $a \in \d u$, i.e.\ that $a \nleq u$ implies $k \leq a$.
	Suppose $a \nleq u$.
	Then $u \lneq u \lor a$.
	Thus, since $k \neswarrow u$, we have $k \leq u \lor a$.
	By distributivity of $L$ (see \cref{r:distributivity-and-cut-rule}), from $k \nleq u$ and $k \leq u \lor a$ we deduce $a \land k \nleq u$.
	From $k \neswarrow u$, $a \land k \nleq u$ and $a \land k \leq k$, we deduce $k = k \land a$, i.e.\ $k \leq a$.
\end{proof}

\begin{theorem}[Raney $\Leftrightarrow$ distributive + bifounded] \label{t:char-Raney-dist-bifounded}
	A complete lattice is a Raney lattice if and only if it is distributive and bifounded.
\end{theorem}

\begin{proof}
	Any Raney lattice is clearly distributive because it is the lattice of upsets of some poset (\cref{t:char-Raney}).
	By \cref{l:weak-strong}, a distributive complete lattice is a Raney lattice if and only if it is bifounded.
\end{proof}

\begin{theorem}[Bi-dcpo $\Rightarrow$ bifounded] \label{t:bifounded}
	Every bi-dcpo is bifounded (as a purified polarity), and the complete lattice in any embedded bi-dcpo is bifounded.
\end{theorem}

\begin{proof}
	Using double compactness, an application of Zorn's lemma shows that every bi-dcpo is bifounded.
	By the one-to-one correspondence between bi-dcpos and embedded bi-dcpos (\cref{t:one-to-one}), and by \cref{p:bifounded-iff-bifounded}, the complete lattice in any embedded bi-dcpo is bifounded.
\end{proof}

\begin{theorem} [Distributive embedded bi-dcpo $\Rightarrow$ Raney] \label{t:distr-spatial}
	If the lattice $L$ in an embedded bi-dcpo $(L, \K, \O)$ is distributive, it is a Raney lattice.
\end{theorem}

\begin{proof}
	By \cref{t:char-Raney-dist-bifounded,t:bifounded}.
\end{proof}

\begin{corollary} [Distributivity $\Leftrightarrow$ spatiality] \label{c:spatial-iff-distributive}\hfill
	\begin{enumerate}
		
		\item \label{i:epl-spatial-distr}
		An embedded bi-dcpo is spatial if and only if it is distributive.
	
		\item \label{i:pl-spatial-distr}
		A bi-dcpo is spatial if and only if it is distributive.
	
	\end{enumerate}
\end{corollary}

\begin{proof}
	\eqref{i:epl-spatial-distr}.
	($\Rightarrow$). This implication follows from the fact that joins and meets in the lattice of upsets of a poset are unions and intersections, and that every power set is distributive.
	($\Leftarrow$).
	Suppose that $(L, \K, \O)$ is a distributive embedded bi-dcpo.
	Then, by \cref{t:distr-spatial}, $L$ is a Raney lattice, and so there is a poset $X$---namely, the set of completely prime pairs of $L$---such that $L \cong \Up(X)$.
	Without loss of generality, we may suppose $L = \Up(X)$.
	For each $x \in X$, the pair $(\u x, X \setminus \d x)$ is a completely prime pair of $L = \Up(X)$.
	By \cref{l:same,l:weak-strong}, $\u x \in \K$ and $X \setminus \d x \in \O$.
	The remaining properties are easily verified (using that arbitrary joins and meets in $\Up(X)$ are unions and intersections).
	
	\eqref{i:pl-spatial-distr}. $(\Rightarrow)$.
	Let $(X, \K, \O)$ be a ko-space.
	We shall prove that the bi-dcpo $(\K, \O, \subseteq)$ is distributive.
	By \eqref{i:epl-spatial-distr}, the embedded bi-dcpo $(\Up(X), \K, \O)$ is distributive, and by \cref{t:corr-bi-dcpos-embedded-bi-dcpos} also the bi-dcpo $(\K, \O, \subseteq)$ is distributive.
	$(\Leftarrow)$. Suppose that $(\K, \O, \ko)$ is a distributive bi-dcpo.
	Then, by \cref{t:corr-bi-dcpos-embedded-bi-dcpos} and \eqref{i:epl-spatial-distr}, the associated embedded bi-dcpo is spatial. It follows that also $(\K, \O, \ko)$ is such.
\end{proof}

\Cref{c:spatial-iff-distributive} is analogous to the fact that a continuous lattice is spatial if and only if it is distributive \cite[Thm.~V-5.5]{GierzHofmannEtAl2003}.
In conclusion, in our context, distributivity, spatiality and Gentzen's cut rule coincide.

We now have all the ingredients for (i) a bijection between ko-spaces and distributive embedded bi-dcpos and (ii) a bijection between distributive embedded bi-dcpos and distributive bi-dcpos.
Before putting the ingredients together, we define \emph{completely prime pairs} in a purified polarity to make explicit the composite of the two correspondences. 
Completely prime pairs will play the role of ``abstract points''.
The sensibleness of this definition is shown in \cref{l:same-cp}.

\begin{definition}[Completely prime pairs of purified polarities]
	Let $(\K, \O, \ko)$ be a purified polarity.
	A \emph{completely prime pair} of $(\K, \O, \ko)$ is a pair $(k,u) \in \K \times \O$ such that $k$ is the minimum of $\{l \in \K \mid l \not\ko u\}$ and $u$ is the maximum of $\{v \in \O \mid k \not\ko v\}$.
\end{definition}

The following shows that the notion of a completely prime pair in a purified polarity is the right one.

\begin{lemma}[Completely prime in polarity = completely prime in lattice] \label{l:same-cp}
	Let $(L, \K, \O)$ be a double base lattice.
	The set of completely prime pairs in the lattice $L$ coincides with the set of completely prime pairs of the polarity $(\K, \O, \leq)$.
\end{lemma}

\begin{proof}
	($\subseteq$). If $(a,b)$ is a completely prime pair in $L$, then $a \in \K$ and $b \in \O$, by density.
	
	($\supseteq$). Suppose that $(k,u)$ is a completely prime pair in $(\K, \O, \leq)$, and let us prove $k = \min \{a \in L \mid a \nleq u\}$.
	Let $a \in L$ be such that $a \nleq u$.
	By density, there is $l \in \K \cap \d a$ such that $l \nleq u$.
	Then $k \leq l \leq a$, and so $k \leq a$.
	Therefore, $k = \min \{a \in L \mid a \nleq u\}$. Dually, $u = \max\{a \in L \mid k \nleq a\}$.
\end{proof}

\begin{corollary}[Completely prime $\Leftrightarrow$ $\neswarrow$]
	Let $(\K, \O, \ko)$ be a purified polarity with a distributive concept lattice.
	A pair $(k,u) \in \K \times \O$ is completely prime if and only if $k \neswarrow u$.
\end{corollary}

\begin{proof}
	By \cref{l:weak-strong,l:same,l:same-cp}.
\end{proof}

\begin{remark}[Point of a locale = completely prime pair] \label{r:point=cp-pair}
	Let $D$ be an open-filter-determined dcpo (such as a continuous domain or a spatial frame).
	The completely prime pairs of the bi-dcpo $(\ScottOpFilt(D), D, \ni)$ (see \cref{ex:open-filter-determined}) are precisely the pairs $(K, u) \in \ScottOpFilt(D) \times D$ such that $K = D \setminus \d u$; these are in bijection with the elements $u$ of $D$ such that $D \setminus \d u$ is a filter.
	
	If the dcpo $D$ has all finite meets, then $D \setminus \d u$ is a filter if and only if $u$ is meet-prime.
	If $D$ is a frame, by distributivity, an element is meet-prime if and only if it is meet-irreducible; therefore, $D \setminus \d u$ is a filter if and only if $u \neq \top$ and for every $v, w \in D$ if $v \wedge w = u$ then $v = u$ or $w = u$.
	In conclusion, for spatial frames, the usual notion of a point coincides with the notion of a completely prime pair of the bi-dcpo $(\ScottOpFilt(D), D, \ni)$ associated to it.
\end{remark}

We are now ready to state the main result of this section.

\begin{theorem}[Ko-spaces $\stackrel{1:1}{\leftrightarrow}$ distributive bi-dcpos $\stackrel{1:1}{\leftrightarrow}$ distributive embedded bi-dcpos]\label{t:bij-corr}
	There is a bijective correspondence between (isomorphism classes of)
	\begin{enumerate}
	
		\item
		ko-spaces,
		
		\item
		distributive bi-dcpos, and
		
		\item
		distributive embedded bi-dcpos.
		
	\end{enumerate}
\end{theorem}

In the following, we prove the theorem and make the bijective correspondences explicit.

\begin{description}
	
	\item[Ko-space $\rightsquigarrow$ distributive embedded bi-dcpo]
	To any ko-space $(X, \K, \O)$ we assign the distributive embedded bi-dcpo $(\Up(X),\K, \O)$.
	
	\item
	[Distributive embedded bi-dcpo $\rightsquigarrow$ ko-space]
	To any distributive embedded bi-dcpo $(L, \K, \O)$ we assign the ko-space $(X, \hat{\K}, \hat{\O})$, where $X$ is the set of completely prime pairs of $L$ ordered with the order induced by $L^\op \times L^\op$, and $\hat{\K}$ and $\hat{\O}$ are as follows.
	For each $a \in L$, we set
	\[
	\widehat{a} \coloneqq \{(k_0,u_0) \in X \mid k_0 \leq a\} = \{(k_0,u_0) \in X \mid a \not\leq u_0\}.
	\]
	Then, we set $\hat{\K} \coloneqq \{\hat{k} \mid k \in \K\}$ and $\hat{\O} \coloneqq \{\hat{u} \mid u \in \O\}$.	
\end{description}

The first assignment is well-defined by \cref{ex:d-proto-gives-embedded-bi-dcpo,c:spatial-iff-distributive}.
The second assignment is well-defined by the proof of \cref{c:spatial-iff-distributive} and the description of the assignments in \cref{t:char-Raney}.
Doing the first assignment and then the second one is the identity on ko-spaces by \cref{t:char-Raney}.
Doing the second assignment and then the first one is the identity on distributive embedded bi-dcpos by \cref{c:spatial-iff-distributive}.

\begin{description}
	
	\item[Distributive embedded bi-dcpo $\rightsquigarrow$ distributive bi-dcpo]
	To any distributive embedded bi-dcpo $(L, \K, \O)$ we assign the distributive bi-dcpo $(\K, \O, \leq)$.
	\item[Distributive bi-dcpo $\rightsquigarrow$ distributive embedded bi-dcpo]
	To any distributive bi-dcpo $(\K, \O, \ko)$ we assign the distributive embedded bi-dcpo $(\mathfrak{B}(\K, \O, \ko), \mathrm{im}(\iota_\K), \mathrm{im}(\iota_\O))$ where $\mathfrak{B}(\K, \O, \ko)$ is the concept lattice of $(\K, \O, \ko)$ and $\mathrm{im}(\iota_\K)$ and $\mathrm{im}(\iota_\O)$ are the images of the basic functions.
\end{description}
	
The fact that these two assignments are well-defined and mutually inverse is \cref{t:corr-bi-dcpos-embedded-bi-dcpos}.

\begin{description}

	\item[Ko-space $\rightsquigarrow$ distributive bi-dcpo]
	To any ko-space $(X, \K, \O)$ we assign the distributive bi-dcpo $(\K, \O, \subseteq)$.

	\item
	[Distributive bi-dcpo $\rightsquigarrow$ ko-space]
	To any distributive bi-dcpo $(\K, \O, \ko)$ we assign the ko-space $(X,\hat{\K}, \hat{\O})$ where $X$ is the set of completely prime pairs of $(\K, \O, \ko)$, equipped with the partial order
	\[
	(k_0,u_0)\leq(k_1,u_1)
	\quad\Longleftrightarrow\quad
	k_0\geq k_1
	\quad\Longleftrightarrow\quad
	u_0\geq u_1,
	\]
	and equipped with two subsets $\hat{\K}$ and $\hat{\O}$ as follows.
	$\hat{\K}$ is the set of elements of the following form for some $k \in \K$:
	\[
	\widehat{k} \coloneqq \{(k_0,u_0) \in X \mid k_0 \leq k\} = \{(k_0,u_0) \in X \mid k \not\ko u_0\}.
	\]
	$\hat{\O}$ is the set of elements of the following form for some $u \in \O$:
	\[
	\widehat{u} \coloneqq \{(k_0,u_0) \in X \mid k_0 \ko u\} = \{(k_0,u_0) \in X \mid u \not\leq u_0\}.
	\]
\end{description}

These assignments are obtained as the composites of the assignments above; \cref{l:same-cp} allows replacing completely prime pairs of the concept lattice of $(\K, \O, \ko)$ by completely prime pairs of the bi-dcpo $(\K, \O, \ko)$.

\begin{example}[Dcpos with Scott-open sets]
	Let $D$ be a dcpo.
	The bi-dcpo $(\ScottOp(D), D, \ni)$ is distributive, and the ko-space associated to it is the ko-space
	\[
	(D^\op ,\{C \subseteq D \mid C \text{ is Scott-closed}\},\{D \setminus \u u \mid u \in D\}),
	\]
	already encountered in \cref{ex:dcpo}.
	In this case, the bi-dcpo $(\ScottOp(D), D, \ni)$ has so many k-elements that its representation via a ko-space does not give any simplification in representing $D$.
\end{example}

We have started with bi-dcpos, which are very general.
But we have ended up obtaining certain completely distributive lattices, which are very special.
If this seems puzzling, the explanation is that the completely distributive lattice we obtain is the lattice of upsets of some poset, and these are always completely distributive.


\section{\texorpdfstring{On morphisms: ko-spaces $\stackrel{1:1}{\leftrightarrow}$ distributive (embedded) bi-dcpos}{On morphisms: ko-spaces <-> distributive (embedded) bi-dcpos}}

In this section, we turn the bijections of \cref{t:bij-corr} into (dual) equivalences.

\begin{notation}
	Given a relation $R \colon X \rel Y$ and subsets $A \subseteq X$ and $B \subseteq Y$, we set
	\[
		R[A] \coloneqq \{y \in Y \mid \exists x \in A : x \r{R} y\} \quad \text{ and } \quad R^{-1}[B] \coloneqq \{x \in X \mid \exists y \in B : x \r{R} y\}.
	\]
	We use the shorthands $R[x] \coloneqq R[\{x\}]$ and $R^{-1}[y] \coloneqq R^{-1}[\{y\}]$.
	Moreover, we set\footnote{We took the notation $R^{\leftarrow}$ from \cite[p.~14]{Erne2004}.
Alternative notations are $\forall R$ \cite{JungKegelmannEtAl2001} and $\Box_R$ (e.g.\ \cite{BezhanishviliGabelaiaEtAl2019}).}
	\[
	R^\leftarrow[B] \coloneqq \{ x \in X \mid R[x] \subseteq B\} = X \setminus R^{-1}[Y \setminus B]
	\]
\end{notation}

\begin{definition}[Weakening relation]
	A \emph{weakening relation}\footnote{The name ``weakening relation'' (see e.g.\ \cite{GalatosJipsen2020}) stems from the connection with the weakening rule in logic, which derives from the sequent $\Gamma \Rightarrow \Delta$ the sequents $\Gamma, A \Rightarrow \Delta$ and $\Gamma \Rightarrow \Delta, A$. Weakening relations have been used with different names in the literature (for example, they are called \emph{monotone relations} in \cite[Def.~2.1]{BilkovaKurzEtAl2011} and \emph{stable relations} in \cite[Def.~5]{StellSchmidtEtAl2016}), and are the specialization to posets of \emph{weighted relations} in the context of quantales (see e.g.\ \cite{KurzTzimoulis}). In the context of categories enriched over some symmetric closed monoidal category, weakening relations generalize to profunctors, sometimes called bimodules, distributors or relators.} from a poset $X$ to a poset $Y$ is a relation $R \colon X \rel Y$ such that, for all $x,x' \in X$ and $y,y' \in Y$, $x' \leq x \r{R} y \leq y'$ implies $x' \r{R} y'$.
\end{definition}


\subsection{Morphisms of ko-spaces}

\begin{definition}[C-relation] \label{d:c-relation}
	A \emph{c-relation} from a ko-space $(X, \K_X, \O_X)$ to a ko-space $(Y, \K_Y, \O_Y)$ is a weakening relation $R \colon X \rel Y$ such that $R[K] \in \K_Y$ for all $K \in \K_X$, and $R^\leftarrow[U] \in \O_X$ for all $U \in \O_Y$.
\end{definition}

The name ``c-relation'' should suggest ``closed relation'', as it restricts to the standard notion of closed relation between compact Hausdorff spaces; more on this is in \cref{l:c-relations=closed-relations}.

\begin{definition}[$\KO$]
	We let $\KO$ denote the category of ko-spaces and c-relations between them. 
	The composition of c-relations is the usual composition of relations.
	The identity morphism is the specialization order.
\end{definition}

To give a piece of evidence that c-relations are a meaningful notion, we show that they preserve the relevant structure; this is a version of Esakia's lemma (see e.g.\ \cite[Lem.~3.3.12]{Esakia2019}).

\begin{lemma}[A version of Esakia's lemma] \label{l:Esakia}
	Let $R \colon X \rel Y$ be a c-relation.
	\begin{enumerate}
		
		\item \label{i:fact-preserves-codirected} 
		For every codirected subset $\F$ of $\K_X$ we have $R[\bigcap \F] = \bigcap \{ R[K] \mid K \in \F \}$.
		
		\item \label{i:fact-preserves-directed}
		For every directed subset $\I$ of $\O_Y$ we have $R^\leftarrow[\bigcup \I] = \bigcup \{R^{\leftarrow}[U] \mid U \in \I\}$.
	
	\end{enumerate}
\end{lemma}

\begin{proof}
	\eqref{i:fact-preserves-codirected}.
	$(\subseteq)$. This inclusion is immediate.
	$(\supseteq)$.
	Since $R$ is a weakening relation, $R[\bigcap \F]$ is an upset, and so it is the intersection of all o-sets containing it.
	So, it is enough to show that every o-set of $Y$ containing $R[\bigcap \F]$ contains $\bigcap \{ R[K] \mid K \in \F \}$.
	Let $U$ be an o-set of $Y$ such that $R[\bigcap \F] \subseteq U$.
	Then $\bigcap \F \subseteq R^\leftarrow[U]$.
	Therefore, since $R^\leftarrow[U]$ is an o-set, by cocompactness of o-sets there is $L \in \F$ such that $L \subseteq R^\leftarrow[U]$, and hence $R[L] \subseteq U$.
	Therefore, $\bigcap \{ R[K] \mid K \in \F \} \subseteq R[L] \subseteq U$, as desired.
	
	\eqref{i:fact-preserves-directed}.
	This is the dual of \eqref{i:fact-preserves-codirected}.
\end{proof}

\begin{remark}[De Groot self-duality on ko-spaces]
	If $R \colon X \rel Y$ is a c-relation between ko-spaces, then the converse relation $R^\partial$ is a c-relation from the de Groot dual $Y^\partial$ of $Y$ to the de Groot dual $X^\partial$ of $X$.
	Therefore, the category $\KO$ is self-dual.
\end{remark}

Given compact Hausdorff spaces $X$ and $Y$, the c-relations from $X$ to $Y$ are precisely the closed subsets of $X \times Y$.
This follows from \cref{l:c-relations=closed-relations} below, which is in the more general setting of stably compact spaces.
Following \cite[p.~213]{JungKegelmannEtAl2001}, a \emph{closed relation} from a stably compact space $X$ to a stably compact space $Y$ is a closed subset of the product of $X$ and the de Groot dual $Y^\partial$ of $Y$. 
Spelling out this condition, a closed relation from $X$ to $Y$ is a relation $R \colon X \rel Y$ such that for all $x \in X$ and $y \in Y$ with $(x,y) \notin R$ there are an open neighborhood $U$ of $x$ and a compact saturated set $K \subseteq Y$ with $y \notin K$ such that $(U \times (Y \setminus K)) \cap R = \varnothing$.\footnote{Equivalently, it is a weakening relation that is a closed subset of the product of $X$ and $Y$ equipped with the patch topologies.}

\begin{lemma}[Closed relation $=$ c-relation] \label{l:c-relations=closed-relations}
	The closed relations (in the sense of \cite{JungKegelmannEtAl2001}) from a stably compact space $X$ to a stably compact space $Y$ are exactly the c-relations from the ko-space $(X, \KSat(X), \Op(X))$ to the ko-space $(Y, \KSat(Y), \Op(Y))$.
\end{lemma}

\begin{proof}
	($\subseteq$).
	Suppose that $R$ is a closed subset of $X \times Y^\partial$.
	It is clear that $R$ is a weakening relation.
	By \cite[Prop.~3.3]{JungKegelmannEtAl2001}, $R^\leftarrow[-]$ maps open sets to open sets.
	Dually, $R[-]$ maps compact saturated sets to compact saturated sets.
	Therefore, $R$ is a c-relation.
	
	($\supseteq$).
	Suppose that $R \colon X \rel Y$ is a c-relation, and let us prove that it is a closed relation.
	Let $(x,y) \in (X \times Y) \setminus R$.
	Since $R$ is a weakening relation and $(x,y) \notin R$, we have $R[\u x] \subseteq Y \setminus \d y$.
	Since $Y$ is locally compact and $R[\u x]$ is compact saturated and $Y \setminus \d y$ open, there are an open set $U \subseteq Y$ and a compact saturated subset $K \subseteq Y$ such that $R[\u x] \subseteq U \subseteq K \subseteq Y \setminus \d y$.
	From the three inequalities we get $x \in R^\leftarrow[U]$, $(R^\leftarrow[U] \times (Y \setminus K)) \cap R = \varnothing$, and $y \notin K$, respectively.
	Thus, $R$ is closed.
\end{proof}

Therefore, the category of stably compact spaces and closed relations is a full subcategory of the category $\KO$ of ko-spaces and c-relations between them.
Note that the self-duality of $\KO$ restricts to de Groot self-duality for stably compact spaces.

To give a further example of a c-relation, we recall that the \emph{hypergraph} of an order-preserving function $f \colon X \to Y$ between posets is the relation $R_f \colon X \rel Y$ defined by $x \mathrel{R_f} y$ if and only if $f(x) \leq y$.

\begin{example}[Hypergraph is c-relation]
	The hypergraph of a continuous function $f \colon X \to Y$ between well-filtered $T_0$ spaces is a c-relation from $(X, \KSat(X), \Op(X))$ to $(Y, \KSat(Y), \Op(Y))$.
\end{example}

\begin{proof}
	It is easily seen that the hypergraph $R_f$ of $f$ is a weakening relation since $f$ is order-preserving.
	For every open set $U$ of $Y$, the set $R_f^\leftarrow[U]$ is $f^{-1}[U]$, which is open.
	For every compact saturated set $K \subseteq X$, the set $R_f[K]$ is compact and saturated since it is the saturation $\u f[K]$ of the compact set $f[K]$.
\end{proof}

\begin{remark}[Isomorphism = structure-preserving bijection] \label{r:iso-are-bijections}
	By a \emph{structure-preserving bijection} between two ko-spaces $X$ and $Y$ we mean an order-isomorphism $f \colon X \to Y$ such that, for every $A \subseteq X$, (i) $A$ is a k-set if and only if $f[A]$ is a k-set, and (ii) $A$ is an o-set if and only if $f[A]$ is an o-set.
	Every structure-preserving bijection $f \colon X \to Y$ between ko-spaces gives rise to an isomorphism in $\KO$, witnessed by the relations $R \colon X \rel Y$ and $S \colon Y \rel X$ defined by $x \r{R} y$ if and only if $f(x) \leq y$ and $y \r{S} x$ if and only if $f(x) \geq y$.
	Moreover, every isomorphism in $\KO$ arises in this way; this fact relies on the well-known analogous fact holding in the category of posets and weakening relations (see, e.g., \cite[Sec.~5.2.2]{KurzMoshierEtAl2023}).
\end{remark}


\subsection{Morphisms of (embedded) bi-dcpos}

\subsubsection{On morphisms: purified polarities $\stackrel{1:1}{\leftrightarrow}$ double base lattices}

\begin{definition}[Galois morphism {\cite[p.~129]{Erne2004a}}] \label{d:Galois-morphism}
	A \emph{Galois morphism} from a purified polarity $(\K_1, \O_1, \ko_1)$ to a purified polarity $(\K_2, \O_2, \ko_2)$ is a pair of functions $(\blacklozenge \colon \K_1 \to \K_2, \Box \colon \O_2 \to \O_1)$ such that, for all $k \in \K_1$ and $u \in \O_2$, $\blacklozenge(k) \ko_2 u$ if and only if $k \ko_1 \Box(u)$. 
\end{definition}

Galois morphisms generalize Galois connections (i.e., adjunctions) between posets, and, if we see polarities as 2-valued Chu spaces, coincide with \emph{Chu transforms} \cite{Pratt1995}.
As in the case of Galois connections between posets, one component of a Galois morphism $(\blacklozenge, \Box)$ between purified polarities determines the other, and hence the whole morphism.
Thus, one may regard $\blacklozenge$ (or $\Box$) alone as a morphism between the given polarities; this is called a \emph{Galois map} in {\cite[p.~129]{Erne2004a}}.

\begin{definition}[Embedded Galois morphism] \label{l:double-base-morphism}
	An \emph{embedded Galois morphism} (called \emph{double base lattice morphism} in \cite{Erne2004a}) from a double base lattice $(L_1, \K_1, \O_1)$ to a double base lattice $(L_2, \K_2, \O_2)$ is a pair of maps $(\blacklozenge \colon L_1 \to L_2, \Box \colon L_2 \to L_1)$ such that
	\begin{enumerate}
	
		\item
		$\blacklozenge$ is left adjoint to $\Box$, i.e.: for all $x \in L_1$ and $y \in L_2$, $\blacklozenge(x) \leq y$ if and only if $x \leq \Box(y)$;
		
		\item
		$\blacklozenge$ maps elements of $\K_1$ to elements of $\K_2$ and
		$\Box$ maps elements of $\O_2$ to elements of $\O_1$.
	
	\end{enumerate}
\end{definition}

The Fundamental Theorem of Formal Concept Analysis can be strengthened to an equivalence of categories:

\begin{theorem}[Categorical equivalence between purified polarities and double base lattices {\cite{Erne2004a}}] \label{t:equiv-formal-concept-analysis}
	The category of purified polarities and Galois morphisms between them is equivalent to the category of double base lattices and embedded Galois morphisms between them.
\end{theorem}

The assignment on objects is as in \cref{t:fundam-thm-FCA}.
The bijection on homsets between two double base lattices and their associated purified polarities is as follows.
Let $(L_1, \K_1, \O_1)$ and $(L_2, \K_2, \O_2)$ be double base lattices.
\begin{description}
	\item[From double base lattices to purified polarities] (See the functor $\K_0$ in \cite[p.~130]{Erne2004a}.)
	To a morphism 
	\[
	(\blacklozenge \colon L_1 \to L_2, \Box \colon L_2 \to L_1)
	\]
	of double base lattices from $(L_1, \K_1, \O_1)$ to $(L_2, \K_2, \O_2)$ we associate the pair of restrictions
	\[
		(\restr{\blacklozenge}{\K_1,\K_2} \colon \K_1 \to \K_2, \restr{\Box}{\O_2, \O_1} \colon \O_2 \to \O_1),
	\]
	which is a morphism of purified polarities.
	
	\item [From purified polarities to double base lattices]
	(See the \emph{base concept lattice functor $\mathcal{L}_0$} in \cite[p.~131]{Erne2004a}.)
	To a morphism of polarities
	\[
	(\blacklozenge \colon \K_1 \to \K_2, \Box \colon \O_2 \to \O_1)
	\]
	from $(\K_1, \O_1, \leq)$ to $(\K_2, \O_2, \leq)$
	we associate the pair of functions
	\[
	(\bar{\blacklozenge} \colon L_1 \to L_2, \bar{\Box} \colon L_2 \to L_1)
	\]
	given by, for $a \in L_1$ and $b \in L_2$,
	\[
	\bar{\blacklozenge}(a) = \bigvee_{k \in {{\downarrow_{\K_1}} a}} \blacklozenge(k) \quad \text{ and } \quad \bar{\Box}(b) = \bigwedge_{u \in {{\uparrow_{\O_2}} b}} \Box(u),
	\]
	which is a morphism of double base lattices from $(L_1, \K_1, \O_1)$ to $(L_2, \K_2, \O_2)$.
\end{description}

\subsubsection{Morphisms of bi-dcpos}

\begin{definition}[$\Bidcpo$, $\DBidcpo$]
	We let $\Bidcpo$ denote the category of bi-dcpos and Galois morphisms (see \cref{d:Galois-morphism}) between them, and $\DBidcpo$ its full subcategory on distributive bi-dcpos.
\end{definition}

We suggest thinking of a Galois morphism $(\blacklozenge, \Box)$ as the pointfree analogue of a c-relation $R$:
\begin{enumerate}

	\item
	$\blacklozenge$ is the \emph{forward image}, which maps $K$ to $R[K]$: the forward image of a compact set is compact;
	
	\item
	$\Box$ is the so-called \emph{universal preimage}, which maps $U$ to $R^\leftarrow[U] = (R^{-1}[U^c])^c$: the universal preimage of an open set is open.
	
\end{enumerate}

Our usage of the symbols $\blacklozenge$ and $\Box$ in the definition of Galois morphisms is suggested by the literature on modal logic and on the Vietoris hyperspaces and powerlocales.\footnote{First, this connects to Kripke semantics for modal logic and in its duality-theoretic version called ``J\'onsson-Tarski duality for modal algebras'' \cite{JonssonTarski1951,Esakia1974,Goldblatt1976,SambinVaccaro1988}. In these semantics, $\Box$ is interpreted as the universal preimage, as we do here. While the white diamond $\Diamond$ is usually interpreted as the preimage, in temporal logic there is also the black diamond $\blacklozenge$, which is interpreted as the forward image and is adjoint to the universal preimage $\Box$, as we do here. A standard reference for temporal logic is \cite{Goldblatt1987}.
Second, in the Vietoris powerlocale (\cite{Johnstone1985}, and also \cite[Sec.~II.4]{Johnstone1982}), $\Box$ preserves all finite meets and directed joins, and $\Diamond$ preserves all joins.
Something similar happens in Galois morphisms between bi-dcpos.
In particular, $\Box$ preserves all directed joins (as we prove in \cref{l:morphism-preserves-dir-codir}); moreover, since $\Box$ extends to a right adjoint function between the associated embedded bi-dcpos, $\Box$ preserves all the meets that are ``truly meets'', meaning those that are meets also in the embedded bi-dcpo. In particular, by \cref{l:preserve}, the function $\Box$ in a Galois morphism $(\blacklozenge \colon \K_1 \to \K_2, \Box \colon \O_2 \to \O_1)$ preserves all existing finite meets if the set of o-elements above any k-element in $\K_2$ is codirected, i.e.\ if each element of $\K_2$ can be identified with a (necessarily Scott-open) filter of $\O_2$. Dually, $\blacklozenge$ preserves codirected meets and all existing joins that are preserved by the inclusion of $\K_1$ in the concept lattice of $(\K_1, \O_1, \ko_1)$.}

We prove that the morphisms preserve the whole relevant structure. The proof is essentially the same as that of \cref{l:Esakia}.

\begin{lemma} [A version of Esakia's lemma, pointfree]\label{l:morphism-preserves-dir-codir}
	Let $(\blacklozenge \colon \K_1 \to \K_2, \Box \colon \O_2 \to \O_1)$ be a Galois morphism from a bi-dcpo $(\K_1, \O_1, \ko_1)$ to a bi-dcpo $(\K_2, \O_2, \ko_2)$. The functions $\blacklozenge$ and $\Box$ preserve codirected meets and directed joins, respectively.
\end{lemma}

\begin{proof}
	We prove that $\blacklozenge$ preserves codirected meets.
	Let $\F$ be a codirected subset of $\K_1$.
	We prove $\blacklozenge(\bigwedge \F) = \bigwedge_{k \in \F} \blacklozenge(k)$.
	The inequality $\blacklozenge(\bigwedge \F) \leq \bigwedge_{k \in \F} \blacklozenge(k)$ is immediate.
	Let us prove the converse inequality, i.e.\ $\bigwedge_{k \in \F} \blacklozenge(k) \leq \blacklozenge(\bigwedge \F)$.
	By definition of the order on $\K_2$, we shall prove that, for every $u \in \O_2$ with $\blacklozenge(\bigwedge \F) \ko_2 u$, we have $\bigwedge_{k \in \F} \blacklozenge(k) \ko_2 u$.
	Let $u \in \O_2$ be such that $\blacklozenge(\bigwedge \F) \ko_2 u$.
	Then, by adjointness, $\bigwedge \F \ko_1 \Box(u)$.
	Then, by double compactness, there is $l \in \F$ such that $l \ko_1 \Box(u)$.
	Then, by adjointness, $\blacklozenge(l) \ko_2 u$.
	Then, $\bigwedge_{k \in \F} \blacklozenge(k) \leq \blacklozenge(l) \ko_2 u$, and so $\bigwedge_{k \in \F} \blacklozenge(k) \ko_2 u$, as desired.
	
	Dually, $\Box$ preserves directed joins.
\end{proof}

\begin{remark}[Lawson self-duality on bi-dcpos]
	If $(\blacklozenge \colon \K_1 \to \K_2, \Box \colon \O_2 \to \O_1)$ is a Galois morphism from a bi-dcpo $(\K_1, \O_1, \ko_1)$ to a bi-dcpo $(\K_2, \O_2, \ko_2)$, then $(\Box \colon \O_2 \to \O_1, \blacklozenge \colon \K_1 \to \K_2)$ is a Galois morphism from $(\O_2, \K_2, \vartriangleright_2)$ to $(\O_1, \K_1, \vartriangleright_1)$.
	Therefore, the category $\Bidcpo$ is self-dual.
\end{remark}

\begin{remark} \label{r:identify-morphisms}
	Recall from \cref{r:identify} that a bi-dcpo is the same thing as a dcpo $\O$ equipped with an order-separating set $\K$ of Scott-open subsets that is closed under directed unions.
	Then, a Galois morphism from $(\K_1, \O_1, \ko_1)$ to $(\K_2, \O_2, \ko_2)$ is precisely a function $\Box \colon \O_2 \to \O_1$ such that the preimage under $\Box$ of an element of $\K_1$ belongs to $\K_2$.
\end{remark}

\cref{r:identify} has the following consequence.

\begin{lemma}[Scott-continuous $\stackrel{1:1}{\leftrightarrow}$ Galois morphisms] \label{l:scott-continuous}
		The set of Scott-continuous functions from a dcpo $D_2$ to a dcpo $D_1$ is in bijection with the set of Galois morphisms from the bi-dcpo $(\ScottOp(D_1), D_1, \ni)$ to the bi-dcpo $(\ScottOp(D_2), D_2, \ni)$.
\end{lemma}

The assignments witnessing the bijection in \cref{l:scott-continuous} are:
\begin{itemize}

	\item
	to a Scott-continuous function $f \colon D_2 \to D_1$ we associate the Galois morphism
	\[
	(f^{-1}[-]\colon \ScottOp(D_1) \to \ScottOp(D_2), f \colon D_2 \to D_1);
	\]
	
	\item
	to a Galois morphism $(\blacklozenge \colon \ScottOp(D_1) \to \ScottOp(D_2), \Box \colon D_2 \to D_1)$ we associate the function $\Box$.
\end{itemize}

Therefore, the category of dcpos and Scott-continuous functions is a full subcategory of the category $\Bidcpo$ of bi-dcpos and Galois morphisms.

An \emph{open filter morphism} \cite[Def.~IV-2.2, p.~281]{GierzHofmannEtAl2003} between open-filter-determined dcpos is a function $f \colon D_1 \to D_2$ such that the preimage under $f$ of a Scott-open filter is a Scott-open filter.
Every open filter morphism between open-filter-determined dcpos is Scott-continuous.\footnote{The proof is essentially the same as the ones of \cref{l:Esakia,l:morphism-preserves-dir-codir}. First of all, we prove that $f$ is order-preserving. Let $u, v \in D_1$ with $u \leq v$, and let us prove $f(u) \leq f(v)$. Since $D_2$ is open-filter-determined, it is enough to prove that, for every Scott-open filter $F$ with $f(u) \in F$, we have $f(v) \in F$. Let $F$ be a Scott-open filter with $f(u) \in F$. Then $u \in f^{-1}[F]$. Since $f^{-1}[F]$ is a Scott-open filter, it is an upset, and hence $v \in f^{-1}[F]$, which implies $f(v) \in F$, as desired. This proves that $f$ is order-preserving. Let $I$ be a directed subset of $D_1$, and let us prove that $\bigvee f[I] = f(\bigvee I)$. The inequality $\leq$ is clear since $f$ is order-preserving. Let us prove the converse inequality, i.e.\ $f(\bigvee I) \leq \bigvee f[I]$. Since $D_2$ is open-filter-determined, it is enough to prove that $\bigvee f[I]$ belongs to every Scott-open filter to which $f(\bigvee I)$ belongs.
Let $F$ be a Scott-open filter such that $f(\bigvee I) \in F$.
Then, $\bigvee I \in f^{-1}[F]$.
Since $f^{-1}[F]$ is a Scott-open filter, there is $u \in I$ such that $u \in f^{-1}[F]$.
Then, $f(u) \in F$, and thus also $\bigvee f[I]$ (which is above $f(u)$) belongs to $F$.
This proves that $f$ preserves directed joins.}
The open filter morphisms between continuous semilattices (= continuous domains which are also semilattices\footnote{Semilattice = poset with all finite meets.}, see e.g.\ \cite[Def.~I-1.6.(iv)]{GierzHofmannEtAl2003}) are precisely the Scott-continuous semilattice homomorphisms (\cite[Prop.~IV-2.1]{GierzHofmannEtAl2003}).
In particular, the open filter morphisms between locally compact frames are precisely the preframe homomorphisms.

\cref{r:identify} has also the following consequence.

\begin{lemma} [Open filter morphisms $\stackrel{1:1}{\leftrightarrow}$ Galois morphisms]\label{l:Scott-open-filter}
	Let $D_1$ and $D_2$ be open-filter-determined dcpos (such as continuous domains or spatial frames).
	The set of open filter morphisms from $D_2$ to $D_1$ is in bijection with the set of Galois morphisms from $(\ScottOpFilt(D_1), D_1, \ni)$ to $(\ScottOpFilt(D_2), D_2, \ni)$.
\end{lemma}

Therefore, also the category of open-filter-determined dcpos and open filter morphisms is a full subcategory of the category $\Bidcpo$ of bi-dcpos and Galois morphisms.
Note that the self-duality of $\Bidcpo$, restricted to the full subcategory of continuous domains, gives Lawson self-duality for continuous domains.

We recall that the open filter morphisms between continuous semilattices are precisely the Scott-continuous maps preserving finite meets \cite[Cor.~IV-2.3]{GierzHofmannEtAl2003}. Therefore, \cref{l:Scott-open-filter} has the following specialization.

\begin{lemma}[Continuous semilattice homomorphisms $\stackrel{1:1}{\leftrightarrow}$ Galois morphisms] \label{l:cont-semilattices}
	Let $\O_1$ and $\O_2$ be continuous semilattices (such as locally compact frames). The set of continuous semilattice homomorphisms from $\O_2$ to $\O_1$ (a.k.a.\ preframe homomorphisms, for locally compact frames) is in bijection with the set of Galois morphisms from $(\ScottOpFilt(\O_1), \O_1, \ni)$ to $(\ScottOpFilt(\O_2), \O_2, \ni)$.
\end{lemma}

Therefore, the category of continuous semilattices and continuous semilattice homomorphisms is a full subcategory of the category $\Bidcpo$ of bi-dcpos and Galois morphisms.
Note that the self-duality of $\Bidcpo$, restricted to the full subcategory of continuous semilattices, gives Lawson self-duality for continuous semilattices.

\subsubsection{Morphisms of embedded bi-dcpos}

Given a c-relation $R \colon X \rel Y$ between ko-spaces, the function $R[-] \colon \Up(X) \to \Up(Y)$ is left adjoint to the function $R^\leftarrow[-] \colon \Up(Y) \to \Up(X)$.
The pair of functions $(R[-], R^\leftarrow[-])$ will be the prototypical morphism from the embedded bi-dcpo $\Up(X)$ to the embedded bi-dcpo $\Up(Y)$. 

\begin{definition}[$\EBidcpo$, $\DEBidcpo$]
	We let $\EBidcpo$ denote the category of embedded bi-dcpos and embedded Galois morphisms between them (see \cref{l:double-base-morphism}). Moreover, we let  $\DEBidcpo$ denote the full subcategory of $\EBidcpo$ on distributive embedded bi-dcpos.
\end{definition}

\subsubsection{On morphisms: bi-dcpos $\stackrel{1:1}{\leftrightarrow}$ embedded bi-dcpos}

The following is straightforward.

\begin{theorem}[Categorical equivalence between bi-dcpos and embedded bi-dcpos] \label{t:equiv-protolocale}
	The equivalence in \cref{t:equiv-formal-concept-analysis} restricts to an equivalence between the full subcategories $\Bidcpo$ of bi-dcpos and $\EBidcpo$ of embedded bi-dcpos.
\end{theorem}

\subsection{The (dual) equivalences}

\subsubsection{Raney duality and its extension to relations}

In this subsection we recall Raney duality for posets and order-preserving maps and its extension to weakening relations (in place of order-preserving maps).

The idea behind Raney duality is that a poset can be recovered by its lattice of upsets, as in \cref{t:char-Raney}.
Given an order-preserving function $f \colon X \to Y$, the preimage function $f^{-1} \colon \Up(Y) \to \Up(X)$ is a complete lattice homomorphism, and every complete lattice homomorphism $\Up(Y) \to \Up(X)$ is of this form.
This is the main content of Raney duality.

\begin{theorem}[Raney duality] \label{t:Raney}
	The category of Raney lattices and complete lattice homomorphisms is dually equivalent to the category of posets and order-preserving functions.
\end{theorem}

Given Raney lattices $L_1$ and $L_2$, we call an \emph{adjoint pair} from $L_1$ to $L_2$ a pair $(\blacklozenge,\Box)$ where $\blacklozenge \colon L_1 \to L_2$ and $\Box \colon L_2 \to L_1$ are functions such that $\blacklozenge$ is left adjoint to $\Box$, i.e.: for all $x \in L_1$ and $y \in L_2$, $\blacklozenge(x) \leq y$ if and only if $x \leq \Box(y)$.

Raney duality (\cref{t:Raney}) admits the following generalization.

\begin{theorem}[Raney duality for posets and weakening relations] \label{t:Raney-for-relations}
	The category of Raney lattices and adjoint pairs between them is equivalent (and thus also dually equivalent) to the category of posets and weakening relations between them.
\end{theorem}

\begin{proof}
	The object part is the same as in the classical Raney duality (see \cref{t:char-Raney}).
	The bijection between homsets is \cite[Lem.~6.6]{GalatosJipsen2020}.
\end{proof}

The covariant bijective correspondence between the homset between two posets and the homset between their upset lattices is as follows.
A weakening relation $R \colon X \rel Y$ corresponds to the pair 
\[
	(R[-] \colon \Up(X) \to \Up(Y), R^\leftarrow[-] \colon \Up(Y) \to \Up(X)),
\]
and an adjoint pair
\[
(\blacklozenge \colon \Up(X) \to \Up(Y), \Box \colon \Up(Y) \to \Up(X))
\]
corresponds to the relation $R \colon X \rel Y$ defined by $x \r{R} y$ if and only if $\blacklozenge(\u x) \nleq Y \setminus \d y$ if and only if $\u x \nleq \Box(Y \setminus \d y)$.

\subsubsection{On morphisms: ko-spaces $\stackrel{1:1}{\leftrightarrow}$ distributive (embedded) bi-dcpos}

We arrive at our main result.

\begin{theorem}[Categorical equivalence between ko-spaces and distributive (embedded) bi-dcpos] \label{t:MAIN}
	
	The following categories are equivalent (and thus, being self-dual, also dually equivalent):
	\begin{enumerate}
	
		\item
		the category $\KO$ of ko-spaces and c-relations;
		
		\item
		the category $\DBidcpo$ of distributive bi-dcpos and Galois morphisms;
		
		\item
		the category $\DEBidcpo$ of distributive embedded bi-dcpos and embedded Galois morphisms.
	
	\end{enumerate}
\end{theorem}

\begin{proof}
	In light of the one-to-one correspondence between ko-spaces and distributive embedded bi-dcpos established in \cref{t:bij-corr}, the bijective correspondence on morphisms between $\KO$ and $\DEBidcpo$ follows from the correspondence between adjoint pairs and weakening relations in \cref{t:Raney-for-relations}.
	
	In light of the one-to-one correspondence between distributive embedded bi-dcpos and distributive bi-dcpos, the bijective correspondence on morphisms between $\DEBidcpo$ and $\DBidcpo$ follows from the correspondence between embedded Galois morphisms and Galois morphisms in \cref{t:equiv-formal-concept-analysis}.
\end{proof}

The equivalence between ko-spaces and distributive bi-dcpos generalizes the equivalence between sober spaces and spatial frames.
So, via \cref{t:MAIN}, we can embed the duality between sober spaces and spatial frames into a self-symmetric context.
This context includes also all well-filtered $T_0$ spaces.


\section{Local compactness} \label{s:local-compactness}

In the previous section we arrived at our main result: the equivalences (and dual equivalences) between ko-spaces, distributive bi-dcpos and distributive embedded bi-dcpos.
This section contains some additional material, in which we focus on local compactness. To be precise, we should say that we focus on \emph{bicontinuity}, a strengthening of local compactness that collapses to the latter in the topological setting.

The message is that bicontinuity makes the two-sorted approach equally expressive to the one-sorted approach.
This is embodied in our presentation of continuous domains as a special class of bi-dcpos (\cref{t:continuous-domains-and-bicontinuous-bi-dcpos}), which may be seen as an explanation for the existence of Lawson duality for continuous domains.
Moreover, we prove that, under bicontinuity, the distributivity of a bi-dcpo can be checked just by looking at the distributivity of either the dcpo or the co-dcpo of which the bi-dcpo is made, if the given dcpo or co-dcpo is a lattice (\cref{t:distributivity-from-O-K}).
In particular, this gives another perspective on why local compactness of a frame implies spatiality: local compactness allows one to transfer the distributivity of the underlying lattice of a frame to the distributivity (and hence spatiality) of the concept lattice of the bi-dcpo associated to it.

We recall that a topological space $X$ is \emph{locally compact} if for every $x \in X$ and every open set $U$ with $x \in U$ there are an open set $U'$ and a compact set $K'$ such that $x \in U' \subseteq K' \subseteq U$ (see \cite[p.~211]{HofmannMislove1981}; for textbook references, see \cite[O-5.9]{GierzHofmannEtAl2003} or \cite[Def.~4.8.1]{GoubaultLarrecq2013}).

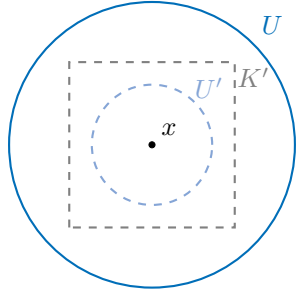
\begin{center}
	\begin{tikzpicture}
	
	  \draw[thick, NavyBlue] (0,0) circle (1.9);
	  \node[NavyBlue] at (1.6,1.6) {$U$};
	
	  \draw[thick, gray, dashed] (-1.1,-1.1) rectangle (1.1,1.1);
	  \node[gray] at (1.35,0.9) {$K'$};
	
	  \draw[thick, dashed, NavyBlue!50, dashed] (0,0) circle (0.8);
	  \node[NavyBlue!50] at (0.76,0.76) {$U'$};

	  \filldraw[black] (0,0) circle (0.04) node[above right] {$x$};
	
	\end{tikzpicture}
\end{center}
For Hausdorff spaces (but not in general \cite[p.~93]{GoubaultLarrecq2013}), this condition is equivalent to the condition that every point has a compact neighborhood \cite[Prop.~4.8.7]{GoubaultLarrecq2013}. For any topological space, local compactness can be expressed equivalently as\footnote{For the equivalence, see e.g.\ \cite[Prop.~4.8.14]{GoubaultLarrecq2013} or \cite[Lem.~7.3]{Erne2007}.}:
\begin{quotation}
	for every compact saturated set $K$ and every open set $U$ with $K \subseteq U$, there are an open set $U'$ and a compact saturated set $K'$ such that $K \subseteq U' \subseteq K' \subseteq U$.
\end{quotation}
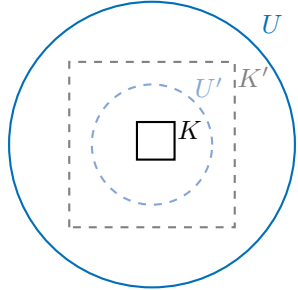
\begin{center}
	\begin{tikzpicture}
	
	  \draw[thick, NavyBlue] (0,0) circle (1.9);
	  \node[NavyBlue] at (1.6,1.6) {$U$};
	
	  \draw[thick, gray, dashed] (-1.1,-1.1) rectangle (1.1,1.1);
	  \node[gray] at (1.35,0.9) {$K'$};
	
	  \draw[thick, dashed, NavyBlue!50, dashed] (0,0) circle (0.8);
	  \node[NavyBlue!50] at (0.76,0.76) {$U'$};
	
	  \draw[thick] (-0.2,-0.2) rectangle (0.3,0.3);
	  \node at (0.5,0.2) {$K$};
	
	\end{tikzpicture}
\end{center}
Because of its symmetry between open sets and compact saturated sets, this is the condition we adapt to our context.
We now introduce some notation that captures the inclusion of an open set into a compact set in the setting of bi-dcpos.

\begin{notation}[$\ok$] \label{n:con}
	Whenever the white triangle $\ko$ denotes the relation involved in a bi-dcpo $(\K, \O, \ko)$, the black triangle $\ok$ denotes the relation $\ok \colon \O \rel \K$ defined by: 
	\[
	\text{$u \ok k$ $\iff$ for all $l \in \dk u$ and $v \in \uo k$ we have $l \ko v$.}
	\]
	Equivalently, $\uo k \subseteq \uo u$, or, still equivalently, $\dk u \subseteq \dk k$. 
\end{notation}

\begin{remark} \label{r:ko}
	The relation $\ok$ is the restriction to $\O \times \K$ of the hierarchical order on the concept lattice of $\ko$.
	This follows from the facts that $\K$ and $\O$ are join- and meet- dense.
\end{remark}

\begin{definition}[Locally compact]\hfill
	\begin{enumerate}
	
		\item
		A ko-space $(X, \K, \O)$ is \emph{locally compact} if for all $K \in \K$ and $U \in \O$ with $K \subseteq U$ there are $U' \in \O$ and $K' \in \K$ such that $K \subseteq U' \subseteq K' \subseteq U$.
		
		\item
		An embedded bi-dcpo $(L, \K, \O)$ is \emph{locally compact} if for all $k \in \K$ and $u \in \O$ with $k \leq u$ there are $u' \in \O$ and $k' \in \K$ such that $k \leq u' \leq k' \leq u$.
		
		\item
		A bi-dcpo $(\K, \O, \ko)$ is \emph{locally compact} if for all $k \in \K$ and $u \in \O$ with $k \ko u$ there are $u' \in \O$ and $k' \in \K$ such that $k \ko u' \ok k' \ko u$.
	
	\end{enumerate}
\end{definition}

For local compactness to express its power, a further condition is helpful, which in the topological setting is invisible; we call this strengthening of local compactness ``bicontinuity''. The terminology is inspired by the notion of continuity for a dcpo \cite[Definition I-1.6.(ii)]{GierzHofmannEtAl2003}.

\begin{definition}[Bicontinuous]\hfill
	\begin{enumerate}
		
		\item
		A ko-space $(X, \K, \O)$ is \emph{bicontinuous} if it is locally compact, for every $U \in \O$ the set $\{K \in \K \mid K \subseteq U\}$ is directed, and for every $K \in \K$ the set $\{U \in \O \mid K \subseteq U\}$ is codirected.
	
		\item
		A bi-dcpo $(\K, \O, \ko)$ is \emph{bicontinuous} if it is locally compact, for every $u \in \O$ the set $\dk u$ is directed, and for every $k \in \K$ the set $\uo k$ is codirected.
		
		\item
		An embedded bi-dcpo $(L, \K, \O)$ is \emph{bicontinuous} if it is locally compact, for every $u \in \O$ the set $\dk u$ is directed, and for every $k \in \K$ the set $\uo k$ is codirected.
		
	\end{enumerate}
\end{definition}

In the topological context, i.e.\ for triples of the form $(X, \KSat(X), \Op(X))$ for a well-filtered $T_0$ topological space $X$, bicontinuity is equivalent to local compactness; this is because open sets are closed under finite intersections and compact saturated sets under finite unions.

Roughly speaking, the reason why bicontinuity works well is that it allows building a k-set out of a Scott-open filter of o-sets (\cref{t:HofMis} below). 

\subsection{A two-sorted presentation of continuous domains}

In this subsection, we show that continuous domains are in one-to-one correspondence with bicontinuous bi-dcpos (under the axiom of dependent choice).
We start by proving that a continuous domain yields a bicontinuous bi-dcpo.

\begin{lemma} \label{l:cont-dom-is-bicont-bi-dcpo}
	If $D$ is a continuous domain, then $(\ScottOpFilt(D), D, \ni)$ is a bicontinuous bi-dcpo.
\end{lemma}

\begin{proof}
	Every continuous domain is open-filter-determined \cite[Prop.~I-3.3.(ii)]{GierzHofmannEtAl2003}, and so the bi-dcpo $(\ScottOpFilt(D),D, \ko)$ with $\ko$ the reverse membership $\ni$ is a bi-dcpo (see \cref{ex:open-filter-determined}).
	
	For every $u \in D$ and every Scott-open filter $K$, we have
	\begin{align*}
		u \ok K & \iff \dk u \subseteq \dk K\\
		& \iff \text{for every Scott-open filter $L$ with $u \in L$ we have $K \subseteq L$}\\
		& \iff K \subseteq \u u &&\text{(since $D$ is open-filter-determined).}
	\end{align*}
	
	We show that $(\ScottOpFilt(D), D, \ko)$ is locally compact.
	Let $K \in \ScottOpFilt(D)$ and $u \in D$ be such that $K \ko u$, i.e.\ $u \in K$.
	We shall prove that there are $v \in D$ and $L \in \ScottOpFilt(D)$ such that $K \ko v \ok L \ko u$, i.e.\ $v \in K$, $L \subseteq \u v$ and $u \in L$.
	This is \cite[Cor.~3]{Erne2016} (and requires the axiom of dependent choice).
	Therefore, $(\ScottOpFilt(D), D, \ko)$ is locally compact.
	
	For every $K \in \ScottOpFilt(D)$, the set ${\uparrow_D}K = K$ is a codirected subset of $D$.
	Moreover, for every $u \in D$, the set $\downarrow_{\ScottOpFilt(D)} u =  \{K \in \ScottOpFilt(D) \mid u \in K\}$ is directed with respect to the reverse inclusion order \cite[Lem.~IV-2.9(ii)]{GierzHofmannEtAl2003}.
	Therefore, $(\ScottOpFilt(D), D, \ni)$ is bicontinuous.
\end{proof}

The following is reminiscent of the fact that, in a locally compact space, an open set is way below another one if and only if there is a compact set between the first one and the second one \cite[Prop.~I-1.4]{GierzHofmannEtAl2003}.

\begin{lemma}\label{l:directed-with-join}
	Let $(\K, \O, \ko)$ be a locally compact bi-dcpo, and let $u \in \O$ be such that $\dk u$ is directed.
	\begin{enumerate}
	
		\item \label{i:directed-with-join}
		The set $\{v \in \O \mid \exists k \in \K : v \ok k \ko u\}$ is directed and $u$ is its join  (both in $\O$ and in the associated embedded bi-dcpo).
		
		\item\label{i:way-below-equiv}
		For all $v \in \O$, $v$ is way below $u$ in the dcpo $\O$ if and only if there is $k \in \K$ such that $v \ok k \ko u$.
		
		\item \label{i:is-dir-join}
		In the dcpo $\O$, $u$ is the directed join of the elements way below it.
	
	\end{enumerate}
	 
\end{lemma}

\begin{proof}
	\eqref{i:directed-with-join}.
	Set $I \coloneqq \{k \in \K \mid k \ko u\}$ and $J \coloneqq \{v \in \O \mid \exists k \in \K : v \ok k \ko u\}$.
	Every element of $J$ has an upper bound in $I$ (in the concept lattice of $(\K, \O, \ko)$) by definition of $J$ and $I$, and every element of $I$ has an upper bound in $J$ by local compactness.
	Thus, since $I$ is directed and its join is $u$, the same holds for $J$.
		
	\eqref{i:way-below-equiv}.
	Let $v \in \O$.
	If there is $k \in \K$ such that $v \ok k \ko u$, then, by double compactness of the bi-dcpo $(\K, \O, \ko)$, $v$ is way below $u$, settling the right-to-left implication in the statement.
	To prove the left-to-right implication, we suppose that $v$ is way below $u$.
	By \eqref{i:directed-with-join}, $u$ is the directed join of $\{w \in \O \mid \exists k \in \K : w \ok k \ko u\}$; since $v$ is way below $u$, it follows that $v$ is below some element of $\{w \in \O \mid \exists k \in \K : w \ok k \ko u\}$; this leads to the desired conclusion.
	
	\eqref{i:is-dir-join}.
	This follows immediately from \eqref{i:directed-with-join} and \eqref{i:way-below-equiv}.
\end{proof}

From \cref{l:directed-with-join} and its dual we obtain the following result.

\begin{corollary} \label{c:way-char-dir-codir}
	If $(\K, \O, \ko)$ is a bicontinuous bi-dcpo, then $\O$ and $\K^\op$ are continuous domains.
\end{corollary}

\begin{theorem} \label{t:HofMis}
	Let $(\K, \O, \ko)$ be a bicontinuous bi-dcpo.
	\begin{enumerate}
	
		\item \label{i:HM-1}
		The assignment $k \mapsto \uo k$ is a bijection from $\K$ to the set of Scott-open filters of $\O$.
		
		\item \label{i:HM-2}
		The assignment $u \mapsto \dk u$ is a bijection from $\O$ to the set of Scott-open filters of $\K^\op$.
	
	\end{enumerate}
	In particular, $\K^\op$ and $\O$ are Lawson dual continuous domains.
\end{theorem}

\begin{proof}
	We only prove \eqref{i:HM-1}, since \eqref{i:HM-2} is dual.
	
	For every $k \in \K$, the set $\uo k$ is a filter by bicontinuity and is Scott-open by double compactness; thus, the map in the statement is well-defined.
	Moreover, it is injective since $(\K, \O, \ko)$ is purified.
	To prove surjectivity, let $F$ be a Scott-open filter of $\O$.
	Set
	\[
		G \coloneqq \{k \in \K \mid \exists u \in F : u \ok k\}.
	\]
	By \cref{l:directed-with-join} and by Scott-openness of $F$, for every $u \in F$ there are $v \in F$ and $k \in \K$ such that $v \ok k \ko u$; i.e., for every $u \in F$ there is $k \in G$ such that $k \ko u$.
	This, together with the definition of $G$ and the fact that $F$ is a filter, implies that $G$ is codirected.
	We claim that $F = \uo {\bigwedge G}$.
	The left-to-right inclusion $F \subseteq \uo {\bigwedge G}$ follows from the already established fact that for every $u \in F$ there is $k \in G$ such that $k \ko u$.
	Let us prove the converse inclusion, i.e.\ $\uo {\bigwedge G} \subseteq F$.
	Let $u \in \uo  {\bigwedge G}$.
	Since $G$ is codirected, by double compactness there is $l \in G$ such that $l \ko u$.
	Then, by definition of $G$, there is $v \in F$ such that $v \ok l$.
	Therefore, $v \ok l \ko u$, and thus $v \leq u$, and hence, since $F$ is upwards closed, $u \in F$.
\end{proof}

We arrive at our presentation of continuous domains as special bi-dcpos.

\begin{theorem} \label{t:continuous-domains-and-bicontinuous-bi-dcpos}
	The following categories are equivalent.
	\begin{enumerate}
	
		\item \label{i:contdom}
		The category of continuous domains and open filter morphisms.
		
		\item \label{i:bicont}
		The category of bicontinuous bi-dcpos and Galois morphisms between them.
		
		\item \label{i:bicont-emb}
		The category of bicontinuous embedded bi-dcpos and embedded Galois morphisms between them.
	
	\end{enumerate}
\end{theorem}

\begin{proof}
	\eqref{i:contdom}$\leftrightarrow$\eqref{i:bicont}.
	On objects, the assignment \eqref{i:contdom}$\rightarrow$\eqref{i:bicont} maps a continuous domain $D$ to the bicontinuous bi-dcpo $(\ScottOpFilt(D), D, \ni)$.
	This is a well-defined assignment by \cref{l:cont-dom-is-bicont-bi-dcpo}.
	On objects, the assignment \eqref{i:bicont}$\rightarrow$\eqref{i:contdom} maps a bicontinuous bi-dcpo $(\K, \O, \ko)$ to the continuous domain $\O$.
	This is a well-defined assignment by \cref{c:way-char-dir-codir}.
	These assignments are mutually inverse by \cref{t:HofMis}.
	
	On morphisms, to an open filter morphism $f \colon D_1 \to D_2$ we assign the pair 
	\[
	(f^{-1}[-] \colon \ScottOpFilt(D_2) \to \ScottOpFilt(D_1), f \colon D_1 \to D_2).
	\]
	Conversely, to a Galois morphism $(\blacklozenge,\Box)$ from $(\K_1, \O_1, \ko_1)$ to $(\K_2, \O_2, \ko_2)$ we assign the function $\Box \colon \O_2 \to \O_1$.
	The bijective correspondence on morphisms is the content of \cref{l:Scott-open-filter}.
	
	\eqref{i:bicont}$\leftrightarrow$\eqref{i:bicont-emb}.
	This is the restriction of the equivalence in \cref{t:equiv-protolocale}.
\end{proof}

\begin{remark} \label{r:Lawson-duality}
	Under the equivalence in \cref{t:continuous-domains-and-bicontinuous-bi-dcpos}, Lawson self-duality for continuous domains (proved by Lawson in \cite{Lawson1979}; see also \cite[Thm.~IV-2.14]{GierzHofmannEtAl2003} for details) corresponds to the self-duality for bicontinuous bi-dcpos that, on objects, maps $(\K, \O, \ko)$ to $(\O, \K, \ko^\partial)$, i.e.\ that swaps k-elements and o-elements.
	While Lawson self-duality makes use of the axiom of dependent choice (to construct Scott-open filters), the duality for bicontinuous bi-dcpos does not.
	This is similar to \cite{Erne2016}, where the category of $\delta$-domains is proved to be self-dual under no choice principles, and equivalent to the category of continuous domains under dependent choice.
\end{remark}

\subsection{Interdefinability of opens and compacts}

In passing, we draw the following corollaries of \cref{t:HofMis}, which show that $\K$ and $\O$ are interdefinable under bicontinuity.
These are reminiscent of what is called ``Hofmann--Mislove Theorem III'' in \cite[Thm.~IV-2.18]{GierzHofmannEtAl2003}, which asserts that, for a locally compact sober space $X$, the topology of $X$ and the dcpo of compact saturated subsets of $X$ ordered by reverse inclusion are Lawson dual continuous domains.

\begin{corollary}[of \cref{t:HofMis}] \label{c:HM-III}
	Let $(X, \K, \O)$ be a bicontinuous ko-space. 
		\begin{enumerate}
	
		\item \label{i:space-Lawson}(Lawson duals) 
		$\O$ and $\K^\op$ are Lawson dual continuous domains:
		
		\begin{enumerate}
	
			\item \label{i:HM-1-pointset}
			the assignment $K \mapsto \{ U \in \O \mid K \subseteq U \}$ is a bijection from $\K$ to the set of Scott-open filters of $\O$;
			
			\item \label{i:HM-2-pointset}
			the assignment $U \mapsto \{ K \in \K \mid K \subseteq U \}$ is a bijection from $\O$ to the set of Scott-open filters of $\K^\op$.
		
		\end{enumerate}
		
		\item\label{i:space-deGroot} (De Groot duals) 
		For every subset $S$ of $X$,
		
		\begin{enumerate}
	
			\item \label{i:K-is-compact-saturated}
			$S \in \K$ if and only if $\{U \in \O \mid S \subseteq U\}$ is a Scott-open filter of $\O$ with intersection $S$, i.e.
			\begin{enumerate}
			
				\item (compact)
				for every directed subset $\I$ of $\O$ with $S \subseteq \bigcup \I$ there is $U \in \I$ with $S \subseteq U$,
				
				\item \label{i:codir-intersection} (saturated)
				$S$ is an upset and $\{U \in \O \mid S \subseteq U\}$ is codirected;
			
			\end{enumerate}
	
			\item  \label{i:O-is-cocompact-saturated}
			$S \in \O$ if and only if $\{K \in \K \mid K \subseteq S\}$ is a Scott-open filter of $\K^\op$ with union $S$, i.e.
			\begin{enumerate}
			
				\item (co-compact)
				for every codirected subset $\F$ of $\K$ with $\bigcap \F \subseteq S$ there is $K \in \F$ with $K \subseteq S$,
				
				\item \label{i:dir-union} (co-saturated) 
				$S$ is an upset and $\{K \in \K \mid K \subseteq S\}$ is directed.
		
			\end{enumerate}
	
		\end{enumerate}
	
	\end{enumerate}
\end{corollary}

\begin{proof}
	\eqref{i:space-Lawson} follows from \cref{t:HofMis} since $(\K, \O, \subseteq)$ is a bicontinuous bi-dcpo, and \eqref{i:space-deGroot} follows from \eqref{i:space-Lawson}.
\end{proof}

The ``Hofmann--Mislove Theorem III'' in \cite[Thm.~IV-2.18]{GierzHofmannEtAl2003} follows from \cref{c:lc-d-prespace} because, if $X$ is a locally compact sober space, then $(X, \KSat(X), \Op(X))$ is a locally compact ko-space with $\Op(X)$ closed under finite intersections and $\KSat(X)$ under finite unions (more details on this implication are in the proof of \cref{t:lcf-lcss}).

\begin{corollary} \label{c:lc-d-prespace}
	Let $(X, \K, \O)$ be a locally compact ko-space with $\O$ closed under finite intersections and $\K$ under finite unions.
	\begin{enumerate}
	
		\item
		$\K$ is the set of upsets $S$ of $X$ such that for every directed subset $\I$ of $\O$ with $S \subseteq \bigcup \I$ there is $U \in \I$ with $S \subseteq U$.
	
		\item
		$\O$ is the set of upsets $S$ of $X$ such that for every codirected subset $\F$ of $\K$ with $\bigcap \F \subseteq S$ there is $K \in \F$ with $K \subseteq S$.
	
	\end{enumerate}
\end{corollary}

We obtain also the following pointfree version of \cref{c:HM-III}.

\begin{corollary}[of \cref{t:HofMis}] \label{c:double-duals-pointfree}
	Let $(L, \K, \O)$ be a bicontinuous embedded bi-dcpo. 
	\begin{enumerate}
	
		\item (Lawson duals)
		$\O$ and $\K^\op$ are Lawson dual continuous domains:
		
		\begin{enumerate}
	
			\item
			the assignment $k \mapsto \uo k$ is a bijection from $\K$ to the set of Scott-open filters of $\O$;
			
			\item
			the assignment $u \mapsto \dk u$ is a bijection from $\O$ to the set of Scott-open filters of $\K^\op$.
		
		\end{enumerate}
		
		\item (De Groot duals)
		For every $a \in L$,
		
		\begin{enumerate}
	
			\item \label{i:K-is-compact-saturated-pointfree}
			$a \in \K$ if and only if $\{u \in \O \mid a \leq u\}$ is a Scott-open filter of $\O$, i.e.
			\begin{enumerate}
			
				\item 
				for every directed subset $I$ of $\O$ with $a \leq \bigvee I$ there is $u \in I$ with $a \leq u$, and
				
				\item\label{i:dG-codir}
				$\{u \in \O \mid a \leq u\}$ is codirected;
			
			\end{enumerate}
	
			\item \label{i:O-is-cocompact-saturated-pointfree}	
			$a \in \O$ if and only if $\{k \in \K \mid k \leq a\}$ is a Scott-open filter of $\K^\op$, i.e.
			\begin{enumerate}
			
				\item
				for every codirected subset $F$ of $\K$ with $\bigwedge F \leq a$ there is $k \in F$ with $k \leq a$, and
				
				\item\label{i:dG-dir}
				$\{k \in \K \mid k \leq a\}$ is directed.
		
			\end{enumerate}
	
		\end{enumerate}
	
	\end{enumerate}
\end{corollary}

\begin{remark}\label{r:necessary-dir-codir}
	The ``directed/codirected'' conditions \eqref{i:codir-intersection} and \eqref{i:dir-union} in \cref{c:HM-III} and \eqref{i:dG-codir} and \eqref{i:dG-dir} in \cref{c:double-duals-pointfree} are necessary, as witnessed by the bicontinuous ko-space $(\varnothing, \varnothing, \varnothing)$ and by its associated bicontinuous embedded bi-dcpo $(\{*\}, \varnothing, \varnothing)$.
\end{remark}

\subsection{A two-sorted presentation of continuous semilattices}

The Lawson self-duality for continuous domains restricts to a self-duality for the full subcategory of continuous semilattices \cite[Thm.~IV-2.16]{GierzHofmannEtAl2003}.
This can be viewed in the setting of bicontinuous bi-dcpos with the following.

\begin{theorem} \label{t:meets-and-joins}
	Let $(\K, \O, \ko)$ be a bicontinuous bi-dcpo.
	$\O$ has finite meets if and only if $\K$ has finite joins.\footnote{The analogous statement where both occurrences of ``finite'' are replaced by ``binary'' is false. A counterexample is the bi-dcpo whose associated embedded bi-dcpo is the four-element diamond-shaped lattice $\bot < a, b < \top$ in which $\O = \K = \{\bot, a, b\}$. This corresponds to the fact that the Lawson dual of a continuous domain with binary meets may fail to possess binary meets, with $\{ \bot, a, b\}$ as the corresponding counterexample.}
\end{theorem}

\begin{proof}
	We only prove the left-to-right implication; the right-to-left one is dual.
	Suppose that $\O$ has finite meets. Let $k_1,\dots,k_n\in\K$.
	For each $i$, by \cref{t:HofMis}\eqref{i:HM-1}, the set $\uo k_i$
	is a Scott-open filter of $\O$. Since $\O$ has finite meets, the
	intersection $F \coloneqq \bigcap_{i=1}^n \uo k_i$
	is again a Scott-open filter of $\O$. Hence, again by
	\cref{t:HofMis}\eqref{i:HM-1}, there is $k\in\K$ such that $F=\uo k$.
	We claim that $k$ is the join of $k_1,\dots,k_n$ in $\K$.
	
	Indeed, for every $u\in\O$,
	\[
	k \ko u
	\quad\Longleftrightarrow\quad
	u\in F
	\quad\Longleftrightarrow\quad
	\forall i,\; k_i \ko u.
	\]
	By the definition of the specialization order on $\K$, this says
	precisely that $k$ is the least upper bound of $k_1,\dots,k_n$ in
	$\K$. Thus, $\K$ has finite joins.
\end{proof}

\black

\begin{remark}
	Lawson self-duality for continuous semilattices \cite[Thm.~IV-2.16]{GierzHofmannEtAl2003} corresponds to the self-duality for bicontinuous bi-dcpos $(\K, \O, \ko)$ such that $\O$ has finite meets or, equivalently, $\K$ has finite joins.
	The latter, on objects, maps $(\K, \O, \ko)$ to $(\O, \K, \ko^\partial)$, i.e.\ swaps k-elements with o-elements.
\end{remark}

We conclude this subsection by deducing from \cref{t:meets-and-joins} its spatial version (\cref{c:unions-iff-intersections} below). For this, we use the following lemma.

\begin{lemma} \label{l:closed-under-bicont}
	Let $(L, \K, \O)$ be a bicontinuous embedded bi-dcpo.
	\begin{enumerate}
		
		\item
		$\K$ is closed under finite joins in $L$ if and only if the poset $\K$ has finite joins.

		\item
		$\O$ is closed under finite meets in $L$ if and only if the poset $\O$ has finite meets.
	
	\end{enumerate}
\end{lemma}

\begin{proof}
	This follows from \cref{l:closed-under}.
\end{proof}

\begin{corollary}[of \cref{t:meets-and-joins}] \label{c:unions-iff-intersections}
	Let $(X, \K, \O)$ be a bicontinuous ko-space.
	$\K$ is closed under finite unions if and only if $\O$ is closed under finite intersections.
\end{corollary}

\begin{proof}
	This follows from \cref{t:meets-and-joins} and \cref{l:closed-under-bicont}.
\end{proof}

\subsection{Local compactness and spatiality}

We now show that, under bicontinuity, the distributivity of a bi-dcpo can be checked just by looking at the distributivity of either the dcpo or the co-dcpo of which the bi-dcpo is made, if any of these happens to be a lattice.
In particular, this gives another perspective on why local compactness of a frame implies spatiality: local compactness allows one to transfer the distributivity of the underlying lattice of a frame to the distributivity (and hence spatiality) of the concept lattice of the bi-dcpo associated to it.

\begin{theorem} \label{t:distributivity-from-O-K}
	Let $(\K, \O, \ko)$ be a bicontinuous bi-dcpo.
	\begin{enumerate}
	
		\item \label{i:O-is-lattice}
		Suppose that $\O$ has binary meets and binary joins.\footnote{By \cref{l:preservation-of-meets-and-joins,r:preserve-binary}, these hypotheses are equivalent to $\O$ being closed under binary meets and binary joins in the associated embedded bi-dcpo.} Then the bi-dcpo $(\K, \O, \ko)$ is distributive if and only if the lattice $\O$ is distributive.
		
		\item \label{i:K-is-lattice}
		Suppose that $\K$ has binary joins and binary meets.
		Then the bi-dcpo $(\K, \O, \ko)$ is distributive if and only if the lattice $\K$ is distributive.
		
	\end{enumerate}
\end{theorem}

\begin{proof}
	\eqref{i:O-is-lattice}.
	We work in the associated embedded bi-dcpo $(L, \K, \O)$.
	First of all, we note that all binary joins in $\O$ are computed as in $L$ by \cref{l:preservation-of-meets-and-joins}, and all binary meets in $\O$ are computed as in $L$ by \cref{r:preserve-binary}.
	This already proves the implication ``$\Rightarrow$''.
	We are left to prove the implication ``$\Leftarrow$''.
	
	Let $k,l\in\mathcal K$ and $u,v\in\mathcal O$ be such that $k \leq u \lor l$, $l \leq v$ and $v \land k \leq u$.
	We must prove $k \leq u$.
	
	From $k \leq u \lor l$ and $l\leq v$, we deduce $k \leq u\vee v$.
	
	Define
	\[
	H \coloneqq  \{w\in\mathcal O\mid k \leq w\vee u\}.
	\]
	We claim that $H$ is a Scott-open filter of $\mathcal O$.
	It is upward closed. It is nonempty, since $v\in H$. If $w_1,w_2\in H$, then
	$k \leq w_1\vee u$ and $k \leq w_2\vee u$.
	Then, using the fact that binary joins and binary meets in $\O$ are computed as in $L$, and that $\O$ is distributive, we have
	\[
	k \leq (w_1 \vee u) \land (w_2 \vee u) = (w_1 \land w_2) \vee u.
	\]
	Therefore, $w_1 \land w_2 \in H$.
	Finally, if $D\subseteq\mathcal O$ is directed and
	$\bigvee D\in H$, then
	\[
	k \leq \Bigl(\bigvee D\Bigr) \lor u = \bigvee_{d\in D}(d\vee u)
	\]
	Since $(\K,\O, \ko)$ is bicontinuous, by \cref{t:HofMis} the set $\uo k$ is Scott-open.
	Therefore, there is $d\in D$ with $d\vee u\in \uo k$, i.e., $d\in H$. Thus, $H$ is Scott-open.
	
	Therefore, $H$ is a Scott-open filter of $\O$.
	Thus, by \cref{t:HofMis}, there is $h \in \K$ such that $H = \uo h$.
	Since $\uo k\subseteq H \subseteq \uo h$, we have $h\leq k$; and since
	$v\in H=\uo h$, we have $h \leq v$.
	Therefore, $h \leq v \land k \leq u$.
	Thus, $u\in \uo h=H$.
	Therefore $k \leq u\vee u$, and so $k \leq u$, as required.
	
	\eqref{i:K-is-lattice}. This is dual to \eqref{i:O-is-lattice}.
\end{proof}

In brief,
\[
\text{bicontinuity + lattice distributivity of $\O$ or $\K$ $\Rightarrow$ spatiality}.
\]

We recall that a frame is \emph{locally compact} if it is continuous as a dcpo (see e.g.\ \cite[p.~310]{Johnstone1982}). A corollary of \cref{t:distributivity-from-O-K} is the following well-known fact (see e.g.\ \cite[Thm.~4.3, p.~311]{Johnstone1982}).

\begin{corollary} \label{c:lc-is-spatial}
	Every locally compact frame is spatial.
\end{corollary}

\begin{proof}
	By \cref{l:cont-dom-is-bicont-bi-dcpo}, $(\ScottOpFilt(\O), \O, \ni)$ is a bicontinuous bi-dcpo.
	By \cref{t:distributivity-from-O-K}\eqref{i:O-is-lattice}, and since any frame is a distributive lattice, the bi-dcpo $(\ScottOpFilt(\O), \O, \ni)$ is distributive.
	By \cref{c:spatial-iff-distributive}, $(\ScottOpFilt(\O), \O, \ni)$ is spatial as a bi-dcpo. Let $(X, \K', \O')$ be the corresponding ko-space.
	The inclusion $\O' \subseteq \Up(X)$ preserves directed joins.
	By \cref{l:preservation-of-meets-and-joins}, it preserves also finite joins, and by \cref{l:preserve} it preserves also finite meets.
	Therefore, $\O'$ is a topology.\footnote{Another way to prove the statement is by observing that every locally compact frame has enough Scott-open filters, and that therefore it is distributive by \cref{ex:distributive}. Yet another way is to observe that distributivity implies having enough completely prime pairs (\cref{t:distr-spatial}) and that these bijectively correspond to points of $\O$ in the sense of frame theory (\cref{r:point=cp-pair}).}
\end{proof}

We now strengthen \cref{c:lc-is-spatial} to recover the duality between locally compact frames and locally compact sober spaces.

\begin{theorem} \label{t:lcf-lcss}
	There is a one-to-one correspondence between
	\begin{enumerate}
	
		\item \label{i:lcf}
		locally compact frames;
		
		\item \label{i:lcf-pl}
		bicontinuous bi-dcpos $(\K, \O, \ko)$ with $\O$ a bounded distributive lattice;
		
		\item \label{i:lcf-epl}
		locally compact distributive embedded bi-dcpos $(L, \K, \O)$ with $\O$ closed in $L$ under finite meets and finite joins and $\K$ under finite joins;
		
		\item \label{i:loccomp-sob-dps}
		locally compact ko-spaces $(X, \K, \O)$ with $\O$ closed under finite intersections and finite unions and $\K$ under finite unions;
		
		\item \label{i:locally-compact-sober-space}
		locally compact sober spaces.
	\end{enumerate}
\end{theorem}

\begin{proof}

	\eqref{i:lcf}$\leftrightarrow$\eqref{i:lcf-pl}.
	Given a locally compact frame $\O$, $(\ScottOpFilt(\O), \O, \ni)$ is a bicontinuous bi-dcpo by \cref{l:cont-dom-is-bicont-bi-dcpo}, and $\O$ is clearly a bounded distributive lattice (the specialization order on $\O$ is the same as the original order of the frame $\O$).
	Let now $(\K, \O, \ko)$ be a bicontinuous bi-dcpo such that $\O$ is a bounded distributive lattice.
	By \cref{t:HofMis}\eqref{i:HM-1}, the assignment $k \mapsto \uo k$ is a bijection from $\K$ to the set of Scott-open filters of $\O$.
	To prove that $\O$ is a locally compact frame, it remains to prove the infinite distributive law restricted to directed joins.\footnote{In the main body we give a direct proof for this restricted distributivity law. However, we note that this restricted distributive law is also a consequence of an interesting result in \cite[Thm.~5.1]{Erne2007}, which affirms that any dcpo with all finite meets and with enough Scott-open filters is a preframe.}
	Let $I$ be a directed subset of $\O$, and let $u \in \O$; we shall prove
	\[
		u \land \bigvee I = \bigvee_{v \in I} (u \land v).
	\]
	The inequality $\geq$ is immediate.
	Let us prove the converse inequality, i.e.\ $u \land \mleft(\bigvee I \mright) \leq \bigvee_{v \in I} (u \land v)$.
	It is enough to prove that every $k \in \K$ below $u \land \bigvee I$ is below $\bigvee_{v \in I} (u \land v)$.
	Let $k$ be such.
	Then, $k \leq u$ and $k \leq \bigvee I$.
	Therefore, by double compactness, there is $v_0 \in I$ such that $k \leq v_0$.
	Therefore, $k \leq u \land v_0 \leq \bigvee_{v \in I} (u \land v)$.
	This proves the infinite distributive law restricted to directed joins.
	Therefore, $\O$ is a locally compact frame.
	This proves the bijective correspondence \eqref{i:lcf}$\leftrightarrow$\eqref{i:lcf-pl}.

	\eqref{i:lcf-pl}$\leftrightarrow$\eqref{i:lcf-epl}.
	The claimed bijection is a restriction of the bijection between distributive bi-dcpos and distributive embedded bi-dcpos (\cref{t:bij-corr}).
	Let us prove that it is well-defined.
	The assignment \eqref{i:lcf-pl}$\rightarrow$\eqref{i:lcf-epl} is well-defined by \cref{t:distributivity-from-O-K}\eqref{i:O-is-lattice} (which guarantees the distributivity of the embedded bi-dcpo), \cref{l:closed-under} (which guarantees the closure of $\O$ under finite meets), \cref{l:the-simple-closure} (which guarantees the closure of $\O$ under finite joins) and \cref{t:meets-and-joins,l:closed-under-bicont} (which guarantee the closure of $\K$ under finite joins).
	It is straightforward that the assignment \eqref{i:lcf-epl}$\rightarrow$\eqref{i:lcf-pl} is well-defined.
	
	\eqref{i:lcf-epl}$\leftrightarrow$\eqref{i:loccomp-sob-dps}.
	The claimed bijection is a restriction of the bijection between distributive embedded bi-dcpos and ko-spaces (\cref{t:bij-corr}).
	The fact that this restriction is well-defined is immediate.

	\eqref{i:loccomp-sob-dps}$\leftrightarrow$\eqref{i:locally-compact-sober-space}.
	We start by recalling that a topological space is locally compact and sober if and only if it is locally compact, well-filtered and $T_0$.
	The assignment \eqref{i:locally-compact-sober-space}$\rightarrow$\eqref{i:loccomp-sob-dps} maps a locally compact well-filtered $T_0$ space $X$ to $(X, \KSat(X), \Op(X))$; the latter is a ko-space by \cref{l:well-filtered-spaces}, and all other required properties are clear.
	The assignment \eqref{i:loccomp-sob-dps}$\rightarrow$\eqref{i:locally-compact-sober-space} maps a ko-space $(X, \K, \O)$ with $\O$ closed under finite intersections and finite unions and $\K$ under finite unions to the set $X$ equipped with $\O$ as topology. By \cref{c:lc-d-prespace}, $\K$ is precisely the set of compact saturated subsets of the topological space $(X, \O)$. Therefore, $(X, \O)$ is well-filtered and $T_0$ by \cref{l:well-filtered-spaces}, $(X, \O)$ is locally compact because $(X, \K, \O)$ is locally compact, and the assignment \eqref{i:loccomp-sob-dps}$\rightarrow$\eqref{i:locally-compact-sober-space}$\rightarrow$\eqref{i:loccomp-sob-dps} is the identity. The fact that the assignment \eqref{i:locally-compact-sober-space}$\rightarrow$\eqref{i:loccomp-sob-dps}$\rightarrow$\eqref{i:locally-compact-sober-space} is the identity is immediate.
	This proves the one-to-one correspondence \eqref{i:loccomp-sob-dps}$\leftrightarrow$\eqref{i:locally-compact-sober-space}.
\end{proof}

For well-filtered $T_0$ spaces $X$ and $Y$, we use the expression ``c-relation from $X$ to $Y$'' for a c-relation from $(X, \KSat(X), \Op(X))$ to $(Y, \KSat(Y), \Op(Y))$ in the sense of \cref{d:c-relation}; i.e., for a weakening relation $R \colon X \rel Y$ (with respect to the specialization orders) such that the forward image of a compact saturated set is compact saturated and the preimage of a closed set is closed.
We now improve \cref{t:lcf-lcss} to include morphisms.
The equivalence established in our main result (\cref{t:MAIN}) restricts to the two full subcategories of locally compact frames and locally compact sober spaces:

\begin{theorem}
	The category of locally compact frames and preframe homomorphisms is equivalent (as well as dually equivalent) to the category of locally compact sober spaces and c-relations.
\end{theorem}

\begin{proof}
	For morphisms, it is enough to recall that, in this context, preframe homomorphisms correspond precisely to Galois morphisms (see \cref{l:cont-semilattices}).
\end{proof}

One can then further restrict to frame homomorphisms and continuous functions.

\subsection{Bicontinuous dirspaces}
\label{s:dirspaces}

We have dealt with ko-spaces, which are sets equipped with two families $\K$ and $\O$ of subsets. To what degree can we get rid of one of the two classes (say $\K$)?

Local compactness (more precisely, bicontinuity) guarantees that not only compact saturated sets are definable from open sets, but also open sets are definable from compact saturated sets.
In particular, we will now show that a bicontinuous ko-space $(X, \K, \O)$ is determined by $(X, \O)$, and characterize the structures of the form $(X, \O)$ arising in this way (\cref{t:lcf-lcss-weak}).
Because of the symmetry of bicontinuous ko-spaces, these structures will provide a solution to the problem mentioned in the introduction of obtaining the two Wilker's conditions at the price of one.

\begin{definition}[Dirspace]
	A \emph{dirspace} $(X, \O)$ is a set $X$ equipped with a set $\O$ of subsets closed under directed unions.
\end{definition}

The elements of $\O$ are said to be the \emph{open} subsets of $X$.

The prefix ``dir'' stands for ``directed'', which refers to the fact that the set of open sets is required to be closed under directed unions.

A special case of a dirspace is a \emph{prespace} \cite[Sec.~2]{Erne2007}, i.e., a set equipped with a set of subsets closed under directed unions and finite intersections.

\begin{definition}[Compact, saturated]
	Let $(X, \O)$ be a dirspace and $S$ a subset of $X$.
	\begin{enumerate}
		\item
		$S$ is \emph{compact} if for every directed subset $\I$ of $\O$ with $S \subseteq \bigcup \I$ there is $U \in \I$ with $S \subseteq U$.
		
		\item
		$S$ is \emph{saturated} if it is the codirected intersection of $\{U \in \O \mid S \subseteq U\}$.
	\end{enumerate}
	We denote by $\KSat(X, \O)$ the set of compact saturated subsets of $X$.
\end{definition}

\begin{remark}\label{r:Lawson-dual}
	A subset $S$ of a dirspace $(X, \O)$ is compact and saturated if and only if $\{U \in \O \mid S \subseteq U\}$ is a Scott-open filter of $\O$ and $S$ is its intersection.
\end{remark}

\begin{remark}[Saturated vs intersection of opens vs upset]
	The following implications hold for any subset of a dirspace:
	\begin{center}
		saturated $\Longrightarrow$ intersection of open sets $\Longrightarrow$ upset,
	\end{center}
	where ``upset'' refers to the specialization preorder, defined as follows: $x \leq y$ if and only if every open containing $x$ contains $y$.
	The converses of both implications fail in general.\footnote{A counterexample to the converse of the first implication is the empty dirspace with no open subsets: the empty set is the intersection of the empty family of open sets but is not saturated.
		A counterexample to the converse of the second implication is a singleton dirspace where the only open set is the whole set. (This is even a locally compact well-filtered $T_0$ prespace.) The empty set is an upset that is not an intersection of open sets.}

	For prespaces, we have
	\begin{center}
		saturated $\Longleftrightarrow$ intersection of open sets $\Longrightarrow$ upset.
	\end{center}	
	In fact ``intersection of open sets'' is the definition of ``saturated'' in \cite[Sec.~2]{Erne2007}.
	So, our definition is compatible with Ern\'e's one in the context of prespaces.
	The converse of the last implication fails in general.\footnote{We can use the same counterexample as before: a singleton where the only open set is the whole set.}
	
	Finally, for topological spaces, all conditions are equivalent:
	\begin{center}
		saturated $\Longleftrightarrow$ intersection of open sets $\Longleftrightarrow$ upset.
	\end{center}
	So our definition agrees also with the standard terminology for topological spaces.
	
	Our motivation for our definition of saturated set is that (besides agreeing with the already existing topological and pretopological terminology) this is the condition appearing in \cref{c:double-duals-pointfree,c:HM-III} above (while the other ones would not work, see \cref{r:necessary-dir-codir}).
\end{remark}

\begin{definition}[Well-filtered, locally compact, bicontinuous]
	A dirspace $(X, \O)$ is...
	\begin{enumerate}
	
		\item
		...\emph{well-filtered} if for every codirected set $\F$ of compact saturated sets and every $U \in \O$ with $\bigcap \F \subseteq U$ there is $K \in \F$ with $K \subseteq U$;
		\item
		...\emph{locally compact}\footnote{(Comparison with the definition of local compactness for prespaces in \cite{Erne2007}.)
		The definition of local compactness given in \cite{Erne2007} includes also the condition that every open set is a directed union of the compact saturated sets it contains; we include this important condition in the notion of ``bicontinuity''. Our definition and Erné's one are equivalent when restricted to the topological setting.} if, for every compact saturated set $K$ and open set $U$ with $K \subseteq U$, there are an open set $U'$ and a compact saturated set $K'$ with $K \subseteq U' \subseteq K' \subseteq U$.
		
		\item
		...\emph{bicontinuous} if it is well-filtered and locally compact, and every open set $U$ is a directed union of the compact saturated sets contained in $U$.
	
	\end{enumerate}

\end{definition}

\begin{definition}[$T_0$]
	A dirspace $(X, \O)$ is said to be $T_0$ if for all distinct $x,y \in X$ there is $U \in \O$ to which exactly one of $x$ and $y$ belongs.
\end{definition}

\begin{theorem} \label{t:lcf-lcss-weak}
	There is a one-to-one correspondence between
	\begin{enumerate}
		
		\item \label{i:loccomp-sob-dps-weak}
		bicontinuous ko-spaces;
		
		\item \label{i:locally-compact-sober-space-weak}
		$T_0$ bicontinuous dirspaces where
		\begin{enumerate}
			
			\item \label{i:el-open-codirected}
			for every element $x$ the set of open sets containing $x$ is codirected,
			
			\item \label{i:el-open-avoid}
			for every element $x$ there is a largest open set not containing $x$.
			
		\end{enumerate}
		
	\end{enumerate}
\end{theorem}

\begin{proof}
	The assignment \eqref{i:locally-compact-sober-space-weak}$\rightarrow$\eqref{i:loccomp-sob-dps-weak} maps $(X, \O)$ to $(X, \KSat(X, \O), \O)$, where $X$ is equipped with the specialization order.
	We show that this assignment is well-defined.
	Since $X$ is $T_0$, the specialization preorder is a partial order.
	All open sets and all compact saturated sets are intersections of open sets and hence are upsets.
	$\O$ is closed under directed unions by definition of dirspace.
	Since $X$ is well-filtered, $\KSat(X, \O)$ is closed under codirected intersections.
	The condition \eqref{i:compactness-ps} in \cref{d:ko-space} (compactness of k-sets) holds by definition of compact sets, and the condition \eqref{i:cocompactness-ps} in \cref{d:ko-space} holds since $X$ is well-filtered.
	For every $x \in X$, the set $\u x$ is clearly compact.
	We prove that it is saturated.
	The set $\{U \in \O \mid \u x \subseteq U\}$ equals $\{U \in \O \mid x \in U\}$ by definition of the specialization order.
	The set $\{U \in \O \mid x \in U\}$ is codirected by \eqref{i:el-open-codirected}; moreover, its intersection is $\u x$ by definition of the specialization order.
	Moreover, for every $x \in X$, by definition of the specialization order the set $X \setminus \d x$ is the union of all the open sets to which $x$ does not belong.
	Since by \eqref{i:el-open-avoid} there is a largest open set not containing $x$, it is $X \setminus \d x$.
	This proves that $(X, \KSat(X, \O), \O)$ is a ko-space.
	Let us prove that it is bicontinuous.
	The ko-space $(X, \KSat(X, \O), \O)$ is locally compact because $X$ is locally compact.
	For every compact saturated set, the set of open sets containing it is codirected by definition of ``saturated''.
	For every open set, the set of compact saturated sets contained in it is directed by definition of bicontinuous dirspace.
	This proves that the ko-space $(X, \KSat(X, \O), \O)$ is bicontinuous.
	
	The assignment \eqref{i:loccomp-sob-dps-weak}$\rightarrow$\eqref{i:locally-compact-sober-space-weak} maps a bicontinuous ko-space $(X, \K, \O)$ to $(X, \O)$. By \cref{c:HM-III}, $\K$ is precisely the set of compact saturated subsets of $(X, \O)$. Therefore, $(X, \O)$ is locally compact and well-filtered. Moreover, the order on the poset $X$ of the ko-space $(X, \K, \O)$ is the specialization order of $(X, \O)$ by the equivalence \eqref{eq:order} at p.~\pageref{eq:order}; therefore, $(X, \O)$ is $T_0$.
	Finally, for every $x \in X$ the set $X \setminus \d x$ is the largest element of $\O$ not containing $x$.
	Moreover, since $\u x \in \K$ and $(X, \K, \O)$ is bicontinuous, the set $\{U \in \O \mid \u x \subseteq U\}$, which equals $\{U \in \O \mid x \in U\}$, is codirected.
	Furthermore, since $(X, \K, \O)$ is bicontinuous, for every $U \in \O$ the set $\{K \in \K \mid K \subseteq U\}$ is directed, and, by \cref{c:HM-III}\eqref{i:O-is-cocompact-saturated}, $U=\bigcup\{K\in\mathcal K\mid K\subseteq U\}$.
	
	The fact that the assignment \eqref{i:locally-compact-sober-space-weak}$\rightarrow$\eqref{i:loccomp-sob-dps-weak}$\rightarrow$\eqref{i:locally-compact-sober-space-weak} is the identity is immediate.
	The assignment \eqref{i:loccomp-sob-dps-weak}$\rightarrow$\eqref{i:locally-compact-sober-space-weak}$\rightarrow$\eqref{i:loccomp-sob-dps-weak} is the identity by \cref{c:HM-III}.
\end{proof}

Because of the symmetry of bicontinuous ko-spaces, the class of dirspaces in \cref{t:lcf-lcss-weak}\eqref{i:locally-compact-sober-space-weak} enjoys a symmetry, which we still call de Groot duality.

\begin{definition}[De Groot dual]\hfill
	\begin{enumerate}
	
		\item
		The \emph{de Groot dual} $(X, \O)^\partial$ of a dirspace $(X, \O)$ is the pair
		\[
		(X, \{X \setminus K \mid K \text{ compact and saturated}\}).
		\]
		
		\item
		We say that a dirspace $(X, \O)$ \emph{has de Groot duality}\footnote{It remains unclear to us the exact relationship between this notion and a similar notion of ``having duality'' for topological spaces, which postulates that the posets of open sets and of compact saturated sets are Lawson dual, and which was investigated for example in \cite[IV-2.18]{GierzHofmannEtAl2003}, \cite{HofmannLawson1984} and \cite[Cor.~8.13]{EscardoLawsonEtAl2004}.} if its de Groot dual $(X, \O)^\partial$ is a dirspace (i.e., if compact saturated sets are closed under codirected intersections) and 
		\[
		(X, \O)^{\partial\partial} = (X, \O).
		\]
	\end{enumerate}
\end{definition}

\begin{remark}[Characterization of ``having de Groot duality'']
	A dirspace $(X, \O)$ has de Groot duality if and only if
	\begin{enumerate}
		
		\item
		it is well-filtered,
		
		\item
		every $U \in \O$ is the directed union of the compact saturated sets contained in it,
		
		\item
		if a subset $S$ of $X$ is the directed union of the compact saturated sets contained in it, and for every codirected subset $\F$ of $\KSat(X,\mathcal O)$ with $\bigcap \F \subseteq S$ there is $K \in \F$ with $K \subseteq S$, then $S$ is open.
		
	\end{enumerate}
\end{remark}

\begin{remark}[Hausdorff space $\not\Rightarrow$ with de Groot duality]
	While all well-filtered $T_0$ topological spaces can be treated via ko-spaces (see \cref{l:well-filtered-spaces}), not all of them are dirspaces with de Groot duality.
	In fact, even a Hausdorff space may fail to coincide with its double de Groot dual.\footnote{Examples are the Arens-Fort space \cite[Part II, 26]{SteenSeebach1978} and any Fortissimo space \cite[Part II, 25]{SteenSeebach1978}. These are non-discrete Hausdorff spaces with the property that the compact subsets are precisely the finite ones, which implies that their double de Groot dual is discrete.}
\end{remark}

\begin{remark}
	De Groot duality is a well-defined involution of the class of dirspaces in \cref{t:lcf-lcss-weak}\eqref{i:locally-compact-sober-space-weak}. This follows from the fact that de Groot duality for ko-spaces is a well-defined involution on the class of bicontinuous ko-spaces.
\end{remark}

De Groot duality is a well-defined involution also on an interesting larger class (which we do not identify with any class of ko-spaces):

\begin{theorem} \label{t:de-groot-is-involution}
	De Groot duality is a well-defined involution on the class of bicontinuous dirspaces.
\end{theorem}

\begin{proof}
	Let $(X, \O)$ be a bicontinuous dirspace.
	Set $\K \coloneqq \KSat(X, \O)$.
	Since $(X, \O)$ is well-filtered, $(X, \{X \setminus K \mid K \in \K\})$ is a dirspace.
	It is easy to see that the triple $( \K, \O, \subseteq)$ is a bicontinuous bi-dcpo.
	Therefore, by \cref{t:HofMis}\eqref{i:HM-2}, the assignment $U \mapsto \{ K \in  \K \mid K \subseteq U\}$ is a bijection from $\O$ to the set of Scott-open filters of $\K^\op$.
	Therefore, the assignment $U \mapsto \bigcap\{X\setminus K\mid K\in\K,\ K\subseteq U\} = X \setminus U$ is a bijection from $\O$ to the set of compact saturated subsets of $(X, \{X \setminus K \mid K \in  \K\})$.
	Therefore, $(X, \O)$ has de Groot duality.
	It is easily seen that also $(X, \{X \setminus K \mid K \in  \K\})$ is a locally compact well-filtered dirspace in which every open set is the directed union of the compact saturated sets contained in it.
\end{proof}

In the introduction we expressed the desire for a setting that is more general than locally compact Hausdorff spaces and that has a symmetry that allows one to prove only one of the two Wilker's conditions \eqref{i:W1} and \eqref{i:W2}. Both the class in \cref{t:lcf-lcss-weak} and the class in \cref{t:de-groot-is-involution} present one such setting.\footnote{One can get a solution already using prespaces. Indeed, de Groot duality is a well-defined involution on the dirspaces in \cref{t:lcf-lcss-weak}\eqref{i:locally-compact-sober-space-weak} that are prespaces. (For prespaces, \eqref{i:el-open-codirected} is superfluous.) Then \eqref{i:W1} (with ``compact saturated'' instead of ``compact'') holds in every such prespace in which moreover open sets are closed under binary unions, and \eqref{i:W2} holds in every such prespace in which moreover compact saturated sets are closed under binary intersections; the two statements are then interderivable and they both specialize to locally compact Hausdorff spaces.}
We first obtain the desired result in the pointfree setting, which we will use later and for which we need some notation.

\begin{notation} [$\ll$]\label{n:wb}
	Let $(L, \K, \O)$ be a locally compact embedded bi-dcpo.
	For $a,b \in L$, we write $a \ll b$ if there are $k \in \K$ and $u \in \O$ such that $a \leq k \leq u \leq b$.
\end{notation}

\begin{notation}[{${\doubledownarrow_\O}$, ${\doubleuparrow_\K}$}]
	Let $(\K, \O, \ko)$ be a locally compact bi-dcpo.
	\begin{enumerate}
		
		\item
		For $u \in \O$ such that $\dk u$ is directed, we set 
		\[
		\ddo u \coloneqq \{v \in \O \mid v \ll u\} = \{v \in \O \mid \exists k \in \K \text{ such that } v \ok k \ko u\},
		\]
		where the notation $\ll$ from \cref{n:wb} is used in virtue of the embedding of $\O$ in the concept lattice of $(\K, \O, \ko)$.
		By \cref{l:directed-with-join}\eqref{i:way-below-equiv}, $\ddo u$ is the set of elements way below $u$ in the dcpo $\O$.
		
		\item
		For $k \in \K$ such that $\uo k$ is codirected, we set 
		\[
		\uuk k \coloneqq \{l \in \K \mid k \ll l\} = \{l \in \K \mid \exists u \in \O \text{ such that } k \ko u \ok l\}.
		\]
	\end{enumerate}
	
\end{notation}

\begin{lemma}[Wilker's condition, pointfree] \label{l:Wilker}
	Let $(\K, \O, \ko)$ be a locally compact bi-dcpo.
	\begin{enumerate}
		
		\item \label{i:Wilker1}
		Suppose that $\O$ has binary joins, let $u,v \in \O$ and $k \in \K$ with $k \ko u \lor v$, and suppose that the set $\dk u$ is directed. Then there is $u' \in \ddo u$ such that $k \ko u' \lor v$. 
		
		\item \label{i:Wilker2}
		Suppose that $\K$ has binary meets, let $k,l \in \K$ and $u \in \O$ with $k \land l \ko u$, and suppose that the set $\uo k$ is codirected. Then there is $k' \in \uuk k$ such that $k' \land l \ko u$.
		
	\end{enumerate} 
	
\end{lemma}

\begin{proof}
	\eqref{i:Wilker1}.
	By \cref{l:directed-with-join}\eqref{i:directed-with-join}, the set $\{w \in \O \mid w \ll u\}$ is directed with join $u$.
	Therefore, in the associated embedded bi-dcpo, 
	\[
	k \leq u \lor v = \mleft(\bigvee_{u' \ll u} u'\mright) \lor v = \bigvee_{u' \ll u} (u' \lor v).
	\]
	Since this join is directed, by double compactness there is $u' \ll u$ with $k \leq u' \lor v$.
	
	\eqref{i:Wilker2}. This is the dual of \eqref{i:Wilker1}.
\end{proof}

\begin{theorem}[Wilker's conditions for bicontinuous dirspaces] \label{t:wilker-directed}
	Let $(X, \O)$ be a bicontinuous dirspace.
	\begin{enumerate}
		
		\item \label{i:dir-spaces-Wilker1}
		Suppose that open sets are closed under binary unions.
		Let $U_1, U_2$ be open sets and $K$ a compact saturated set with $K \subseteq U_1 \cup U_2$.
		There are compact saturated sets $L_1, L_2$ such that $L_1 \subseteq U_1$, $L_2 \subseteq U_2$ and $K \subseteq L_1 \cup L_2$.

		\item \label{i:dir-spaces-Wilker2}
		Suppose that compact saturated sets are closed under binary intersections.
		Let $K_1, K_2$ be compact saturated sets and $U$ an open set with $K_1 \cap K_2 \subseteq U$.
		There are open sets $V_1, V_2$ such that $K_1 \subseteq V_1$, $K_2 \subseteq V_2$ and $V_1 \cap V_2 \subseteq U$.
	\end{enumerate}
	
	\noindent\begin{minipage}{0.5\textwidth}
		\begin{center}
			
			\begin{tikzpicture}
				
				\draw[thick, black] (-0.8,0.1) circle (1.2);
				\node[black] at (-1.9,1.2) {$U_1$};
				
				\draw[thick, gray, dashed] (-1.3,-0.5) rectangle (0.09,0.7);
				\node[gray] at (-1.65,0.15) {$\exists L_1$};
				
				\draw[thick, black] (0.8,-0.1) circle (1.2);
				\node[black] at (1.9,-1.2) {$U_2$};
				
				\draw[thick, gray, dashed] (-0.09,-0.7) rectangle (1.3,0.5);
				\node[gray] at (1.65,-0.15) {$\exists L_2$};
				
				\draw[thick, NavyBlue] (-0.9,-0.4) rectangle (0.9,0.4);
				\node[NavyBlue] at (1.07,0.3) {$K$};

			\end{tikzpicture}
			
		\end{center}
	\end{minipage}%
	\begin{minipage}{0.5\textwidth}
		\begin{center}
			\begin{tikzpicture}
				
				\draw[thick, dashed, gray] (-0.8,0.1) circle (1.2);
				\node[gray] at (-1.9,1.2) {$\exists V_1$};
				
				\draw[thick, black] (-1.3,-0.5) rectangle (0.09,0.7);
				\node[black] at (-1.65,0.15) {$K_1$};

				\draw[thick, dashed, gray] (0.8,-0.1) circle (1.2);
				\node[gray] at (1.9,-1.2) {$\exists V_2$};
				
				\draw[thick, black] (-0.09,-0.7) rectangle (1.3,0.5);
				\node[black] at (1.65,-0.15) {$K_2$};
				
				\draw[thick, NavyBlue] (0,0) circle (1);
				\node[NavyBlue] at (0.2,1.3) {$U$};

			\end{tikzpicture}
		\end{center}
	\end{minipage}
	
\end{theorem}

\begin{proof}
	\eqref{i:dir-spaces-Wilker1} holds by \cref{l:Wilker}. \eqref{i:dir-spaces-Wilker2} holds by \eqref{i:dir-spaces-Wilker1} and \cref{t:de-groot-is-involution}.
\end{proof}

Note that in \cref{t:wilker-directed} we deduced the second statement from the first one in virtue of de Groot duality.
Specializing \cref{t:wilker-directed} to the topological $T_0$ setting gives the following, which when further specialized to the Hausdorff setting gives \eqref{i:W1} and \eqref{i:W2} as stated in the introduction.
We recall that a topological space is locally compact and sober if and only if it is locally compact, well-filtered and $T_0$; see \cite[Prop.~8.3.5]{GoubaultLarrecq2013} for the left-to-right implication and \cite[Prop.~8.3.8]{GoubaultLarrecq2013} for the right-to-left implication.

\begin{corollary}[Wilker's conditions for topological spaces] \hfill
	\begin{enumerate}
		
		\item
		Let $X$ be a locally compact sober space (as, for example, any locally compact Hausdorff space), $U_1, U_2$ open sets and $K$ a compact saturated set with $K \subseteq U_1 \cup U_2$.
		There are compact saturated sets $L_1, L_2$ such that $L_1 \subseteq U_1$, $L_2 \subseteq U_2$ and $K \subseteq L_1 \cup L_2$.

		\item
		Let $X$ be a stably locally compact space\footnote{A \emph{stably locally compact space} is a locally compact sober space in which the intersection of two compact saturated sets is compact (and necessarily saturated); see e.g.\ \cite[Def.~8.3.31]{GoubaultLarrecq2013}.} (as, for example, any locally compact Hausdorff space), $K_1, K_2$ compact saturated sets and $U$ an open set with $K_1 \cap K_2 \subseteq U$.
		There are open sets $V_1, V_2$ such that $K_1 \subseteq V_1$, $K_2 \subseteq V_2$ and $V_1 \cap V_2 \subseteq U$.
	\end{enumerate}
\end{corollary}

\appendix

\section{Meet- and join- preservation properties of basic functions}
In this section, we collect some elementary meet- and join- preservation properties of the basic functions of a polarity.

\subsection{Meets of k-elements, joins of o-elements}

\begin{lemma} \label{l:preservation-of-meets-and-joins}
	Let $(\K, \O, \ko)$ be a purified polarity.
	\begin{enumerate}
	
		\item \label{i:basic-K-pres-meets}
		The basic function $\iota_\K \colon \K \hookrightarrow \mathfrak{B}(\K, \O, \ko)$ preserves all existing meets.
		
		\item \label{i:basic-O-pres-joins}
		The basic function $\iota_\O \colon \O \hookrightarrow \mathfrak{B}(\K, \O, \ko)$ preserves all existing joins.
	
	\end{enumerate}
\end{lemma}

\begin{proof}
	\eqref{i:basic-K-pres-meets}.
	Let $F$ be a subset of $\K$ with a meet in $\K$, and let us prove $\bigwedge \iota_\K[F] = \iota_\K(\bigwedge F)$.
	The inequality $\bigwedge \iota_\K[F]  \geq \iota_\K(\bigwedge F)$ follows from the fact that $\iota_\K$ is order-preserving.
	Let us prove the converse inequality.
	The element $\bigwedge \iota_\K[F]$ is a join of elements in $\im(\iota_\K)$; let then $I$ be a subset of $\K$ such that $\bigwedge \iota_\K[F] = \bigvee \iota_\K[I]$.
	For every $x \in I$, $x$ is below every element of $F$ (since $\iota_\K$ is order-reflecting) and hence below $\bigwedge F$; therefore, since $\iota_\K$ is order-preserving, $\iota_\K(x) \leq \iota_\K(\bigwedge F)$.
	It follows that $\bigvee \iota_\K[I] \leq \iota_\K(\bigwedge F)$.
	Therefore, $\bigwedge \iota_\K[F] = \bigvee \iota_\K[I] \leq \iota_\K(\bigwedge F)$. 
	This proves that $\iota_\K$ preserves all existing meets.
	
	\eqref{i:basic-O-pres-joins}. This is dual to \eqref{i:basic-K-pres-meets}.
\end{proof}

In terms of double base lattices, \cref{l:preservation-of-meets-and-joins} entails the following.

\begin{lemma} \label{l:the-simple-closure}
	
	Let $(L, \K, \O)$ be a double base lattice.
	\begin{enumerate}
		
		\item \label{i:K-closed-under-meets}
		$\K$ is closed under finite meets in $L$ if and only if the poset $\K$ has finite meets.
	
		\item \label{i:O-closed-under-joins}
		$\O$ is closed under finite joins in $L$ if and only if the poset $\O$ has finite joins.
		
	\end{enumerate}
\end{lemma}

\begin{proof}
	
	\eqref{i:K-closed-under-meets}. The left-to-right implication is obvious, and the right-to-left one follows from the fact that the inclusion $\K \hookrightarrow L$ preserves all existing meets (\cref{l:preservation-of-meets-and-joins}).

	\eqref{i:O-closed-under-joins}.  This is dual to \eqref{i:K-closed-under-meets}.
\end{proof}

\subsection{Joins of k-elements, meets of o-elements}

\begin{lemma} \label{l:preserve}
	Let $(\K, \O, \ko)$ be a purified polarity.
	\begin{enumerate}
	
		\item \label{i:preserve-finite-joins}
		If for every $u \in \O$ the set $\dk u$ is directed, then the basic function $\iota_\K \colon \K \hookrightarrow \B(\K, \O, \ko)$ preserves all existing finite joins.
	
		\item \label{i:preserve-finite-meets}
		If for every $k \in \K$ the set $\uo k$ is codirected, then the basic function $\iota_\O \colon \O \hookrightarrow \B(\K, \O, \ko)$ preserves all existing finite meets.

	\end{enumerate}
	
\end{lemma}

\begin{proof}
	\eqref{i:preserve-finite-joins}.
	Let $A$ be a finite subset of $\K$ admitting a join.
	The inequality $\bigvee \iota_\K[A] \leq \iota_\K(\bigvee A)$ is clear; it remains to prove $\iota_\K(\bigvee A) \leq \bigvee \iota_\K[A]$.
	The element $\bigvee \iota_\K[A]$ is the meet of all the upper bounds of $\iota_\K[A]$ belonging to $\im(\iota_\O)$, so it is enough to prove that $\iota_\K(\bigvee A)$ is below any such element.
	Let $u \in \O$ be such that for every $l \in A$ we have $l \ko u$. 
	Since $\dk u$ is directed and $A$ is finite, there is an upper bound $j \in \K$ of $A$ such that $j \ko u$.
	Then, $ \bigvee A \leq j$.
	Therefore, $\bigvee A \ko u$, and hence 
	$\iota_\K(\bigvee A) \leq \iota_\O(u)$.
	
	\eqref{i:preserve-finite-meets}.
	This is dual to \eqref{i:preserve-finite-joins}.
\end{proof}

In terms of double base lattices, \cref{l:preserve} amounts to the following.

\begin{lemma}\label{l:closed-under}
	Let $(L, \K, \O)$ be a double base lattice.
	\begin{enumerate}
	
		\item \label{i:K-closed-under-joins}
		$\K$ is closed under finite joins in $L$ if and only if the poset $\K$ has finite joins and, for every $u \in \O$, the set $\dk u$ is directed.

		\item \label{i:O-closed-under-meets}
		$\O$ is closed under finite meets in $L$ if and only if the poset $\O$ has finite meets and, for every $k \in \K$, the set $\uo k$ is codirected.
			
	\end{enumerate}
\end{lemma}
\begin{proof}
	 \eqref{i:K-closed-under-joins}. The left-to-right implication is obvious, and the right-to-left one follows from \cref{l:preserve}.
	
	\eqref{i:O-closed-under-meets}. This is dual to \eqref{i:K-closed-under-joins}.
\end{proof}

\begin{remark} \label{r:preserve-binary}
	A slight adjustment of the proof of \cref{l:preserve} shows that (i) if for all $k \in \K$ the set $\uo k$ has a lower bound for each pair of its elements, then the basic function $\iota_\O \colon \O \to\B(\K, \O, \ko)$ preserves all existing binary meets, and, dually, (ii) if for all $u \in \O$ the set $\dk u$ has an upper bound for each pair of its elements, then the basic function $\iota_\K \colon \K \to \B(\K, \O, \ko)$ preserves all existing binary joins.
\end{remark}

\makeatletter
\begingroup
\let\addcontentsline\@gobblethree
\section*{Acknowledgments}
The first author would like to thank Drew Moshier for a helpful discussion in the early stages of this project.
We are grateful also to Guram Bezhanishvili, Mai Gehrke and Alexander Kurz for their help on parts of this work. Finally, we thank the referee for their suggestions.

\textit{Funding.}
This work was funded by UK Research and Innovation (UKRI) under the UK government’s Horizon Europe funding guarantee (grant number EP/Y015029/1, Project ``DCPOS''). The ``Horizon Europe guarantee'' scheme provides funding to researchers and innovators who were unable to receive their Horizon Europe funding (in this case, for a Marie Skłodowska-Curie Actions (MSCA) grant) while the UK was in the process of associating.
\endgroup
\makeatother

\bibliography{Biblio}

@Article{Lorenzen2017,
  author  = {Lorenzen, P.},
  journal = {arXiv:1710.08138},
  title   = {Algebraic and logistic investigations on free lattices},
  year    = {2017},
  note    = {English translation by Stefan Neuwirth of “Algebraische und logistische Untersuchungen über freie Verbände”, J. Symb. Log., 16(2), 81–106, 1951},
}

@article {EscardoLawsonEtAl2004,
    AUTHOR = {Escard\'o, M. and Lawson, J. and Simpson, A.},
     TITLE = {Comparing {C}artesian closed categories of (core) compactly
              generated spaces},
   JOURNAL = {Topology Appl.},
  FJOURNAL = {Topology and its Applications},
    VOLUME = {143},
      YEAR = {2004},
    NUMBER = {1-3},
     PAGES = {105--145},
      ISSN = {0166-8641,1879-3207},
   MRCLASS = {54D50 (18B30 54B30 54C35 54D55)},
  MRNUMBER = {2080286},
       DOI = {10.1016/j.topol.2004.02.011},
       URL = {https://doi.org/10.1016/j.topol.2004.02.011},
}

@inproceedings{GierzKeimel1981,
  title={Continuous ideal completions and compactifications},
  author={Gierz, G. and Keimel, K.},
  booktitle={Continuous Lattices: Proceedings of the Conference on Topological and Categorical Aspects of Continuous Lattices (Workshop IV) Held at the University of Bremen, Germany, November 9--11, 1979},
  pages={97--124},
  year={1981},
  organization={Springer}
}

@article {Smyth1992,
    AUTHOR = {Smyth, M. B.},
     TITLE = {Stable compactification. {I}},
   JOURNAL = {J. London Math. Soc. (2)},
  FJOURNAL = {Journal of the London Mathematical Society. Second Series},
    VOLUME = {45},
      YEAR = {1992},
    NUMBER = {2},
     PAGES = {321--340},
      ISSN = {0024-6107,1469-7750},
   MRCLASS = {54D35 (54E05)},
  MRNUMBER = {1171559},
MRREVIEWER = {Y.\ T.\ Rhineghost},
       DOI = {10.1112/jlms/s2-45.2.321},
       URL = {https://doi.org/10.1112/jlms/s2-45.2.321},
}

@incollection {Smyth1986,
    AUTHOR = {Smyth, M. B.},
     TITLE = {Finite approximation of spaces (extended abstract)},
 BOOKTITLE = {Category theory and computer programming ({G}uildford, 1985)},
    SERIES = {Lecture Notes in Comput. Sci.},
    VOLUME = {240},
     PAGES = {225--241},
 PUBLISHER = {Springer, Berlin},
      YEAR = {1986},
      ISBN = {3-540-17162-2},
   MRCLASS = {68Q55},
  MRNUMBER = {875691},
       DOI = {10.1007/3-540-17162-2\_125},
       URL = {https://doi.org/10.1007/3-540-17162-2_125},
}

@book {SteenSeebach1978,
    AUTHOR = {Steen, L. A. and Seebach, Jr., J. A.},
     TITLE = {Counterexamples in topology},
   EDITION = {Second},
 PUBLISHER = {Springer-Verlag, New York-Heidelberg},
      YEAR = {1978},
     PAGES = {xi+244},
      ISBN = {0-387-90312-7},
   MRCLASS = {54-01},
  MRNUMBER = {507446},
}

@Article{JaklSuarez2025,
  author   = {Jakl, T. and Suarez, A. L.},
  journal  = {Appl. Categ. Structures},
  title    = {Canonical extensions via fitted sublocales},
  year     = {2025},
  issn     = {0927-2852,1572-9095},
  number   = {2},
  pages    = {10},
  volume   = {33},
  doi      = {10.1007/s10485-025-09802-6},
  fjournal = {Applied Categorical Structures. A Journal Devoted to Applications of Categorical Methods in Algebra, Analysis, Computer Science, Logic, Order and Topology},
  mrclass  = {06D22 (54F05)},
  mrnumber = {4871873},
  url      = {https://doi.org/10.1007/s10485-025-09802-6},
}

@Article{Lawson1979,
 Author = {Lawson, J. D.},
 Title = {The duality of continuous posets},
 FJournal = {Houston Journal of Mathematics},
 Journal = {Houston J. Math.},
 ISSN = {0362-1588},
 Volume = {5},
 Pages = {357--386},
 Year = {1979},
 Language = {English},
 zbMATH = {3666860},
 Zbl = {0428.06003}
}

@article {Jakl2020,
    AUTHOR = {Jakl, T.},
     TITLE = {Canonical extensions of locally compact frames},
   JOURNAL = {Topology Appl.},
  FJOURNAL = {Topology and its Applications},
    VOLUME = {273},
      YEAR = {2020},
     PAGES = {106976},
      ISSN = {0166-8641,1879-3207},
   MRCLASS = {06D22 (06B23 06B35 54H99)},
  MRNUMBER = {4074760},
MRREVIEWER = {Tomasz\ Kubiak},
       DOI = {10.1016/j.topol.2019.106976},
       URL = {https://doi.org/10.1016/j.topol.2019.106976},
}

@Article{Esakia1974,
 Author = {Esakia, L. L.},
 Title = {Topological {Kripke} models},
 FJournal = {Soviet Mathematics. Doklady},
 Journal = {Sov. Math., Dokl.},
 ISSN = {0197-6788},
 Volume = {15},
 Pages = {147--151},
 Year = {1974},
 Language = {English},
 zbMATH = {3463628},
 Zbl = {0296.02030}
}

@Book{Esakia2019,
  Author    = {Esakia, L.},
  Title     = {Heyting algebras. {Duality} theory. },
  FSeries   = {Trends in Logic -- Studia Logica Library},
  Series    = {Trends Log. Stud. Log. Libr.},
  ISSN      = {1572-6126},
  Volume    = {50},
  ISBN      = {978-3-030-12095-5; 978-3-030-12098-6; 978-3-030-12096-2},
  Year      = {2019},
  Publisher = {Springer, Cham},
  Language  = {English},
  DOI       = {10.1007/978-3-030-12096-2},
  zbMATH    = {7024700},
  Zbl       = {1436.06001},
  Note      = {Edited by G.\ Bezhanishvili and W.\ H.\ Holliday. {Translated} from the {Russian} by {A}. {Evseev}}
}

@article{Wilker1970,
	author = {Wilker, P.},
	doi = {10.2140/pjm.1970.34.269},
	fjournal = {Pacific Journal of Mathematics},
	issn = {1945-5844},
	journal = {Pac. J. Math.},
	language = {English},
	pages = {269--283},
	title = {Adjoint product and hom functors in general topology},
	volume = {34},
	year = {1970},
	zbl = {0205.52703},
	zbmath = {3327075}}

@Article{KeimelLawson2005,
 Author = {Keimel, K. and Lawson, J. D.},
 Title = {Measure extension theorems for {{\(T_{0}\)}}-spaces},
 FJournal = {Topology and its Applications},
 Journal = {Topology Appl.},
 ISSN = {0166-8641},
 Volume = {149},
 Number = {1-3},
 Pages = {57--83},
 Year = {2005},
 Language = {English},
 DOI = {10.1016/j.topol.2004.02.019},
 zbMATH = {2172998},
 Zbl = {1152.06301}
}

@article{Wille1985,
  title={Tensorial decomposition of concept lattices},
  author={Wille, R.},
  journal={Order},
  volume={2},
  pages={81--95},
  year={1985},
  publisher={Springer}
}

@Article{Lorenzen1951,
  author   = {Lorenzen, P.},
  journal  = {J. Symb. Log.},
  title    = {Algebraische und logistische {Untersuchungen} {\"u}ber freie {Verb{\"a}nde}},
  year     = {1951},
  issn     = {0022-4812},
  note     = {Translation by S.\ Neuwirth: Algebraic and logistic investigations on free lattices, \url{http://arxiv.org/abs/1710.08138}.},
  number   = {2},
  pages    = {81--106},
  volume   = {16},
  doi      = {10.2307/2266681},
  fjournal = {The Journal of Symbolic Logic},
  language = {German},
  zbl      = {0045.29502},
  zbmath   = {3070433},
}

@Article{DeitersErne1998,
 Author = {Deiters, K. and Ern{\'e}, M.},
 Title = {Negations and contrapositions of complete lattices},
 FJournal = {Discrete Mathematics},
 Journal = {Discrete Math.},
 ISSN = {0012-365X},
 Volume = {181},
 Number = {1-3},
 Pages = {91--111},
 Year = {1998},
 Language = {English},
 DOI = {10.1016/S0012-365X(97)00047-2},
 zbMATH = {1135960},
 Zbl = {0898.06003}
}

@InCollection{CederquistCoquand2000,
 Author = {Cederquist, J. and Coquand, T.},
 Title = {Entailment relations and distributive lattices},
 BookTitle = {Logic colloquium '98. Proceedings of the annual European summer meeting of the Association for Symbolic Logic, Prague, Czech Republic, August 9--15, 1998},
 ISBN = {1-56881-113-6; 1-56881-114-4},
 Pages = {127--139},
 Year = {2000},
 Publisher = {Natick, MA: A K Peters, Ltd.},
 Language = {English},
 zbMATH = {1420837},
 Zbl = {0948.03056}
}

@article {DaveyPriestley1996,
    AUTHOR = {Davey, B. A. and Priestley, H. A.},
     TITLE = {Optimal natural dualities. {II}. {G}eneral theory},
   JOURNAL = {Trans. Amer. Math. Soc.},
  FJOURNAL = {Transactions of the American Mathematical Society},
    VOLUME = {348},
      YEAR = {1996},
    NUMBER = {9},
     PAGES = {3673--3711},
      ISSN = {0002-9947,1088-6850},
   MRCLASS = {08B99 (06D05 06D15 18A40)},
  MRNUMBER = {1348858},
MRREVIEWER = {M.\ E.\ Adams},
       DOI = {10.1090/S0002-9947-96-01601-7},
       URL = {https://doi.org/10.1090/S0002-9947-96-01601-7},
}

@article{Bruns1959,
	author = {Bruns, G.},
	doi = {10.1007/BF01240771},
	fjournal = {Archiv der Mathematik},
	issn = {0003-889X},
	journal = {Arch. Math.},
	language = {German},
	pages = {109--112},
	title = {Verbandstheoretische {Kennzeichnung} vollst{\"a}ndiger {Mengenringe}},
	volume = {10},
	year = {1959},
	zbl = {0086.25303},
	zbmath = {3141349}}

@article{Balachandran1955,
	author = {Balachandran, V. K.},
	doi = {10.4064/fm-41-1-38-41},
	fjournal = {Fundamenta Mathematicae},
	issn = {0016-2736},
	journal = {Fundam. Math.},
	language = {English},
	pages = {38--41},
	title = {A characterization of {{\(\Sigma\Delta\)}}-rings of subsets},
	volume = {41},
	year = {1955},
	zbl = {0055.28001},
	zbmath = {3088181}}

@Article{GalatosJipsen2020,
  author        = {Galatos, N. and Jipsen, P.},
  journal       = {Algebra Universalis},
  title         = {The structure of generalized {BI}-algebras and weakening relation algebras},
  year          = {2020},
  issn          = {0002-5240,1420-8911},
  number        = {3},
  pages         = {35},
  volume        = {81},
  doi           = {10.1007/s00012-020-00663-9},
  fjournal      = {Algebra Universalis},
  mrclass       = {06F05 (03B47 03G10 08B15)},
  mrnumber      = {4116171},
  mrreviewer    = {J\=anis\ C\=\i rulis},
  url           = {https://doi.org/10.1007/s00012-020-00663-9},
}

@article{BezhanishviliHarding2020,
	author = {Bezhanishvili, G. and Harding, J.},
	doi = {10.1007/s10485-020-09606-w},
	fjournal = {Applied Categorical Structures. A Journal Devoted to Applications of Categorical Methods in Algebra, Analysis, Order, Topology and Computer Science},
	issn = {0927-2852,1572-9095},
	journal = {Appl. Categ. Structures},
	mrclass = {54D10 (06D10 06D22 06E25)},
	mrnumber = {4163311},
	mrreviewer = {Xiaoquan\ Xu},
	number = {6},
	pages = {963--973},
	title = {Raney algebras and duality for {$T_0$}-spaces},
	url = {https://doi.org/10.1007/s10485-020-09606-w},
	volume = {28},
	year = {2020}}

@article{GehrkeHarding2001,
	author = {Gehrke, M. and Harding, J.},
	doi = {10.1006/jabr.2000.8622},
	fjournal = {Journal of Algebra},
	issn = {0021-8693},
	journal = {J. Algebra},
	mrclass = {06B23 (06A15 06D10 06D15)},
	mrnumber = {1822196},
	mrreviewer = {Friedrich Wehrung},
	number = {1},
	pages = {345--371},
	title = {Bounded lattice expansions},
	url = {https://doi.org/10.1006/jabr.2000.8622},
	volume = {238},
	year = {2001}}

@book{DaveyPriestley2002,
	author = {Davey, B. A. and Priestley, H. A.},
	edition = {Second},
	isbn = {0-521-78451-4},
	language = {English},
	publisher = {Cambridge University Press},
	title = {Introduction to lattices and order.},
	year = {2002},
	zbl = {1002.06001},
	zbmath = {1748069}}

@misc{Strickland2009,
	author = {Strickland, N. P.},
	note = {\url{https://ncatlab.org/nlab/files/StricklandCGHWSpaces.pdf}},
	title = {The category of {CGWH} spaces},
	year = {2009}}

@article{Erne1993,
	author = {Ern\'{e}, M.},
	doi = {10.1007/BF01195382},
	fjournal = {Algebra Universalis},
	issn = {0002-5240,1420-8911},
	journal = {Algebra Universalis},
	mrclass = {06A15 (06A23 06D05 06D10)},
	mrnumber = {1240572},
	number = {4},
	pages = {538--580},
	title = {Distributive laws for concept lattices},
	url = {https://doi.org/10.1007/BF01195382},
	volume = {30},
	year = {1993}}

@article{Raney1952,
	author = {Raney, G. N.},
	doi = {10.2307/2032165},
	fjournal = {Proceedings of the American Mathematical Society},
	issn = {0002-9939,1088-6826},
	journal = {Proc. Amer. Math. Soc.},
	mrclass = {09.1X},
	mrnumber = {52392},
	mrreviewer = {O.\ Bor\ocirc{u}vka},
	pages = {677--680},
	title = {Completely distributive complete lattices},
	url = {https://doi.org/10.2307/2032165},
	volume = {3},
	year = {1952}}

@article{Stone1936,
	author = {Stone, M. H.},
	journal = {Trans. Amer. Math. Soc.},
	number = {1},
	pages = {37--111},
	title = {The theory of representations for {B}oolean algebras},
	volume = {40},
	year = {1936}}

@article{Stone1938,
	author = {Stone, M. H.},
	journal = {{\v C}asopis pro p{\v e}stov{\'a}n{\'\i} matematiky a fysiky},
	pages = {1--25},
	title = {Topological representations of distributive lattices and {B}rouwerian logics},
	volume = {67},
	year = {1938}}

@Article{Kegelmann2002,
  author    = {Kegelmann, M.},
  journal   = {Electron. Notes Theor. Comput. Sci.},
  title     = {Continuous domains in logical form},
  year      = {2002},
  pages     = {1--166},
  volume    = {49},
  publisher = {Elsevier},
}

@article{JungKegelmannEtAl1999,
	author = {Jung, A. and Kegelmann, M. and Moshier, M. A.},
	fjournal = {Fundamenta Informaticae},
	issn = {0169-2968},
	journal = {Fund. Inform.},
	mrclass = {68Q55 (03F05)},
	mrnumber = {1710300},
	mrreviewer = {Martin Hofmann},
	number = {4},
	pages = {369--412},
	title = {Multilingual sequent calculus and coherent spaces},
	volume = {37},
	year = {1999}}

@incollection{KurzMoshierEtAl2023,
	address = {Cham},
	author = {Kurz, A. and Moshier, A. and Jung, A.},
	booktitle = {Samson Abramsky on Logic and Structure in Computer Science and Beyond},
	doi = {10.1007/978-3-031-24117-8_5},
	editor = {Palmigiano, A. and Sadrzadeh, M.},
	isbn = {978-3-031-24117-8},
	pages = {159--215},
	publisher = {Springer International Publishing},
	title = {Stone duality for relations},
	url = {https://doi.org/10.1007/978-3-031-24117-8_5},
	year = {2023}}

@Article{JungKegelmannEtAl2001,
  author        = {Jung, A. and Kegelmann, M. and Moshier, M. A.},
  journal       = {Electron. Notes Theor. Comput. Sci.},
  title         = {Stably compact spaces and closed relations},
  year          = {2001},
  pages         = {209--231},
  volume        = {45},
  fjournal      = {Electronic Notes in Theoretical Computer Science},
  publisher     = {Elsevier},
}

@incollection{JungSuenderhauf1996,
	author = {Jung, A. and S\"{u}nderhauf, P.},
	booktitle = {Papers on general topology and applications ({G}orham, {ME}, 1995)},
	doi = {10.1111/j.1749-6632.1996.tb49171.x},
	mrclass = {54H99 (06E15 54D30)},
	mrnumber = {1429656},
	pages = {214--230},
	publisher = {New York Acad. Sci., New York},
	series = {Ann. New York Acad. Sci.},
	title = {On the duality of compact vs. open},
	url = {https://doi.org/10.1111/j.1749-6632.1996.tb49171.x},
	volume = {806},
	year = {1996}}

@article{JungMoshier2006,
	author = {Jung, A. and Moshier, M. A.},
	journal = {School of Computer Science Research Reports-University of Birmingham CSR},
	publisher = {Citeseer},
	title = {On the bitopological nature of {S}tone duality},
	volume = {13},
	year = {2006}}

@article{JonssonTarski1951,
	author = {J\'{o}nsson, B. and Tarski, A.},
	doi = {10.2307/2372123},
	fjournal = {American Journal of Mathematics},
	issn = {0002-9327},
	journal = {Amer. J. Math.},
	mrclass = {09.0X},
	mrnumber = {44502},
	mrreviewer = {R. C. Lyndon},
	pages = {891--939},
	title = {Boolean algebras with operators. {I}},
	url = {https://doi-org.pros.lib.unimi.it/10.2307/2372123},
	volume = {73},
	year = {1951}}

@Article{JonssonTarski1952,
 Author = {J{\'o}nsson, B. and Tarski, A.},
 Title = {Boolean algebras with operators. {II}},
 FJournal = {American Journal of Mathematics},
 Journal = {Amer. J. Math.},
 ISSN = {0002-9327},
 Volume = {74},
 Pages = {127--162},
 Year = {1952},
 Language = {English},
 DOI = {10.2307/2372074},
 zbMATH = {3070501},
 Zbl = {0045.31601}
}

@inproceedings{JohnstoneVickers1991,
	author = {Johnstone, P. and Vickers, S.},
	booktitle = {Category theory},
	organization = {Springer},
	pages = {193--212},
	title = {Preframe presentations present},
	year = {1991}}

@Book{Johnstone1982,
 Author = {Johnstone, P. T.},
 Title = {Stone spaces},
 FSeries = {Cambridge Studies in Advanced Mathematics},
 Series = {Camb. Stud. Adv. Math.},
 Volume = {3},
 Year = {1982},
 Publisher = {Cambridge University Press, Cambridge},
 Language = {English},
 zbMATH = {3787631},
 Zbl = {0499.54001}
}

@incollection{Johnstone1985,
	author = {Johnstone, P. T.},
	booktitle = {Continuous lattices and their applications},
	pages = {155--180},
	publisher = {CRC Press},
	title = {Vietoris locales and localic semilattices},
	year = {1985}}

@inproceedings{HofmannMislove1981,
	author = {Hofmann, K. H. and Mislove, M. W.},
  booktitle={Continuous Lattices: Proceedings of the Conference on Topological and Categorical Aspects of Continuous Lattices (Workshop IV) Held at the University of Bremen, Germany, November 9--11, 1979},
	pages = {209--248},
	publisher = {Springer},
	title = {Local compactness and continuous lattices},
	year = {1981}}

@article{HofmannLawson1984,
	author = {Hofmann, K. H. and Lawson, J. D.},
	journal={J. Aust. Math. Soc.},
	fjournal = {Journal of the Australian Mathematical Society},
	number = {2},
	pages = {194--212},
	publisher = {Cambridge University Press},
	title = {On the order-theoretical foundation of a theory of quasicompactly generated spaces without separation axiom},
	volume = {36},
	year = {1984}}

@book{GierzHofmannEtAl2003,
	author = {Gierz, G. and Hofmann, K. H. and Keimel, K. and Lawson, J. D. and Mislove, M. and Scott, D. S.},
	doi = {10.1017/CBO9780511542725},
	isbn = {0-521-80338-1},
	mrclass = {06-00 (06B35 54H12 68Q55)},
	mrnumber = {1975381},
	mrreviewer = {James W. Lea, Jr.},
	pages = {xxxvi+591},
	publisher = {Cambridge University Press, Cambridge},
	series = {Encyclopedia of Mathematics and its Applications},
	title = {Continuous lattices and domains},
	url = {https://doi-org.pros.lib.unimi.it/10.1017/CBO9780511542725},
	volume = {93},
	year = {2003}}

@article{Erne2007,
	author = {Ern{\'e}, M.},
	journal={Appl. Categ. Structures},
	fjournal = {Applied Categorical Structures},
	number = {5},
	pages = {541--572},
	publisher = {Springer},
	title = {Choiceless, pointless, but not useless: dualities for preframes},
	volume = {15},
	year = {2007}}

@article{Erne2016,
	author = {Ern{\'e}, M.},
	doi = {10.1007/s10485-016-9444-0},
	fjournal = {Applied Categorical Structures},
	issn = {0927-2852},
	journal = {Appl. Categ. Structures},
	language = {English},
	number = {5},
	pages = {471--496},
	title = {Choice-free dualities for domains},
	volume = {24},
	year = {2016},
	zbl = {1386.06007},
	zbmath = {6641501}}

@incollection{Erne2004,
	author = {Ern\'{e}, M.},
	booktitle = {Galois connections and applications},
	doi = {10.1007/978-1-4020-1898-5\_1},
	mrclass = {06A15 (01A05 06-03 12F10)},
	mrnumber = {2073648},
	mrreviewer = {Ji\v{r}\u{\i} Rach\ocirc{u}nek},
	pages = {1--138},
	publisher = {Kluwer Acad. Publ., Dordrecht},
	series = {Math. Appl.},
	title = {Adjunctions and {G}alois connections: origins, history and development},
	url = {https://doi.org/10.1007/978-1-4020-1898-5_1},
	volume = {565},
	year = {2004}}

@Book{ClarkDavey1998,
  author     = {Clark, D. M. and Davey, B. A.},
  publisher  = {Cambridge University Press, Cambridge},
  title      = {Natural dualities for the working algebraist},
  year       = {1998},
  isbn       = {0-521-45415-8},
  series     = {Cambridge Studies in Advanced Mathematics},
  volume     = {57},
  mrclass    = {18-02 (03G25 06D05 08C15 18A40)},
  mrnumber   = {1663208},
  mrreviewer = {Peter Johnstone},
  pages      = {xii+356},
}

@Book{Birkhoff1967,
  author     = {Birkhoff, G.},
  publisher  = {American Mathematical Society, Providence, R.I.},
  title      = {Lattice theory},
  year       = {1967},
  edition    = {Third},
  series     = {American Mathematical Society Colloquium Publications},
  volume     = {25},
  mrclass    = {06.30},
  mrnumber   = {0227053},
  mrreviewer = {P. A. Fillmore},
  pages      = {vi+418},
}

@article{BezhanishviliGabelaiaEtAl2019,
	author = {G. {Bezhanishvili} and D. {Gabelaia} and J. {Harding} and M. {Jibladze}},
	doi = {10.1007/s10485-019-09573-x},
	fjournal = {{Applied Categorical Structures}},
	issn = {0927-2852},
	journal = {{Appl. Categ. Structures}},
	language = {English},
	msc2010 = {54D30 54G05 54E05 06D22 54B30},
	number = {6},
	pages = {663--686},
	publisher = {Springer Netherlands, Dordrecht},
	title = {{Compact Hausdorff spaces with relations and Gleason spaces}},
	volume = {27},
	year = {2019},
	zbl = {1437.54021}}

@article{Banaschewski1988,
	author = {Banaschewski, B.},
	journal = {Comment. Math. Univ. Carolin.},
	fjournal = {Commentationes Mathematicae Universitatis Carolinae},
	number = {4},
	pages = {647--656},
	publisher = {Charles University in Prague, Faculty of Mathematics and Physics},
	title = {Another look at the localic {T}ychonoff theorem},
	volume = {29},
	year = {1988}}

@book{Vickers1989,
	author = {Vickers, S.},
	fseries = {Cambridge Tracts in Theoretical Computer Science},
	isbn = {0-521-36062-5},
	issn = {0965-9103},
	language = {English},
	publisher = {Cambridge University Press},
	series = {Camb. Tracts Theor. Comput. Sci.},
	title = {Topology via logic},
	volume = {5},
	year = {1989},
	zbl = {0668.54001},
	zbmath = {41225}}

@article{DeGroot1967,
	author = {De Groot, J.},
	journal={Bull. Amer. Math. Soc.},
	fjournal = {Bulletin of the American Mathematical Society},
	number = {3},
	pages = {465--467},
	publisher = {American Mathematical Society},
	title = {An isomorphism principle in general topology},
	volume = {73},
	year = {1967}}

@book{May1999,
	author = {May, J. P.},
	isbn = {0-226-51182-0; 0-226-51183-9},
	language = {English},
	publisher = {Chicago, IL: University of Chicago Press},
	title = {A concise course in algebraic topology},
	year = {1999},
	zbl = {0923.55001},
	zbmath = {1324388}}

@article{Kelly1963,
	author = {Kelly, J. C.},
	doi = {10.1112/plms/s3-13.1.71},
	fjournal = {Proceedings of the London Mathematical Society. Third Series},
	issn = {0024-6115},
	journal = {Proc. Lond. Math. Soc. (3)},
	language = {English},
	pages = {71--89},
	title = {Bitopological spaces},
	volume = {13},
	year = {1963},
	zbl = {0107.16401},
	zbmath = {3174958}}

@article{BanaschewskiBruemmerEtAl1983,
	author = {Banaschewski, B. and Br{\"u}mmer, G. C. L. and Hardie, K. A.},
	doi = {10.1080/16073606.1983.9632289},
	fjournal = {Quaestiones Mathematicae},
	issn = {1607-3606},
	journal = {Quaest. Math.},
	language = {English},
	pages = {13--25},
	title = {Biframes and bispaces},
	volume = {6},
	year = {1983},
	zbl = {0513.06005},
	zbmath = {3809619}}

@inproceedings{Pratt1995,
	author = {Pratt, V. R.},
	booktitle = {Proceedings of Tenth Annual IEEE Symposium on Logic in Computer Science},
	organization = {IEEE},
	pages = {444--454},
	title = {The {S}tone gamut: {A} coordinatization of mathematics},
	year = {1995}}

@book{GoubaultLarrecq2013,
	author = {Goubault-Larrecq, J.},
	isbn = {978-1-107-03413-6},
	mrclass = {54-01 (54D10 54H10 54H99)},
	mrnumber = {3086734},
	mrreviewer = {Alexander Yurievich Shibakov},
	pages = {vi+491},
	publisher = {Cambridge University Press, Cambridge},
	series = {New Mathematical Monographs},
	title = {Non-{H}ausdorff topology and domain theory},
	volume = {22},
	year = {2013}}

@book{GanterWille1999,
	author = {Ganter, B. and Wille, R.},
	doi = {10.1007/978-3-642-59830-2},
	isbn = {3-540-62771-5},
	mrclass = {06-02 (08A02 68P15)},
	mrnumber = {1707295},
	mrreviewer = {T.\ S.\ Blyth},
	note = {Mathematical foundations, Translated from the 1996 German original by Cornelia Franzke},
	pages = {x+284},
	publisher = {Springer-Verlag, Berlin},
	title = {Formal concept analysis},
	url = {https://doi.org/10.1007/978-3-642-59830-2},
	year = {1999}}

@article{Erne2004a,
	author = {Ern{\'e}, M.},
	doi = {10.1016/S0166-8641(03)00204-9},
	fjournal = {Topology and its Applications},
	issn = {0166-8641},
	journal = {Topology Appl.},
	language = {English},
	number = {1-3},
	pages = {125--158},
	title = {General {Stone} duality},
	volume = {137},
	year = {2004},
	zbl = {1043.06008},
	zbmath = {2056848}}

@book{Barr1979,
	author = {Barr, M.},
	fseries = {Lecture Notes in Mathematics},
	issn = {0075-8434},
	language = {English},
	publisher = {Springer-Verlag},
	series = {Lect. Notes Math.},
	address = {Berlin Heidelberg New York},
	title = {*-autonomous categories. {With} an appendix by {Po}-{Hsiang} {Chu}},
	volume = {752},
	year = {1979},
	zbl = {0415.18008},
	zbmath = {3645327}}

@InCollection{BilkovaKurzEtAl2011,
  author    = {B{\'{\i}}lkov{\'a}, M. and Kurz, A. and Petri{\c{s}}an, D. and Velebil, J.},
  booktitle = {Algebra and coalgebra in computer science. 4th international conference, CALCO 2011, Winchester, UK, August 30 -- September 2, 2011. Proceedings},
  publisher = {Berlin: Springer},
  title     = {Relation liftings on preorders and posets},
  year      = {2011},
  isbn      = {978-3-642-22943-5},
  pages     = {115--129},
  doi       = {10.1007/978-3-642-22944-2_9},
  language  = {English},
  url       = {digitalcommons.chapman.edu/cgi/viewcontent.cgi?article=1061&context=engineering_articles},
  zbl       = {1343.18003},
  zbmath    = {5945223},
}

@Article{KurzTzimoulis,
  author  = {Kurz, A. and Tzimoulis, A.},
  journal = {Advances in Modal Logic 2024},
  title   = {Canonical Extensions of Fuzzy Algebras},
}

@Article{StellSchmidtEtAl2016,
  author   = {Stell, J. G. and Schmidt, R. A. and Rydeheard, D.},
  journal  = {J. Log. Algebr. Methods Program.},
  title    = {A bi-intuitionistic modal logic: foundations and automation},
  year     = {2016},
  issn     = {2352-2208},
  number   = {4},
  pages    = {500--519},
  volume   = {85},
  doi      = {10.1016/j.jlamp.2015.11.003},
  fjournal = {Journal of Logical and Algebraic Methods in Programming},
  language = {English},
  zbl      = {1359.03018},
  zbmath   = {6596344},
}

@Article{GehrkeJonsson2004,
  author   = {Gehrke, M. and J{\'o}nsson, B.},
  journal  = {Math. Scand.},
  title    = {Bounded distributive lattice expansions},
  year     = {2004},
  issn     = {0025-5521},
  number   = {1},
  pages    = {13--45},
  volume   = {94},
  doi      = {10.7146/math.scand.a-14428},
  fjournal = {Mathematica Scandinavica},
  language = {English},
  zbl      = {1077.06008},
  zbmath   = {2115711},
}

@Article{GehrkeJonsson1994,
  author   = {Gehrke, M. and J{\'o}nsson, B.},
  journal  = {Math. Japon.},
  title    = {Bounded distributive lattices with operators},
  year     = {1994},
  issn     = {0025-5513},
  number   = {2},
  pages    = {207--215},
  volume   = {40},
  fjournal = {Mathematica Japonica},
  language = {English},
  zbl      = {0855.06009},
  zbmath   = {697070},
}

@Article{Priestley1972,
  author   = {Priestley, H. A.},
  journal  = {Proc. Lond. Math. Soc. (3)},
  title    = {Ordered topological spaces and the representation of distributive lattices},
  year     = {1972},
  issn     = {0024-6115},
  pages    = {507--530},
  volume   = {24},
  doi      = {10.1112/plms/s3-24.3.507},
  fjournal = {Proceedings of the London Mathematical Society. Third Series},
  language = {English},
  zbl      = {0323.06011},
  zbmath   = {3505038},
}

@Article{BanaschewskiPultr1990/91,
  author     = {Banaschewski, B. and Pultr, A.},
  journal    = {Order},
  title      = {Tarski's fixpoint lemma and combinatorial games},
  year       = {1990/91},
  issn       = {0167-8094,1572-9273},
  number     = {4},
  pages      = {375--386},
  volume     = {7},
  doi        = {10.1007/BF00383202},
  fjournal   = {Order. A Journal on the Theory of Ordered Sets and its Applications},
  mrclass    = {90D46 (06B99 90D05)},
  mrnumber   = {1120984},
  mrreviewer = {Imma\ J.\ Curiel},
  url        = {https://doi.org/10.1007/BF00383202},
}

@Article{Whitman1943,
  author     = {Whitman, P. M.},
  journal    = {Amer. J. Math.},
  title      = {Splittings of a lattice},
  year       = {1943},
  issn       = {0002-9327,1080-6377},
  pages      = {179--196},
  volume     = {65},
  doi        = {10.2307/2371781},
  fjournal   = {American Journal of Mathematics},
  mrclass    = {09.1X},
  mrnumber   = {7389},
  mrreviewer = {G.\ Birkhoff},
  url        = {https://doi.org/10.2307/2371781},
}

@Article{Goldblatt1976,
  author   = {Goldblatt, R. I.},
  journal  = {Rep. Math. Logic},
  title    = {Metamathematics of modal logic. {I}},
  year     = {1976},
  issn     = {0137-2904},
  pages    = {41--77},
  volume   = {6},
  fjournal = {Reports on Mathematical Logic},
  language = {English},
  zbl      = {0356.02016},
  zbmath   = {3554258},
}

@Article{SambinVaccaro1988,
  author   = {Sambin, G. and Vaccaro, V.},
  journal  = {Ann. Pure Appl. Logic},
  title    = {Topology and duality in modal logic},
  year     = {1988},
  issn     = {0168-0072},
  number   = {3},
  pages    = {249--296},
  volume   = {37},
  doi      = {10.1016/0168-0072(88)90021-8},
  fjournal = {Annals of Pure and Applied Logic},
  language = {English},
  zbl      = {0643.03014},
  zbmath   = {4047687},
}

@Book{Goldblatt1987,
  author    = {Goldblatt, R.},
  publisher = {CSLI Publications, Stanford, CA},
  title     = {Logics of time and computation},
  year      = {1987},
  isbn      = {0-937073-11-3; 0-937073-12-1},
  series    = {CSLI Lect. Notes},
  volume    = {7},
  fseries   = {CSLI Lecture Notes},
  language  = {English},
  zbl       = {0635.03024},
  zbmath    = {4033710},
}
\bibliographystyle{alpha}
\end{document}